\documentclass{amsart}
\usepackage{graphicx}

\newtheorem{thm}{Theorem}[section]
\newtheorem{lem}[thm]{Lemma}
\newtheorem{prop}[thm]{Proposition}

\theoremstyle{definition}
\newtheorem{defn}[thm]{Definition}

\theoremstyle{remark}
\newtheorem{rem}[thm]{Remark}

\numberwithin{equation}{section}

\def\N{\mathbb{N}}

\def\R{\mathbb{R}}

\def\ra{\rightarrow}
\def\bs{\backslash}

\def\al{\alpha}

\def\ep{\epsilon}

\def\la{\lambda}

\def\si{\sigma}

\def\om{\omega}
\def\Om{\Omega}
\def\de{\delta}

\def\ga{\gamma}
\def\Ga{\Gamma}

\begin{document}

\title[Stability, Uniqueness and Recurrence]{Stability, Uniqueness and Recurrence of Generalized Traveling Waves in Time Heterogeneous Media of Ignition Type}

\author{Wenxian Shen}
\address{Department of Mathematics and Statistics, Auburn University, Auburn, AL 36849}
\email{wenxish@auburn.edu}

\author{Zhongwei Shen}
\address{Department of Mathematics and Statistics, Auburn University, Auburn, AL 36849}
\email{zzs0004@auburn.edu}





\begin{abstract}
The present paper is devoted to the study of stability, uniqueness and recurrence of generalized traveling waves of reaction-diffusion equations in time heterogeneous media of ignition type, whose existence has been proven by the authors of the present paper in a previous work. It is first shown that generalized traveling waves exponentially attract wave-like initial data. Next, properties of generalized traveling waves, such as space monotonicity and exponential decay ahead of interface, are obtained. Uniqueness up to space translations of generalized traveling waves is then proven. Finally, it is shown that the wave profile
and the front propagation velocity of the unique generalized traveling wave are of the same recurrence as the media. In particular, if the media is time almost periodic, then so are the wave profile
 and the  front propagation velocity of the unique generalized traveling wave.
\end{abstract}

\maketitle

\tableofcontents


\section{Introduction}\label{sec-intro}

Consider the one-dimensional reaction-diffusion equation
\begin{equation}\label{general-eqn}
u_{t}=u_{xx}+f(t,x,u), \quad x\in\R,\,\,t\in\R,
\end{equation}
where $f(t,x,u)$ is of ignition type, that is, there exists $\theta\in(0,1)$ such that for all $t\in\R$ and $x\in\R$, $f(t,x,u)=0$ for $u\in[0,\theta]\cup\{1\}$ and $f(t,x,u)>0$ for $u\in(\theta,1)$. Such an equation arises in the combustion theory (see e.g. \cite{BeLaLi90,BeNiSc85}). The number $\theta$ is called the ignition temperature. The front propagation concerning this equation was first investigated by Kanel (see \cite{Ka60,Ka61,Ka62,Ka64}) in the space-time homogeneous media, i.e., $f(t,x,u)=f(u)$; he proved that all solutions, with initial data in some subclass of continuous functions with compact support and values in $[0,1]$, propagate at the same speed $c_{*}>0$, which is the speed of the unique traveling wave solution $\psi(x-c_{*}t)$, where $\psi$ satisfies
\begin{equation*}
\begin{split}
&\psi_{xx}+c_{*}\psi_{x}+f(\psi)=0,\quad\lim_{x\ra-\infty}\psi(x)=1\quad\text{and}\quad\lim_{x\ra\infty}\psi(x)=0.
\end{split}
\end{equation*}
Concerning the stability of $\psi(x-c_{*}t)$, Fife and McLeod proved in \cite{FiMc80} that $\psi(x-c_{*}t)$ attracts wave-like initial data. More precisely, if $u_{0}\in C^{1}(\R)$ is such that $u_{0}(-\infty)=1$, $u_{0}(\infty)=0$ and $(u_{0})_{x}<0$, then there exists $\ga\in C^{1}([0,\infty))$ satisfying $\lim_{t\ra\infty}\dot{\ga}(t)=0$ such that $\lim_{t\ra\infty}|u(t,x;u_{0})-\psi(x-c_{*}t-\ga(t))|=0$ uniformly in $x\in\R$. Also see \cite{ArWe75,ArWe78,FiMc77,FiMc80,Kam76,KPP37,Sa76,Uch78} and references therein for the treatment of traveling wave solutions of \eqref{general-eqn} in space-time homogeneous media and in other homogeneous media.

Recently, equation \eqref{general-eqn} in the space heterogeneous media, i.e., $f(t,x,u)=f(x,u)$, has attracted a lot of attention. In terms of space periodic media, that is, $f(x,u)$ is periodic in $x$, Berestycki and Hamel proved in \cite{BeHa02} the existence of pulsating fronts or periodic traveling waves of the form $\psi(x-c_{*}t,x)$, where $\psi(s,x)$ is periodic in $x$ and satisfies a degenerate elliptic equation with boundary conditions $\lim_{s\ra-\infty}\psi(s,x)=1$ and $\lim_{s\ra\infty}\psi(s,x)=0$ uniformly in $x$. In the work of Weinberger (see \cite{We02}), he proved from the dynamical system viewpoint that solutions with general non-negative compactly supported initial data spread with the speed $c_{*}$. We also refer to \cite{Xin91,Xin92,Xin93} for related works.

In the general space heterogeneous media, Nolen and Ryzhik (see \cite{NoRy09}), and Mellet, Roquejoffre and Sire (see \cite{MeRoSi10}) proved the existence of generalized traveling waves in the sense of Berestycki and Hamel (see \cite{BeHa07,BeHa12}). We recall that 
\begin{defn}\label{def-generalized-critical}
A global-in-time classical solution $u(t,x)$ of \eqref{general-eqn} is called a \textit{generalized traveling wave} (connecting $0$ and $1$) if $u(t,x)\in(0,1)$ for all $(t,x)\in\R\times\R$ and there is a function $\xi:\R\ra\R$, called \textit{interface location function}, such that
\begin{equation*}\label{eqn-time-profile-1}
\lim_{x\ra-\infty}u(t,x+\xi(t))=1\quad\text{and}\quad\lim_{x\ra\infty}u(t,x+\xi(t))=0\,\,\text{uniformly in}\,\,t\in\R.
\end{equation*}
\end{defn}

Later, stability and uniqueness of such generalized traveling waves in the space heterogeneous media were also established in \cite{MNRR09} by Mellet, Nolen, Roquejoffre and Ryzhik. In their work, stability means that generalized traveling waves exponentially attract wave-like initial data, and uniqueness is up to time translations. These results were then generalized by Zlato\v{s} (see \cite{Zl13}) to equations in cylindrical domains.

In a very recent work (see \cite{ShSh14}), the authors of the present paper investigated the equation \eqref{general-eqn} in the time heterogeneous media, that is,
\begin{equation}\label{main-eqn}
u_{t}=u_{xx}+f(t,u), \quad x\in\R,\,\,t\in\R
\end{equation}
and proved the existence of generalized traveling waves with additional properties. For convenience and later use, let us summarize  the main results obtained in \cite{ShSh14}. Consider the following two assumptions on the time heterogeneous nonlinearity $f(t,u)$.
\begin{itemize}
\item[\rm(H1)] {\it There is a $\theta\in(0,1)$, called the ignition temperature, such that for all $t\in\R$,
\begin{equation*}
\begin{split}
f(t,u)&=0,\quad u\in(-\infty,\theta]\cup\{1\},\\
f(t,u)&>0,\quad u\in(\theta,1),\\
f(t,u)&<0,\quad u>1.
\end{split}
\end{equation*}
The family of functions $\{f(\cdot,u), u\in\R\}$ is locally uniformly H\"{o}lder continuous. The family of functions $\{f(t,\cdot), t\in\R\}$ is locally uniformly Lipschitz continuous. For any $t\in\R$, $f(t,u)$ is continuously differentiable for $u\geq\theta$.}

\item[\rm(H2)] {\it There are Lipschitz continuous functions $f_{\inf}$, $f_{\sup}$ satisfying
\begin{equation*}
\begin{split}
&f_{\inf},f_{\sup}\in C^{1}([\theta,\infty),\R),\\
&f_{\inf}(u)=0=f_{\sup}(u)\,\,\text{for}\,\,u\in[0,\theta]\cup\{1\},\\
&0<(f_{\inf})_{u}^{+}(\theta)\leq(f_{\sup})_{u}^{+}(\theta),\\
&0>(f_{\inf})_{u}(1)\geq(f_{\sup})_{u}(1),\\
&0<f_{\inf}(u)<f_{\sup}(u)\,\,\text{for}\,\,u\in(\theta,1)
\end{split}
\end{equation*}
such that $f_{\inf}(u)\leq f(t,u)\leq f_{\sup}(u)$ for $u\in[\theta,1]$ and $t\in\R$.}
\end{itemize}

The main results in \cite{ShSh14} are summarized as follows.

\begin{prop}[\cite{ShSh14}]\label{prop-transition-wave}
Suppose $\rm(H1)$ and $\rm(H2)$. Equation \eqref{main-eqn} admits a generalized traveling wave $u^f(t,x)$ in the sense of Definition \ref{def-generalized-critical} with a continuously differentiable interface location function $\xi^f:\R\ra\R$ satisfying $u^f(t,\xi^f(t))=\theta$ for all $t\in\R$ and $\sup_{t\in\R}|\dot{\xi}^{f}(t)|<\infty$. Moreover, the following properties hold:
\begin{itemize}
\item[\rm(i)] (Space monotonicity) $u_{x}^f(t,x)<0$ for $x\in\R$ and $t\in\R$;
\item[\rm(ii)] (Uniform steepness) for any $M>0$, there is $C(M)>0$ such that
\begin{equation*}
u_{x}^f(t,x+\xi^f(t))<-C(M),\quad x\in[-M,M],\,\,t\in\R;
\end{equation*}
\item[\rm(iii)]  (Uniform decaying estimates) there exists a continuous and strictly decreasing function $v:\R\ra(0,1)$ satisfying $v(x)\geq 1-c_{1}e^{c_{2}x}$, $x\leq-c_{3}$ for some $c_{1},c_{2},c_{3}>0$ and $v(x)=\theta e^{-c_{0}x}$, $x\geq0$ for some $c_{0}>0$ such that
\begin{equation*}
\begin{split}
u^f(t,x+\xi^f(t))&\geq v(x),\quad x\leq0;\\
u^f(t,x+\xi^f(t))&\leq v(x),\quad x\geq0;
\end{split}
\end{equation*}
\item[\rm(iv)] (Uniform decaying estimates of derivative) there is $C>0$ such that
\begin{equation*}
u_{x}^f(t,x+\xi^f(t))\geq -Cv(x),\quad x\geq0.
\end{equation*}
\end{itemize}
\end{prop}

The generalized traveling wave constructed in \cite{ShSh14} has more properties than stated in Proposition \ref{prop-transition-wave}. Here, we only state the properties which will be used in the present paper. Property $\rm(iv)$ in Proposition \ref{prop-transition-wave} is not stated in \cite{ShSh14}, but it is a simple consequence of property $\rm(iii)$ and a prior estimates for parabolic equations. We see that the profile function $\psi=\psi^f(t,x)=u^f(t,x+\xi^f(t))$ is a solution of
\begin{equation}\label{eqn-time-profile}
\begin{cases}
\psi_{t}=\psi_{xx}+\dot{\xi}^f(t)\psi_{x}+f(t,\psi),\cr
\lim_{x\ra-\infty}\psi(t,x)=1,\,\,\lim_{x\ra\infty}\psi(t,x)=0\,\,\text{uniformly in}\,\,t\in\R.
\end{cases}
\end{equation}

The objective of the present paper is to investigate the stability, uniqueness, and recurrence of generalized traveling waves of \eqref{main-eqn} in the sense of Definition \ref{def-generalized-critical}.
Throughout the paper, by a generalized traveling wave of \eqref{main-eqn}, it is then always in the sense of  Definition \ref{def-generalized-critical}.

Besides $\rm(H1)$ and $\rm(H2)$, we assume
\begin{itemize}
\item[\rm(H3)] {\it There exist $\theta_{*}\in(\theta,1)$ and $\beta>0$ such that $f_{u}(t,u)\leq-\beta$ for $u\geq\theta_{*}$ and $t\in\R$.}
\end{itemize}

This assumption is not restrictive. In fact, if $f(t,u)=g(t)f(u)$ with $g(t)$ bounded and uniformly positive, then $\rm(H3)$ is the case provided $f(u)$ has negative continuous derivative near $1$.

Let
$$
C_{\rm unif}^b(\R,\R)=\{u\in C(\R,\R)\,|\, u(x)\,\,\, \text{is uniformly continuous and bounded on}\,\,\, \R\}
$$
with the uniform convergence topology.
Note that for any $u_0\in C_{\rm unif}^b(\R,\R)$ and $t_0\in\R$, \eqref{main-eqn} has a unique solution $u(t,\cdot;t_0,u_0)\in C_{\rm unif}^b(\R,\R)$ with $u(t_0,\cdot;t_0,u_0)=u_0$.

We first study the stability of the generalized traveling wave $u^f(t,x)$ in Proposition \ref{prop-transition-wave}. In what follows, $u^f(t,x)$ will always be this special generalized traveling wave. The main result is stated in

\begin{thm}\label{thm-stability-tw}
Suppose $\rm(H1)$-$\rm(H3)$. Suppose that $t_0\in\R$ and $u_{0}\in C_{\rm unif}^b(\R,\R)$ satisfy
\begin{equation*}
\begin{cases}
u_{0}:\R\ra[0,1],\quad u_{0}(-\infty)=1,\\
|u_{0}(x)-u^f(t_{0},x)|\leq Ce^{-\al_{0}(x-\xi^f(t_{0}))}\,\,\text{for}\,\,x\in\R\,\,\text{for some}\,\,C>0.
\end{cases}
\end{equation*}
Then, there exist $C=C(u_{0})>0$, $\zeta_{*}=\zeta_{*}(u_{0})\in\R$ and $r=r(\al_{0})>0$ such that
\begin{equation*}
\sup_{x\in\R}|u(t,x;t_{0},u_0)-u^f(t,x-\zeta_{*})|\leq Ce^{-r(t-t_{0})}
\end{equation*}
for all $t\geq t_{0}$.
\end{thm}

The proof of Theorem \ref{thm-stability-tw} is a version of the ``squeezing technique", which has been verified to be successful in many situations (see e.g. \cite{BaCh99,Ch09,Ch97,ChGu02,ChGuWu08,MNRR09,MaWu07,Sh99-1,SmZh00}). Our arguments are closer to the arguments in \cite{MNRR09}, where the space heterogeneous nonlinearity is treated. However, while the rightmost interface always moves rightward in the space heterogeneous case due to the time monotonicity, it is not the case here. In fact, $\xi^{f}(t)$ moves back and force in general due to the time-dependence of $f(t,u)$. This unpleasant fact is a source of many difficulties. It is overcome
in this paper
by introducing the modified interface location, which always moves rightward and stays within a neighborhood of the interface location (see Proposition \ref{prop-interface-rightward-propagation-improved}), and thus, shows the rightward propagation nature of the generalized traveling wave $u^{f}(t,x)$.

Next, we explore the monotonicity and exponential  decay ahead of interface for any generalized traveling wave of \eqref{main-eqn}, which play an important role in the study of uniqueness of generalized traveling waves and are also of independent interest. We prove

\begin{thm}\label{thm-monotoncity-exponential-decay}
Suppose $\rm(H1)$-$\rm(H3)$. Let  $v(t,x)$ be an arbitrary generalized traveling wave of \eqref{main-eqn} with interface location function $\xi^{v}(t)$. Then,
\begin{itemize}
\item[\rm(i)] there holds $v_{x}(t,x)<0$ for all $x\in\R$ and $t\in\R$;

\item[\rm(ii)] there are a constant $\hat c>0$ and a twice continuously differentiable function $\hat{\xi}^{v}:\R\ra\R$ satisfying
\begin{equation*}
0<\inf_{t\in\R}\dot{\hat{\xi}}^{v}(t)\leq\sup_{t\in\R}\dot{\hat{\xi}}^{v}(t)<\infty\quad\text{and}\quad\sup_{t\in\R}|\ddot{\hat{\xi}}^{v}(t) |<\infty
\end{equation*}
such that $\sup_{t\in\R}|\hat{\xi}^{v}(t)-\xi^{v}(t)|<\infty$ and
\begin{equation*}
v(t,x+\hat{\xi}^{v}(t))\leq\theta e^{-\hat c x},\quad x\geq 0
\end{equation*}
for all $t\in\R$.
\end{itemize}
\end{thm}

Note that Theorem \ref{thm-monotoncity-exponential-decay}$\rm(i)$  shows the
space monotonicity of generalized traveling waves of \eqref{main-eqn} and  Theorem \ref{thm-monotoncity-exponential-decay}$\rm(ii)$ reflects the exponential decay ahead of interface. We point out that space monotonicity of generalized traveling waves in general time heterogeneous media is only known in the bistable case (see \cite{Sh06}). In the monostable case, it is true in the unique ergodic media (see \cite{Sh11}).

We then study the uniqueness of generalized traveling waves of \eqref{main-eqn} and prove

\begin{thm}\label{thm-uniqueness-introduction}
Suppose $\rm(H1)$-$\rm(H3)$. Let $v(t,x)$ be an arbitrary generalized traveling wave of \eqref{main-eqn}. Then, there exists some $\zeta_{*}\in\R$ such that $v(t,x)=u^f(t,x+\zeta_{*})$ for all $x\in\R$ and $t\in\R$.
Hence generalized traveling waves of \eqref{main-eqn} are unique up to space translations.
\end{thm}

We finally investigate the recurrence of generalized traveling waves of \eqref{main-eqn}. To this end, we further assume

\begin{itemize}

\item[\rm(H4)] {\it The family $\{f(\cdot,u),f_u(\cdot,u)\,|\,u\in\R\}$ of functions is globally uniformly H\"{o}lder continuous.}

\end{itemize}

Let
$$
H(f)={\rm cl}\{f\cdot t|t\in\R\},
$$
where $f\cdot t(\cdot,\cdot)=f(\cdot+t,\cdot)$ and the closure is taken in the open compact topology. Assume $\rm(H1)$-$\rm(H4)$. Then for any $g\in H(f)$, $\rm(H1)$-$\rm(H3)$ are also satisfied with $f$ being replaced by $g$. By Proposition \ref{prop-transition-wave} and Theorem \ref{thm-uniqueness-introduction}, for any $g\in H(f)$, there is a unique generalized traveling wave $u^{g}(t,x)$ of
\begin{equation}\label{main-eqn-1}
u_{t}=u_{xx}+g(t,x)
\end{equation}
with the continuously differentiable interface location function $\xi^{g}(t)$ at $\theta$, i.e., $u^g(t,\xi^g(t))=\theta$ for all $t\in\R$, satisfying the normalization $\xi^g(0)=0$. Setting $\psi^{g}(t,x)=u^{g}(t,x+\xi^{g}(t))$, we prove

\begin{thm}\label{thm-recurrence}
Suppose $\rm(H1)$-$\rm(H4)$. Then
\begin{equation}
\label{recurrence-thm-eq1}
\psi^g(t,\cdot)=\psi^{g\cdot t}(0,\cdot),\quad \forall \,\, t\in\R,\,\,g\in H(f),
\end{equation}
\begin{equation}
\label{recurrence-thm-eq2}
\dot \xi^g(t)=-\frac{\psi_{xx}^g(t,0)+g(t,\psi^g(t,0))}{\psi_x^g(t,0)},\quad \forall\,\,t\in\R,\,\, g\in H(f),
\end{equation}
and
\begin{equation}
\label{recurrence-thm-eq3}
\text{the mapping}\,\,[H(f)\ni g\mapsto \psi^g(0,\cdot)\in C_{\rm unif}^b(\R,\R)]\,\,\text{is continuous}.
\end{equation}
 In particular, if
$f(t,u)$ is almost periodic in $t$ uniformly with respect to $u$ in bounded sets, then so are  $\psi^f(t,x)$ and $\dot\xi^f(t)$, and the average propagation speed
\begin{equation*}
\lim_{t\to\infty}\frac{\xi^f(t)-\xi^f(0)}{t}
\end{equation*}
exists. Moreover, $\mathcal{M}(\psi^f(\cdot,\cdot))$, $\mathcal{M}(\dot\xi^f(\cdot))\subset \mathcal{M}(f(\cdot,\cdot))$,
where $\mathcal{M}(\cdot)$ denotes the frequency module of an almost periodic function.
\end{thm}

We remark that Theorem \ref{thm-recurrence} implies that the wave profile $\psi^f(t,\cdot)$ is of the same recurrence as $f(t,\cdot)$ in the sense that if $f\cdot t_n (t,u)\to f(t,u)$ as $n\to\infty$ in $H(f)$, then $\psi^f(t+t_n,x)=\psi^{f\cdot (t_n+t)}(0,x)\to \psi^{f\cdot t}(0,x)=\psi^f(t,x)$ as $n\to\infty$ in open compact topology (this is due to \eqref{recurrence-thm-eq1}). It also implies that the front propagation velocity $\dot\xi^f(t)$ is of the same recurrence as $f(t,\cdot)$
in  the sense that if $f\cdot t_n (t,u)\to f(t,u)$ as $n\to\infty$ in
$H(f)$, then $\dot\xi^f(t+t_n)\to \dot\xi^f(t)$ as $n\to\infty$ in open compact topology
 (this is due to
  \eqref{recurrence-thm-eq1} and \eqref{recurrence-thm-eq2}). Of course, if $f(t,\cdot)$ is periodic in $t$, then $\psi^{f}(t,\cdot)$ is periodic in $t$ with the same period
  as that of $f(t,\cdot)$. This fact has been obtained in \cite[Theorem 1.3(2)]{ShSh14} by means of the uniqueness of critical traveling waves.

Generalized traveling waves in time heterogeneous bistable and monostable media have been studied in the literature. In time periodic bistable media, Alikakos, Bates and Chen (see \cite{AlBaCh99}) proved the existence, stability and uniqueness of time periodic traveling waves. In the time heterogeneous media, generalized traveling waves with a time-dependent profile satisfying \eqref{eqn-time-profile} and their uniqueness and stability have been investigated by Shen (see e.g. \cite{Sh99-1,Sh99-2,Sh04,Sh06}). There are also similar results for time heterogeneous KPP equations (see e.g. \cite{NaRo12,Sh11}).

Generalized traveling waves have been proven to exist in space heterogeneous Fisher-KPP type equations (see \cite{NRRZ12,Zl12}). Very recently, Ding, Hamel and Zhao proved in \cite{DiHaZh14} the existence of small and large period pulsating fronts in space periodic bistable media. But it is far from being clear in the general space heterogeneous media of bistable type due to the wave blocking phenomenon (see \cite{LeKe00}) except the one established in \cite{NoRy09} under additional assumptions. In \cite{Na14}, Nadin introduced the critical traveling wave, and proved that critical traveling waves exist even in the bistable space heterogeneous media.

The rest of the paper is organized as follows. In Section \ref{sec-modified-interface}, we introduce the modified interface location, which shows the rightward propagation nature of the generalized traveling wave $u^f(t,x)$
of \eqref{main-eqn} and is of great technical importance. In Section \ref{sec-a prior-estimate}, we give an a priori estimate, trapping the solution with wave-like initial data between two space shifts of $u^f(t,x)$ with exponentially small corrections. In Section \ref{sec-stability}, we study the stability of $u^f(t,x)$ and prove Theorem \ref{thm-stability-tw}. Section \ref{sec-property-gtw} is devoted to two general properties of generalized traveling waves defined in Definition \ref{def-generalized-critical}. In Subsection \ref{subsection-space-mono}, space monotonicity of generalized traveling waves is studied and Theorem \ref{thm-monotoncity-exponential-decay}$\rm(i)$ is proved. In Subsection \ref{subsec-exponential-decay}, exponential decay ahead of interface of generalized traveling waves is studied and Theorem \ref{thm-monotoncity-exponential-decay}$\rm(ii)$ is proved. In  Section \ref{sec-uniqueness}, we investigate the uniqueness of generalized traveling waves and prove Theorem \ref{thm-uniqueness-introduction}. In the last section, Section \ref{sec-recurrence}, we explore the recurrence of generalized traveling waves and prove Theorem \ref{thm-recurrence}.


\section{Modified Interface Location}\label{sec-modified-interface}

In this section, we study the rightward propagation nature of the generalized traveling wave $u^{f}(t,x)$ of \eqref{main-eqn}. Throughout this section, if no confusion occurs, we will write $u^f(t,x)$ and $\xi^f(t)$ as $u(t,x)$ and $\xi(t)$, respectively. We assume $\rm(H1)$-$\rm(H3)$ in this section.

Recall $\xi:\R\ra\R$ is such that $u(t,\xi(t))=\theta$ for all $t\in\R$. It is known that the interface location $\xi(t)$ moves back and forth in general due to the time-dependence of the nonlinearity $f(t,u)$. This unpleasant fact causes many technical difficulties. To circumvent it, we modify the interface location $\xi(t)$ properly.

 Let $f_{B}$ be a continuously differentiable function satisfying
\begin{equation}\label{bistable-nonlinearity}
\begin{cases}
f_{B}(0)=0,\,\,f_{B}(u)<0\,\,\text{for}\,\,u\in(0,\theta),\\
f_{B}(u)=f_{\inf}(u)\,\,\text{for}\,\,u\in[\theta,1]\,\,\text{and}\,\,\int_{0}^{1}f_{B}(u)du>0.
\end{cases}
\end{equation}
Since $f_{\inf}(u)>0$ for $u\in(\theta,1)$, such an $f_{B}$ exists. Clearly, $f_{B}$ is of standard bistable type and $f_{B}(u)\leq f(t,u)$ for all $u\in[0,1]$ and $t\in\R$. There exist (see e.g. \cite{ArWe75,ArWe78,FiMc77}) a unique $c_{B}>0$ and a profile $\phi_{B}$ satisfying $(\phi_{B})_{x}<0$, $\phi_{B}(-\infty)=1$ and $\phi_{B}(\infty)=0$ such that $\phi_{B}(x-c_{B}t)$ and its translations are traveling waves of
\begin{equation}\label{eqn-bi-1}
u_{t}=u_{xx}+f_{B}(u).
\end{equation}

Let $f_{I}=f_{\sup}$ be as in $\rm(H2)$. Fix some $\theta_{I}\in(0,\theta)$. Then, there exist (see e.g. \cite{ArWe75,ArWe78,FiMc77}) a unique $c_{I}>0$ and a profile $\phi_{I}$ satisfying $(\phi_{I})_{x}<0$, $\phi_{I}(-\infty)=1$ and $\phi_{I}(\infty)=\theta_{I}$ such that $\phi_{I}(x-c_{I}t)$ and its translations are traveling waves of
\begin{equation}\label{eqn-ig-homogeneous193213}
u_{t}=u_{xx}+f_{I}(u).
\end{equation}
Notice that $\phi_{I}(x-c_{I}t)$ connects $\theta_{I}$ and $1$ instead of $0$ and $1$. 

The following proposition gives the expected modification of $\xi(t)$.

\begin{prop}\label{prop-interface-rightward-propagation-improved}
There exist constants $C_{\max}>0$ and $d_{\max}>0$, and a continuously differentiable function $\tilde{\xi}:\R\ra\R$ satisfying
\begin{equation*}
\frac{c_{B}}{2}\leq\dot{\tilde{\xi}}(t)\leq C_{\max},\quad t\in\R
\end{equation*}
such that
\begin{equation*}
0\leq\tilde{\xi}(t)-\xi(t)\leq d_{\max},\quad t\in\R.
\end{equation*}
\end{prop}

The proof of Proposition \ref{prop-interface-rightward-propagation-improved} needs the rightward propagation estimate of $\xi(t)$, which we present now. For $\la\in(0,1)$, let $\xi_{\la}(t)$ be the interface location function of $u(t,x)$ at $\la$, that is, $u(t,\xi_{\la}(t))=\la$ for all $t\in\R$. It is well-defined by the space monotonicity of $u(t,x)$. By Proposition \ref{prop-transition-wave}, $\sup_{t\in\R}|\xi_{\la}(t)-\xi(t)|<\infty$ for all $\la\in(0,1)$.

\begin{lem}\label{lem-interface-rightward-propagation}
For any $\ep>0$, there is $t_{\ep}>0$ such that
\begin{equation*}
(c_{B}-\ep)(t-t_{0}-t_{\ep})\leq \xi(t)-\xi(t_{0})\leq (c_{I}+\ep)(t-t_{0}+t_{\ep}),\quad t\geq t_{0}.
\end{equation*}
In particular, there are $t_{B}>0$ and $t_{I}>0$ such that
\begin{equation*}
\frac{3c_{B}}{4}(t-t_{0}-t_{B})\leq \xi(t)-\xi(t_{0})\leq \frac{5c_{I}}{4}(t-t_{0}+t_{I}),\quad t\geq t_{0}.
\end{equation*}
\end{lem}
\begin{proof}
We prove the lemma within three steps.

\paragraph{\textbf{Step 1}} We first construct a function $\psi_{*}$ satisfying the following properties:
\begin{equation}\label{function-special}
\begin{cases}
\psi_{*}\in[0,1]\,\,\text{is nonincreasing},\\
\psi_{*}(0)=\theta,\,\,\lim_{x\ra-\infty}\psi_{*}(x)=1,\,\,\lim_{x\ra\infty}\psi_{*}(x)=0,\\
\psi_{*}(x)\leq u(t_{0},x+\xi(t_{0}))\,\,\text{for all}\,\,x\in\R\,\,\text{and}\,\,t_{0}\in\R.
\end{cases}
\end{equation}
For $x\leq0$, define $\psi_{*}(x)=v(x)$, where $v:(-\infty,0]\ra[\theta,1)$ is given by Proposition \ref{prop-transition-wave}$\rm(iii)$. For $x>0$, let $\psi_{*}(x)=0$. Clearly, such defined $\psi_{*}$ satisfies \eqref{function-special}.

Next, fix any $t_{0}\in\R$. Let $u_{B}(t,x;\psi_{*})$ be the solution of \eqref{eqn-bi-1} with initial data $u_{B}(0,x;\psi_{*})=\psi_{*}(x)\leq u(t_{0},x+\xi(t_{0}))$ by \eqref{function-special}. Thus, time homogeneity and comparison principle ensure
\begin{equation*}
u_{B}(t-t_{0},x;\psi_{*})\leq u(t,x+\xi(t_{0})),\quad x\in\R,\,\,t\geq t_{0}.
\end{equation*}
By the stability of traveling waves of \eqref{eqn-bi-1} (see \cite[Theorem 3.1]{FiMc77}) and the conditions satisfied by $\psi_{*}$, there exist $z_{0}\in\R$, $K>0$ and $\om>0$ such that
\begin{equation*}
\sup_{x\in\R}|u_{B}(t-t_{0},x;\psi_{*})-\phi_{B}(x-c_{B}(t-t_{0})-z_{0})|\leq Ke^{-\om(t-t_{0})},\quad t\geq t_{0}.
\end{equation*}
In particular, for $t\geq t_{0}$ and $x\in\R$
\begin{equation}\label{estimate-l-bd}
u(t,x+\xi(t_{0}))\geq u_{B}(t-t_{0},x;\psi_{*})\geq\phi_{B}(x-c_{B}(t-t_{0})-z_{0})-Ke^{-\om(t-t_{0})}.
\end{equation}
Let $T_{0}>0$ be such that $Ke^{-\om T_{0}}=\frac{1-\theta}{2}$ (we may make $K$ larger so that $K>\frac{1-\theta}{2}$ if necessary) and denote by $\xi_{B}(\frac{1+\theta}{2})$ the unique point such that $\phi_{B}(\xi_{B}(\frac{1+\theta}{2}))=\frac{1+\theta}{2}$. Setting $x=c_{B}(t-t_{0})+z_{0}+\xi_{B}(\frac{1+\theta}{2})$ in \eqref{estimate-l-bd}, we find for any $t\geq t_{0}+T_{0}$
\begin{equation*}
u(t,c_{B}(t-t_{0})+z_{0}+\xi_{B}(\frac{1+\theta}{2})+\xi(t_{0}))\geq\phi_{B}(\xi_{B}(\frac{1+\theta}{2}))-Ke^{-\om T_{0}}=\theta.
\end{equation*}
Monotonicity then yields
\begin{equation}\label{propagation-result-1}
\xi(t)-\xi(t_{0})\geq c_{B}(t-t_{0})+z_{0}+\xi_{B}(\frac{1+\theta}{2}), \quad t\geq t_{0}+T_{0}.
\end{equation}

\paragraph{\textbf{Step 2}} Now, fix $\la\in(\theta,1)$. Note that choosing $\la$ closer to $1$ and $\theta_{I}$ closer to $0$, we may assume  $\la>2\theta_{I}$. Let $\psi^{*}:\R\to[\theta_{I},1]$ be a uniformly continuous and nonincreasing function satisfying $\psi^{*}(x)=1$ for $x\leq 0$ and $\psi^{*}(x)=\theta_{I}$ for $x\geq x_{0}$, where $x_{0}>0$ is fixed. Clearly, $u(t_{0},\cdot+\xi_{\theta_{I}}(t_{0}))\leq\psi^{*}$. Applying comparison principle and the stability of ignition traveling waves (see e.g. \cite{Ro94}; also see Theorem \ref{thm-aprior-estimate}), we find for $t\geq t_{0}$ and $x\in\R$
\begin{equation}\label{an-estimate-13189474924}
u(t,x+\xi_{\theta_{I}}(t_{0}))\leq u_{I}(t-t_{0},x;\psi^{*})\leq\phi_{I}(x-c_{I}(t-t_{0})-\xi_{I})+\ep_{I}e^{-\om_{I}(t-t_{0})},
\end{equation}
where $u_{I}(t,x;\psi^{*})$ is the unique solution of \eqref{eqn-ig-homogeneous193213} with $u_{I}(0,\cdot;\psi^{*})=\psi^{*}$. Let $\xi_{I}(\frac{\la}{2})$ be the unique point such that $\phi_{I}(\xi_{I}(\frac{\la}{2}))=\frac{\la}{2}$ (since $\la>2\theta_{I}$, $\xi_{I}(\frac{\la}{2})$ is well-defined) and $T>0$ be such that $\ep_{I}e^{-\om_{I}T}=\frac{\la}{2}$ (we may make $\ep_{I}$ larger so that $\ep_{I}>\frac{\la}{2}$ if necessary). Setting $x=c_{I}(t-t_{0})+\xi_{I}+\xi_{I}(\frac{\la}{2})$ in \eqref{an-estimate-13189474924}, we conclude
\begin{equation*}
u(t,c_{I}(t-t_{0})+\xi_{I}+\xi_{I}(\frac{\la}{2})+\xi_{\theta_{I}}(t_{0}))\leq\la,\quad t\geq t_{0}+T,
\end{equation*}
which leads to\begin{equation*}
\xi_{\la}(t)\leq c_{I}(t-t_{0})+\xi_{I}+\xi_{I}(\frac{\la}{2})+\xi_{\theta_{I}}(t_{0}),\quad t\geq t_{0}+T.
\end{equation*}
Setting $C:=\sup_{t_{0}\in\R}|\xi_{\la}(t_{0})-\xi_{\theta_{I}}(t_{0})|<\infty$ due to Proposition \ref{prop-transition-wave}, we conclude
\begin{equation}\label{propagation-result-1-1234}
\xi_{\la}(t)-\xi_{\la}(t_{0})\leq c_{I}(t-t_{0})+\xi_{I}+\xi_{I}(\frac{\la}{2})+C,\quad t\geq t_{0}+T.
\end{equation}

\paragraph{\textbf{Step 3}} By \eqref{propagation-result-1}, \eqref{propagation-result-1-1234} and the fact $\sup_{t\in\R}|\xi_{\la}(t)-\xi(t)|<\infty$, the lemma holds for $t\geq t_{0}+\max\{T_{0},T\}$. But for $t\in[t_{0},t_{0}+\max\{T_{0},T\}]$, the lemma is trivial, since we always have
\begin{equation*}
u_{B}(t-t_{0},x-\xi(t_{0});\psi_{*})\leq u(t,x)\leq u_{I}(t-t_{0},x-\xi_{\theta_{I}}(t_{0});\psi^{*})
\end{equation*}
and $\sup_{t_{0}\in\R}|\xi(t_{0})-\xi_{\theta_{I}}(t_{0})|<\infty$. This completes the proof.
\end{proof}

The next result is an improvement of Lemma \ref{lem-interface-rightward-propagation}.

\begin{lem}\label{lem-interface-rightward-propagation-improved}
There are $C_{\max}>0$ and $d_{\max}>0$ such that for any $t_{0}\in\R$, there exists a continuously differentiable function $\xi_{t_{0}}:[t_{0},\infty)\ra\R$ satisfying
\begin{equation*}
\frac{c_{B}}{2}\leq\dot{\xi}_{t_{0}}(t)\leq C_{\max},\quad t\geq t_{0}
\end{equation*}
such that
\begin{equation*}
0\leq\xi_{t_{0}}(t)-\xi(t)\leq d_{\max},\quad t\geq t_{0}.
\end{equation*}
Moreover, $\{\dot{\xi}_{t_{0}}\}_{t_{0}\leq0}$ is uniformly bounded and uniformly Lipschitz continuous.
\end{lem}
\begin{proof}
We use the following estiamte to modify $\xi(t)$
\begin{equation}\label{slowest-faster-speed}
\frac{3c_{B}}{4}(t-t_{0}-t_{B})\leq \xi(t)-\xi(t_{0})\leq \frac{5c_{I}}{4}(t-t_{0}+t_{I}),\quad t\geq t_{0},
\end{equation}
which is proven in Lemma \ref{lem-interface-rightward-propagation}.

Fix any $t_{0}\in\R$. Define
\begin{equation*}
\eta(t;t_{0})=\xi(t_{0})+C_{0}+\frac{c_{B}}{2}(t-t_{0}),\quad t\geq t_{0},
\end{equation*}
where $C_{0}>\frac{5}{4}c_{I}t_{I}$ is fixed. Clearly, $\xi(t_{0})<\eta(t_{0};t_{0})$. By \eqref{slowest-faster-speed} and continuity, $\xi(t)$ will hit $\eta(t;t_{0})$ sometime after $t_{0}$. Let $T_1(t_{0})$ be the first time that $\xi(t)$ hits $\eta(t;t_{0})$, that is, 
\begin{equation*}
T_1(t_{0})=\min\big\{t\geq t_{0}\big|\xi(t)=\eta(t;t_{0})\big\}. 
\end{equation*}
It follows that
\begin{equation*}
\xi(t)<\eta(t;t_{0})\,\,\text{for}\,\,t\in[t_{0},T_1(t_{0}))\quad\text{and}\quad\xi(T_1(t_{0}))=\eta(T_1(t_{0});t_{0}).
\end{equation*}
As a simple consequence of \eqref{slowest-faster-speed}, we obtain $T_1(t_{0})-t_{0}\in[T_{\min},T_{\max}]$, where $0<T_{\min}<T_{\max}<\infty$ depend only on $c_{B}$, $t_{B}$, $c_{I}$ and $t_{I}$. In fact, we can take
\begin{equation*}
T_{\min}=\frac{4C_{0}-5c_{I}t_{I}}{5c_{I}-c_{B}}\quad\text{and}\quad T_{max}=\frac{4C_{0}+3c_{B}t_{B}}{c_{B}}.
\end{equation*}

Now, at the moment $T_1(t_{0})$, we define
\begin{equation*}
\eta(t;T_1(t_{0}))=\xi(T_1(t_{0}))+C_{0}+\frac{c_{B}}{2}(t-T_1(t_{0})),\quad t\geq T_1(t_{0}).
\end{equation*}
Similarly, $\xi(T_1(t_{0}))<\eta(T_1(t_{0});T_1(t_{0}))$ and $\xi(t)$ will hit $\eta(t;T_1(t_{0}))$ sometime after $T_1(t_{0})$. Denote by $T_{2}(t_{0})$ the first time that $\xi(t)$ hits $\eta(t;T_1(t_{0}))$. Then,
\begin{equation*}
\xi(t)<\eta(t;T_1(t_{0}))\,\,\text{for}\,\,t\in[T_1(t_{0}),T_{2}(t_{0}))\quad\text{and}\quad\xi(T_{2}(t_{0}))=\eta(T_{2}(t_{0});T_1(t_{0})),
\end{equation*}
and $T_{2}(t_{0})-T_1(t_{0})\in[T_{\min},T_{\max}]$ by \eqref{slowest-faster-speed}.

Repeating the above arguments, we obtain the following: there is a sequence $\{T_{n-1}(t_{0})\}_{n\in\N}$ satisfying $T_0(t_{0})=t_0$,
\begin{equation*}
T_{n}(t_{0})-T_{n-1}(t_{0})\in[T_{\min},T_{\max}]\quad \text{for all}\,\,n\in\N,
\end{equation*}
and  for any $n\in\N$
\begin{equation*}
\begin{split}
&\xi(t)<\eta(t;T_{n-1}(t_{0}))\,\,\text{for}\,\,t\in[T_{n-1}(t_{0}),T_{n}(t_{0}))\quad\text{and}\\
&\xi(T_{n}(t_{0}))=\eta(T_{n}(t_{0});T_{n-1}(t_{0})),
\end{split}
\end{equation*}
where
\begin{equation*}
\eta(t;T_{n-1}(t_{0}))=\xi(T_{n-1}(t_{0}))+C_{0}+\frac{c_{B}}{2}(t-T_{n-1}(t_{0})).
\end{equation*}
Moreover, for any $n\in\N$ and $t\in[T_{n-1}(t_{0}),T_{n}(t_{0})]$
\begin{equation*}
\begin{split}
&\eta(t;T_{n-1}(t_{0}))-\xi(t)\\
&\quad\quad\leq\xi(T_{n-1}(t_{0}))+C_{0}+\frac{c_{B}}{2}(t-T_{n-1}(t_{0}))-\bigg(\xi(T_{n-1}(t_{0}))+\frac{3}{4}c_{B}(t-T_{n-1}(t_{0})-t_{B})\bigg)\\
&\quad\quad=C_{0}+\frac{3}{4}c_{B}t_{B}-\frac{1}{4}c_{B}(t-T_{n-1}(t_{0}))\leq C_{0}+\frac{3}{4}c_{B}t_{B}.
\end{split}
\end{equation*}

\begin{figure}
\begin{center}
\includegraphics{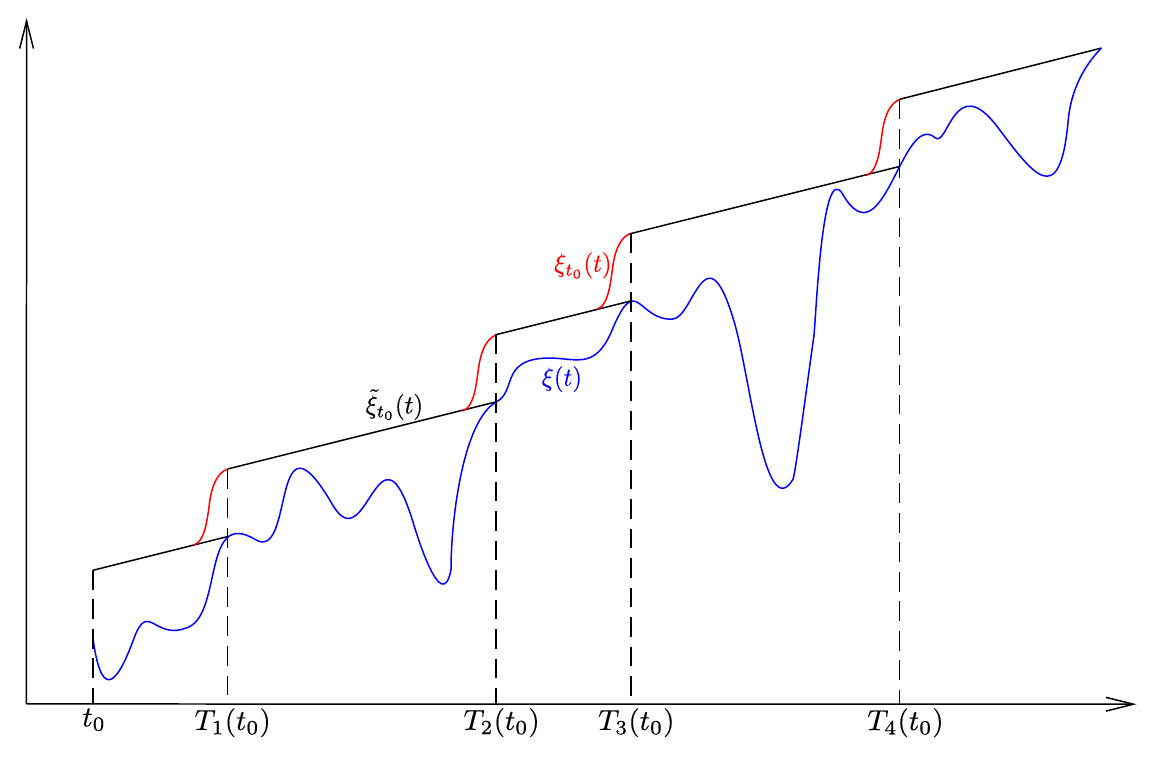}
\end{center}
\caption{\small Modified Interface Location}
\label{graph-modified-interface}
\end{figure}

Now, define $\tilde{\xi}_{t_{0}}:[t_{0},\infty)\ra\R$ by setting
\begin{equation}\label{definition-new-fun}
\tilde{\xi}_{t_{0}}(t)=\eta(t;T_{n-1}(t_{0})),\quad t\in[T_{n-1}(t_{0}),T_{n}(t_{0})),\,\,n\in\N.
\end{equation}
Since $[t_{0},\infty)=\cup_{n\in\N}[T_{n-1}(t_{0}),T_{n}(t_{0}))$, $\tilde{\xi}_{t_{0}}(t)$ is well-defined for all $t\geq t_{0}$ (see Figure \ref{graph-modified-interface} for the illustration). Notice $\tilde{\xi}_{t_{0}}(t)$ is strictly increasing and is linear on $[T_{n-1}(t_{0}),T_{n}(t_{0}))$ with slope $\frac{c_{B}}{2}$ for each $n\in\N$, and satisfies
\begin{equation*}
0\leq\tilde{\xi}_{t_{0}}(t)-\xi(t)\leq C_{0}+\frac{3}{4}c_{B}t_{B},\quad t\geq t_{0}.
\end{equation*}

Finally, we can modify $\tilde{\xi}_{t_{0}}(t)$ near each $T_{n}(t_{0})$ for $n\in\N$ to get $\xi_{t_{0}}(t)$ as in the statement of the lemma. In fact, fix some $\de_{*}\in(0,\frac{T_{\min}}{2})$. We modify $\tilde{\xi}_{t_{0}}(t)$ by redefining it on the intervals $(T_{n}(t_{0})-\de_{*},T_{n}(t_{0}))$, $n\in\N$ as follows: define
\begin{equation*}
\xi_{t_{0}}(t)=\begin{cases}
\tilde{\xi}_{t_{0}}(t),\quad t\in[t_{0},\infty)\bs\cup_{n\in\N}(T_{n}(t_{0})-\de_{*},T_{n}(t_{0})),\\
\xi(T_{n}(t_{0}))+\de(t-T_{n}(t_{0})),\quad t\in(T_{n}(t_{0})-\de_{*},T_{n}(t_{0})),\,\,n\in\N,
\end{cases}
\end{equation*}
where $\de:[-\de_{*},0]\ra[-\frac{c_{B}}{2}\de_{*},1]$ is continuously differentiable and satisfies
\begin{equation*}
\begin{split}
&\de(-\de_{*})=-\frac{c_{B}}{2}\de_{*},\quad\de(0)=1,\\
&\dot{\de}(-\de_{*})=\frac{c_{B}}{2}=\dot{\de}(0)\quad\text{and}\quad\dot{\de}(t)\geq\frac{c_{B}}{2}\,\,\text{for}\,\,t\in(-\de_{*},0).
\end{split}
\end{equation*}
Note the existence of such a function $\de(t)$ is clear. We point out that such a modification is independent of $t_{0}\in\R$ and $n\in\N$. Moreover, there exists some $C_{\max}=C_{\max}(\de_{*})>0$ such that $\dot{\de}(t)\leq C_{\max}$ for $t\in(-\de_{*},0)$. It's easy to see that $\xi_{t_{0}}(t)$ satisfies all required properties. This completes the proof.
\end{proof}

We remark that here we only need the function $\de(t)$ to be continuously differentiable. But $\de(t)$ can be obviously made to be at least twice continuously differentiable. Moreover, in proving Lemma \ref{lem-interface-rightward-propagation-improved}, we only used \eqref{slowest-faster-speed} and the continuity of $\xi(t)$. These observations will be useful later in Lemma \ref{lem-interface-propagation-improved-general}.

Proposition \ref{prop-interface-rightward-propagation-improved} now is a simple consequence of Lemma \ref{lem-interface-rightward-propagation-improved}.

\begin{proof}[Proof of Proposition \ref{prop-interface-rightward-propagation-improved}]
It follows from Lemma \ref{lem-interface-rightward-propagation-improved}, the fact that $\xi(t)$ remains bounded within any finite time interval, Arzel\`{a}-Ascoli theorem and the diagonal argument. In fact, we first see that the sequence of functions $\{\xi_{t_{0}}\}_{t_{0}\leq0}$ converges locally uniformly to some continuous function $\tilde{\xi}$ along some subsequence as $t_{0}\ra-\infty$. For the continuous differentiability, we note that $\{\dot{\xi}_{t_{0}}\}_{t_{0}\leq0}$ is uniformly bounded and uniformly Lipschitz continuous, and thus converges locally uniformly to some continuous function $\zeta$. It then follows that $\dot{\tilde{\xi}}=\zeta$, that is, $\xi$ is continuously differentiable. Other properties of $\tilde{\xi}$ stated in the proposition follow from the properties of the sequence $\{\xi_{t_{0}}\}_{t_{0}\leq0}$ as in Lemma \ref{lem-interface-rightward-propagation-improved}.
\end{proof}


\section{A priori Estimates}\label{sec-a prior-estimate}

In this section, we give an a priori estimate, trapping the solution with wave-like initial data between two space shifts of the generalized traveling wave $u^f(t,x)$ of \eqref{main-eqn} with exponentially small corrections. Throughout this section, if no confusion occurs, we will also write $u^f(t,x)$ and $\xi^f(t)$ as $u(t,x)$ and $\xi(t)$, respectively.

Let $\al_{0}>0$. Fix an initial data $u_0\in C_{\rm unif}^b(\R,\R)$ satisfying
\begin{equation}\label{initial-data}
\begin{cases}
u_{0}:\R\ra[0,1],\quad u_{0}(-\infty)=1,\cr
\exists\,\,t_{0}\in\R\,\,\text{s.t.}\,\,|u_{0}(x)-u(t_{0},x)|\leq Ce^{-\al_{0}(x-\xi(t_{0}))}\,\,\text{for}\,\,x\in\R\,\,\text{for some}\,\,C>0
\end{cases}
\end{equation}
as in the statement of Theorem \ref{thm-stability-tw}. We will show that the solution of \eqref{main-eqn} with initial data $u_{0}$ is trapped between two space shifts of $u(t,x)$ with exponentially small corrections. Before stating the main result, let us fix some parameters.

Let $L_{0}>0$ be such that for any $t\in\R$
\begin{equation}\label{condition-L}
\begin{split}
u(t,x)\geq\frac{1+\theta_{*}}{2}\quad&\text{if}\quad x\leq\xi(t)-\frac{L_{0}}{2}\quad\text{and}\\
\quad u(t,x)\leq\frac{\theta}{2}\quad&\text{if}\quad x\geq\xi(t)+\frac{L_{0}}{2},
\end{split}
\end{equation}
where $\theta_{*}$ is as in $\rm(H3)$. Such an $L_{0}$ exists by Proposition \ref{prop-transition-wave}$\rm(iii)$. Let $\Ga:=\Ga_{\al}:\R\ra[0,1]$ be a smooth function satisfying
\begin{equation}\label{function-Ga}
\begin{split}
&\sup_{x\in\R}\Ga'(x)\le 0,\quad C_{\Ga}:=\sup_{x\in\R}|\Ga''(x)|<\infty\quad\text{and}\\
&\Ga(x)=\begin{cases}
1,\,\,&x\leq -L_{0}-1,\\
e^{-\al(x-L_{0})},\,\,&x\geq L_{0}+1,
\end{cases}
\end{split}
\end{equation}
where $\al=\al(\al_{0}):=\min\{\frac{\al_{0}}{2},\frac{c_{B}}{8},c_{0}\}$ and $c_{B}>0$ is the unique speed of traveling waves of \eqref{eqn-bi-1} and $c_{0}$ is as in Proposition \ref{prop-transition-wave}. By Proposition \ref{prop-transition-wave}$\rm(ii)$, there exists $C_{L_{0}}>0$ such that
\begin{equation}\label{estimate-steepness-super-sub}
u_{x}(t,x)\leq -C_{L_{0}}\quad {\rm for} \,\,\, |x-\xi(t)|\leq L_{0}+1+d_{\max},\,\,t\in\R,
\end{equation}
where $d_{\max}>0$ is as in Proposition \ref{prop-interface-rightward-propagation-improved}. Set
\begin{equation}\label{constant-B}
M=\frac{2C_{\rm Lip}+C_{\Ga}}{C_{L_{0}}},
\end{equation}
where $C_{\rm Lip}>0$ is the Lipschitz constant for $f(t,u)$, that is,
\begin{equation*}
C_{\rm Lip}=\sup_{t\in\R}\sup_{u,v\in[0,1];u\neq v}\frac{|f(t,u)-f(t,v)|}{|u-v|}.
\end{equation*}

We also need
\begin{equation}\label{constant-omega}
\om=\om(\al_{0}):=\min\bigg\{\beta,\frac{\al c_{B}}{4}-\al^{2},C_{\rm Lip}\bigg\},
\end{equation}
where $\beta>0$ is as in $\rm(H3)$. By the choice of $\al$, $\frac{\al c_{B}}{2}-\al^{2}>0$.

Due to condition \eqref{initial-data} and the fact that $u(t_{0},x)$ is strictly decreasing in $x$ by Proposition \ref{prop-transition-wave}$\rm(i)$, for any
\begin{equation}\label{constant-ep}
\ep\in(0,\ep_{0}],\quad\text{ where}\,\,\ep_{0}=\min\bigg\{\frac{\theta}{2},\frac{1-\theta_{*}}{2},\frac{c_{B}}{4M}\bigg\},
\end{equation}
we can find two shifts $\zeta^{-}_{0}<\zeta^{+}_{0}$ (depending only on $\ep$ and $u_{0}$) such that
\begin{equation}\label{initial-condition-two-sided}
\begin{split}
u(t_{0},x-\zeta^{-}_{0})-\ep\Ga(x-\tilde{\xi}(t_{0})-\zeta^{-}_{0})\leq u_{0}(x)\leq u(t_{0},x-\zeta^{+}_{0})+\ep\Ga(x-\tilde{\xi}(t_{0})-\zeta^{+}_{0}).
\end{split}
\end{equation}
Note that we used $\tilde{\xi}(t_{0})$ here instead of $\xi(t_{0})$. Proposition \ref{prop-interface-rightward-propagation-improved} allows us to do so. Moreover, by making $\zeta_{0}^{-}$ smaller and $\zeta_{0}^{+}$ larger, we may assume, without loss of generality, that
\begin{equation}\label{initial-shift-gap-ep}
\zeta_{0}^{+}-\zeta_{0}^{-}\geq\ep.
\end{equation}

Now, we are ready to state and prove the main result in this section. Recall that  $u(t,x;t_{0},u_0)$ is the solution of \eqref{main-eqn} with initial data $u(t_{0},x;t_{0},u_0)=u_{0}(x)$.

\begin{thm}\label{thm-aprior-estimate}
Suppose $\rm(H1)$-$\rm(H3)$. Let $t_{0}\in\R$. For any $\ep\in(0,\ep_{0}]$, there are shifts
\begin{equation*}
\zeta^{-}_{1}=\zeta^{-}_{0}-\frac{M\ep}{\om}\quad\text{and}\quad \zeta^{+}_{1}=\zeta^{+}_{0}+\frac{M\ep}{\om}
\end{equation*}
such that
\begin{equation*}
u(t,x-\zeta^{-}_{1})-q(t)\Ga(x-\tilde{\xi}(t)-\zeta^{-}_{1})\leq u(t,x;t_{0}.u_0)\leq u(t,x-\zeta^{+}_{1})+q(t)\Ga(x-\tilde{\xi}(t)-\zeta^{+}_{1})
\end{equation*}
for all $x\in\R$ and $t\geq t_{0}$, where $q(t)=\ep e^{-\om(t-t_{0})}$ and $\tilde{\xi}:\R\ra\R$ is as in Proposition \ref{prop-interface-rightward-propagation-improved}.
\end{thm}
\begin{proof}
The idea of proof is to construct appropriate super-solution and sub-solution of \eqref{main-eqn} with initial data at time $t_{0}$ satisfying the second and the first estimate in \eqref{initial-condition-two-sided}, respectively.

Let us start with the super-solution. Define for $t\geq t_{0}$
\begin{equation*}
u^{+}(t,x;t_{0})=u(t,x-\zeta^{+}(t))+q(t)\Ga(x-\tilde{\xi}(t)-\zeta^{+}(t)),
\end{equation*}
where
\begin{equation*}
q(t)=\ep e^{-\om(t-t_{0})}\quad\text{and}\quad\zeta^{+}(t)=\zeta_{0}^{+}+\frac{M\ep}{\om}(1-e^{-\om(t-t_{0})}).
\end{equation*}
We show that $u^{+}$ is a super-solution of \eqref{main-eqn}, that is, $u^{+}_{t}\geq u^{+}_{xx}+f(t,u^{+})$. We consider three cases.

\textbf{Case 1.} $x-\tilde{\xi}(t)-\zeta^{+}(t)\leq-L_{0}-1$. In this case, $\Ga(x-\tilde{\xi}(t)-\zeta^{+}(t))=1$ by the definition of $\Ga$, and thus
\begin{equation*}
u^{+}(t,x;t_{0})=u(t,x-\zeta^{+}(t))+q(t).
\end{equation*}
Moreover, by Proposition \ref{prop-interface-rightward-propagation-improved}, 
\begin{equation*}
x-\zeta^{+}(t)\leq\tilde{\xi}(t)-L_{0}-1\leq\xi(t)-L_{0}-1+d_{\max}\leq\xi(t)-\frac{L_{0}}{2}
\end{equation*}
(making $L_{0}$ larger if necessary), which implies $u^{+}(t,x;t_{0})\geq u(t,x-\zeta^{+}(t))\geq\frac{1+\theta_{*}}{2}$ by \eqref{condition-L}, and hence
\begin{equation}\label{estimate-1-case-1}
f(t,u(t,x-\zeta^{+}(t)))-f(t,u^{+}(t,x;t_{0}))\geq \beta q(t)
\end{equation}
by $\rm(H3)$. We compute
\begin{equation*}
\begin{split}
u^{+}_{t}-u^{+}_{xx}-f(t,u^{+})&=u_{t}-\dot{\zeta}^{+}(t)u_{x}+\dot{q}(t)-u_{xx}-f(t,u^{+})\\
&=f(t,u)-f(t,u^{+})-Mq(t)u_{x}-\om q(t)\\
&\geq\beta q(t)-\om q(t)\geq0,
\end{split}
\end{equation*}
where we used \eqref{estimate-1-case-1}, the fact $u_{x}<0$ by Proposition \ref{prop-transition-wave}$\rm(i)$ and \eqref{constant-omega}.

\textbf{Case 2.} $x-\tilde{\xi}(t)-\zeta^{+}(t)\geq L_{0}+1$. In this case,
\begin{equation*}
\Ga(x-\tilde{\xi}(t)-\zeta^{+}(t))=e^{-\al(x-\tilde{\xi}(t)-\zeta^{+}(t)-L_{0})},
\end{equation*} 
and hence,
\begin{equation*}
u^{+}(t,x;t_{0})=u(t,x-\zeta^{+}(t))+q(t)e^{-\al(x-\tilde{\xi}(t)-\zeta^{+}(t)-L_{0})}.
\end{equation*}
Moreover, by Proposition \ref{prop-interface-rightward-propagation-improved},
$x-\zeta^{+}(t)\geq\tilde{\xi}(t)+L_{0}+1\geq\xi(t)+L_{0}+1$, which leads to $u(t,x-\zeta^{+}(t))\leq\frac{\theta}{2}$ by \eqref{condition-L}, and hence, $f(t,u(t,x-\zeta^{+}(t)))$=0. Also, by \eqref{constant-ep}, $u^{+}(t,x;t_{0})\leq u(t,x-\zeta^{+}(t))+\ep\leq\theta$, which yields $f(t,u^{+}(t,x;t_{0}))=0$. We compute
\begin{equation*}
\begin{split}
&u^{+}_{t}-u^{+}_{xx}-f(t,u^{+})\\
&\quad\quad=u_{t}-\dot{\zeta}^{+}(t)u_{x}+\Big[\dot{q}(t)+\al q(t)\Big(\dot{\tilde{\xi}}(t)+\dot{\zeta}^{+}(t)\Big)\Big]e^{-\al(x-\tilde{\xi}(t)-\zeta^{+}(t)-L_{0})}\\
&\quad\quad\quad-u_{xx}-\al^{2}q(t)e^{-\al(x-\tilde{\xi}(t)-\zeta^{+}(t)-L_{0})}-f(t,u^{+})\\
&\quad\quad=-\dot{\zeta}^{+}(t)u_{x}+\Big[\dot{q}(t)+\al q(t)\Big(\dot{\tilde{\xi}}(t)+\dot{\zeta}^{+}(t)\Big)-\al^{2}q(t)\Big]e^{-\al(x-\tilde{\xi}(t)-\zeta^{+}(t)-L_{0})}\\
&\quad\quad\geq0,
\end{split}
\end{equation*}
since $-\dot{\zeta}^{+}(t)u_{x}\geq0$ and, due to \eqref{constant-omega},
\begin{equation*}
\begin{split}
\dot{q}(t)+\al q(t)\Big(\dot{\tilde{\xi}}(t)+\dot{\zeta}^{+}(t)\Big)-\al^{2}q(t)\geq\Big[-\om+\frac{\al c_{B}}{2}+\al Mq(t)-\al^{2}\Big]q(t)\geq0.
\end{split}
\end{equation*}

\textbf{Case 3.} $x-\tilde{\xi}(t)-\zeta^{+}(t)\in[-L_{0}-1, L_{0}+1]$. In this case,
\begin{equation*}
x-\zeta^{+}(t)-\xi(t)=x-\zeta^{+}(t)-\tilde{\xi}(t)+\tilde{\xi}(t)-\xi(t)\in[-L_{0}-1,L_{0}+1+d_{\max}]
\end{equation*}
by Proposition \ref{prop-interface-rightward-propagation-improved}. It then follows from \eqref{estimate-steepness-super-sub} that
\begin{equation}\label{estimate-steepness-super}
u_{x}(t,x-\zeta^{+}(t))<-C_{L_{0}}.
\end{equation}
We compute
\begin{equation*}
\begin{split}
u^{+}_{t}-u^{+}_{xx}-f(t,u^{+})&=u_{t}-\dot{\zeta}^{+}(t)u_{x}+\dot{q}(t)\Ga(x-\tilde{\xi}(t)-\zeta^{+}(t))\\
&\quad-q(t)\big[\dot{\tilde{\xi}}(t)+\dot{\zeta}^{+}(t)\big]\Ga_{x}(x-\tilde{\xi}(t)-\zeta^{+}(t))\\
&\quad-u_{xx}-q(t)\Ga_{xx}(x-\xi_{t_{0}}(t)-\zeta^{+}(t))-f(t,u^{+})\\
&\geq f(t,u)-f(t,u^{+})-\dot{\zeta}^{+}(t)u_{x}-\om q(t)-C_{\Ga}q(t),
\end{split}
\end{equation*}
where we used $\Ga(x-\tilde{\xi}(t)-\zeta^{+}(t))\leq1$, $q(t)\big[\dot{\tilde{\xi}}(t)+\dot{\zeta}^{+}(t)\big]\Ga_{x}(x-\tilde{\xi}(t)-\zeta^{+}(t))\leq0$ and $\Ga_{xx}(x-\tilde{\xi}(t)-\zeta^{+}(t))-f(t,u^{+})\leq C_{\Ga}$. By the Lipschitz continuity and \eqref{estimate-steepness-super}, we deduce
\begin{equation*}
u^{+}_{t}-u^{+}_{xx}-f(t,u^{+})\geq\big(-C_{\rm Lip}+MC_{L_{0}}-\om-C_{\Ga}\big)q(t)\geq0
\end{equation*}
by \eqref{constant-B} and \eqref{constant-omega}.

Hence, \textbf{Case 1}, \textbf{Case 2} and \textbf{Case 3} imply $u^{+}_{t}\geq u^{+}_{xx}+f(t,u^{+})$, i.e., $u^{+}(t,x;t_{0})$ is a super-solution of $\eqref{main-eqn}$. It then follows from
the second inequality in \eqref{initial-condition-two-sided} and the comparison principle that
\begin{equation*}
\begin{split}
u(t,x;t_{0},u_0)\leq u^{+}(t,x;t_{0})&=u(t,x-\zeta^{+}(t))+q(t)\Ga(x-\tilde{\xi}(t)-\zeta^{+}(t))\\
&\leq u(t,x-\zeta^{+}_{1})+q(t)\Ga(x-\tilde{\xi}(t)-\zeta^{+}_{1}),
\end{split}
\end{equation*}
where the last inequality follows from the facts that $u(t,x)$ and $\Ga(x)$ are decreasing in $x$, and $\zeta^{+}(t)$ is strictly increasing and converges to $\zeta^{+}_{1}$ as $t\ra\infty$. This proves half of the theorem.

We now construct a sub-solution of \eqref{main-eqn} to prove the remaining half. Define for $t\geq t_{0}$
\begin{equation*}
u^{-}(t,x;t_{0})=u(t,x-\zeta^{-}(t))-q(t)\Ga(x-\tilde{\xi}(t)-\zeta^{-}(t)),
\end{equation*}
where
\begin{equation*}
\zeta^{-}(t)=\zeta_{0}^{-}-\frac{M\ep}{\om}(1-e^{-\om(t-t_{0})}).
\end{equation*}
We show that $u^{-}$ is a sub-solution of \eqref{main-eqn}, that is, $u^{-}_{t}\leq u^{-}_{xx}+f(t,u^{-})$. We consider three cases.

\textbf{Case I.} $x-\tilde{\xi}(t)-\zeta^{-}(t)\leq-L_{0}-1$. In this case, $\Ga(x-\tilde{\xi}(t)-\zeta^{-}(t))=1$ by the definition of $\Ga$, and thus
\begin{equation*}
u^{-}(t,x;t_{0})=u(t,x-\zeta^{-}(t))-q(t).
\end{equation*}
Moreover, by Proposition \ref{prop-interface-rightward-propagation-improved}, 
\begin{equation*}
x-\zeta^{-}(t)\leq\tilde{\xi}(t)-L_{0}-1\leq\xi(t)-L_{0}-1+d_{\max}\leq\xi(t)-\frac{L_{0}}{2} 
\end{equation*}
(making $L_{0}$ larger if necessary), which implies $u(t,x-\zeta^{-}(t))\geq\frac{1+\theta_{*}}{2}$ by \eqref{condition-L}. Also, $u^{-}(t,x;t_{0})=u(t,x-\zeta^{-}(t))-q(t)\geq\frac{1+\theta_{*}}{2}-\ep\geq\theta_{*}$ by \eqref{constant-ep}. It then follows from $\rm(H3)$ that
\begin{equation}\label{estimate-1-case-I}
f(t,u(t,x-\zeta^{-}(t)))-f(t,u^{-}(t,x;t_{0}))\leq -\beta q(t).
\end{equation}
We compute
\begin{equation*}
\begin{split}
u^{-}_{t}-u^{-}_{xx}-f(t,u^{-})&=u_{t}-\dot{\zeta}^{-}(t)u_{x}-\dot{q}(t)-u_{xx}-f(t,u^{-})\\
&=f(t,u)-f(t,u^{-})+Mq(t)u_{x}+\om q(t)\\
&\leq-\beta q(t)+\om q(t)\leq0,
\end{split}
\end{equation*}
where we used \eqref{estimate-1-case-I}, the fact $u_{x}<0$ by Proposition \ref{prop-transition-wave}$\rm(i)$ and \eqref{constant-omega}.

\textbf{Case II.} $x-\tilde{\xi}(t)-\zeta^{-}(t)\geq L_{0}+1$. In this case,
\begin{equation*}
\Ga(x-\tilde{\xi}(t)-\zeta^{-}(t))=e^{-\al(x-\tilde{\xi}(t)-\zeta^{-}(t)-L_{0})}, 
\end{equation*}
and hence,
\begin{equation*}
u^{-}(t,x;t_{0})=u(t,x-\zeta^{-}(t))-q(t)e^{-\al(x-\tilde{\xi}(t)-\zeta^{-}(t)-L_{0})}.
\end{equation*}
Moreover, by Proposition \ref{prop-interface-rightward-propagation-improved}, $x-\zeta^{-}(t)\geq\tilde{\xi}(t)+L_{0}+1\geq\xi(t)+L_{0}+1$,
which leads to $u(t,x-\zeta^{-}(t))\leq\frac{\theta}{2}$ by \eqref{condition-L}, and hence, $f(t,u(t,x-\zeta^{-}(t)))$=0. Clearly, $f(t,u^{-}(t,x;t_{0}))=0$. We compute
\begin{equation*}
\begin{split}
&u^{-}_{t}-u^{-}_{xx}-f(t,u^{-})\\
&\quad\quad=u_{t}-\dot{\zeta}^{-}(t)u_{x}-\Big[\dot{q}(t)+\al q(t)\Big(\dot{\tilde{\xi}}(t)+\dot{\zeta}^{-}(t)\Big)\Big]e^{-\al(x-\tilde{\xi}(t)-\zeta^{-}(t)-L_{0})}\\
&\quad\quad\quad-u_{xx}+\al^{2}q(t)e^{-\al(x-\tilde{\xi}(t)-\zeta^{-}(t)-L_{0})}-f(t,u^{-})\\
&\quad\quad=-\dot{\zeta}^{-}(t)u_{x}-\Big[\dot{q}(t)+\al q(t)\Big(\dot{\tilde{\xi}}(t)+\dot{\zeta}^{-}(t)\Big)-\al^{2}q(t)\Big]e^{-\al(x-\tilde{\xi}(t)-\zeta^{-}(t)-L_{0})}\\
&\quad\quad\leq0,
\end{split}
\end{equation*}
since $-\dot{\zeta}^{-}(t)u_{x}\leq0$ and, due to \eqref{constant-omega},
\begin{equation*}
\begin{split}
\dot{q}(t)+\al q(t)\Big(\dot{\tilde{\xi}}(t)+\dot{\zeta}^{-}(t)\Big)-\al^{2}q(t)\geq\Big[-\om+\frac{\al c_{B}}{2}-\frac{\al c_{B}}{4}-\al^{2}\Big]q(t)\geq0.
\end{split}
\end{equation*}

\textbf{Case III.} $x-\tilde{\xi}(t)-\zeta^{-}(t)\in[-L_{0}-1, L_{0}+1]$. In this case,
\begin{equation*}
x-\zeta^{-}(t)-\xi(t)=x-\zeta^{-}(t)-\tilde{\xi}(t)+\tilde{\xi}(t)-\xi(t)\in[-L_{0}-1,L_{0}+1+d_{\max}]
\end{equation*}
by Proposition \ref{prop-interface-rightward-propagation-improved}. It then follows from \eqref{estimate-steepness-super-sub} that
\begin{equation}\label{estimate-steepness-super-II}
u_{x}(t,x-\zeta^{-}(t))<-C_{L_{0}}.
\end{equation}
We compute
\begin{equation*}
\begin{split}
u^{-}_{t}-u^{-}_{xx}-f(t,u^{-})&=u_{t}-\dot{\zeta}^{-}(t)u_{x}-\dot{q}(t)\Ga(x-\tilde{\xi}(t)-\zeta^{-}(t))\\
&\quad+q(t)\big[\dot{\tilde{\xi}}(t)+\dot{\zeta}^{-}(t)\big]\Ga_{x}(x-\tilde{\xi}(t)-\zeta^{-}(t))\\
&\quad-u_{xx}+q(t)\Ga_{xx}(x-\tilde{\xi}(t)-\zeta^{-}(t))-f(t,u^{-})\\
&\leq f(t,u)-f(t,u^{-})-\dot{\zeta}^{-}(t)u_{x}+\om q(t)+C_{\Ga}q(t),
\end{split}
\end{equation*}
where we used $\Ga(x-\tilde{\xi}(t)-\zeta^{-}(t))\leq1$, $\Ga_{xx}(x-\tilde{\xi}(t)-\zeta^{+}(t))-f(t,u^{+})\leq C_{\Ga}$ and
\begin{equation*}
q(t)\big[\dot{\tilde{\xi}}(t)+\dot{\zeta}^{-}(t)\big]\Ga_{x}(x-\tilde{\xi}(t)-\zeta^{-}(t))\leq q(t)\big[\frac{c_{B}}{2}-Mq(t)\big]\Ga_{x}(x-\tilde{\xi}(t)-\zeta^{-}(t))\leq0
\end{equation*}
by \eqref{constant-ep}. By the Lipschitz continuity and \eqref{estimate-steepness-super-II}, we deduce
\begin{equation*}
u^{-}_{t}-u^{-}_{xx}-f(t,u^{-})\leq\big(C_{\rm Lip}-MC_{L_{0}}+\om+C_{\Ga}\big)q(t)\leq0
\end{equation*}
by \eqref{constant-B} and \eqref{constant-omega}.

Hence, \textbf{Case I}, \textbf{Case II} and \textbf{Case III} imply $u^{-}_{t}\leq u^{-}_{xx}+f(t,u^{-})$, i.e., $u^{-}(t,x;t_{0})$ is a sub-solution of $\eqref{main-eqn}$. It then follows from
the first inequality in \eqref{initial-condition-two-sided} and the comparison principle that
\begin{equation*}
\begin{split}
u(t,x;t_{0},u_0)\geq u^{-}(t,x;t_{0})&=u(t,x-\zeta^{-}(t))+q(t)\Ga(x-\tilde{\xi}(t)-\zeta^{-}(t))\\
&\geq u(t,x-\zeta^{-}_{1})+q(t)\Ga(x-\tilde{\xi}(t)-\zeta^{-}_{1}),
\end{split}
\end{equation*}
where the last inequality follows from the facts that $u(t,x)$ and $\Ga(x)$ are decreasing in $x$, and $\zeta^{-}(t)$ is strictly decreasing and converges to $\zeta^{-}_{1}$ as $t\ra\infty$. This completes the proof.
\end{proof}

We end up this section with a remark concerning Theorem \ref{thm-aprior-estimate}.

\begin{rem}\label{rem-thm-aprior-estimate}
\begin{itemize}
\item[\rm(i)] Theorem \ref{thm-aprior-estimate} is not tailored for the initial data $u_{0}$. All we need in the proof of Theorem \ref{thm-aprior-estimate} is the initial two-sided estimate \eqref{initial-condition-two-sided} for $u_{0}$. Hence, if initially we have the estimate in the form of \eqref{initial-condition-two-sided}, we will be able to apply Theorem \ref{thm-aprior-estimate}. This observation is helpful in the following sections.

\item[\rm(ii)] From the proof of Theorem \ref{thm-aprior-estimate}, the lower bound and the upper bound for $u(t,x;t_{0},u_0)$ in Theorem \ref{thm-aprior-estimate} are independent: the lower bound and the upper bound for $u(t,x;t_{0},u_0)$ depend only on the lower bound and the upper bound for $u_{0}$ in \eqref{initial-condition-two-sided}, respectively.
\end{itemize}
\end{rem}


\section{Stability of Generalized Traveling Waves}\label{sec-stability}

In this section, we study the stability of the generalized traveling wave $u^f(t,x)$ of \eqref{main-eqn} in Proposition \ref{prop-transition-wave} and prove Theorem \ref{thm-stability-tw}. Throughout this section, we still write $u^f(t,x)$ and $\xi^f(t)$ as $u(t,x)$ and $\xi(t)$, respectively.

The proof of Theorem \ref{thm-stability-tw} is based on the following lemma, which is the time heterogeneous version of \cite[Proposition 2.2]{MNRR09}, where the space heterogeneous nonlinearity is treated.

\begin{prop}\label{prop-stability}
Suppose $\rm(H1)$-$\rm(H3)$. Fix $u_0$ satisfying \eqref{initial-data}.
 Let $\zeta_0^{+}$ and $\zeta_0^{-}$ be as in \eqref{initial-condition-two-sided}.
\begin{itemize}
\item[\rm(i)] There exist a time $T=T(\al_{0},\zeta_{0}^{+}-\zeta_{0}^{-})>0$,  two shifts $\zeta^{-}_{T}<\zeta_{T}^{+}$ with $\zeta_{T}^{+}-\zeta^{-}_{T}\leq1$,
and $q_T$ with $0\le q_T\le \min\{\ep_{0},\zeta_{0}^{+}-\zeta_{0}^{-}\}$ such that
\begin{equation*}
\begin{split}
&u(t_{0}+T,x-\zeta^{-}_{T})-q_{T}\Ga(x-\tilde{\xi}(t_{0}+T)-\zeta^{-}_{T})\\
&\leq u(t_{0}+T,x;t_{0},u_0)\leq u(t_{0}+T,x-\zeta^{+}_{T})+q_{T}\Ga(x-\tilde{\xi}(t_{0}+T)-\zeta^{+}_{T})
\end{split}
\end{equation*}
for all $x\in\R$.

\item[\rm(ii)] There are constants $\tilde{T}=\tilde{T}(\al_{0})>0$, $K=K(\al_{0})>0$ and $\ga=\ga(\al_{0})\in(0,1)$, and three sequences $\{\zeta_{n}^{-}\}_{n\in\N}$, $\{\zeta_{n}^{+}\}_{n\in\N}$ and $\{q_{n}\}_{n\in\N}$ satisfying
\begin{equation}\label{estimate-prop-stab-shifts}
\begin{split}
&0\leq q_{n}\leq\min\{\ep_{0},\zeta_{n}^{+}-\zeta_{n}^{-}\}\leq\zeta_{n}^{+}-\zeta_{n}^{-}\leq K\ga^{n}\quad\text{and}\\
&\zeta_{n}^{-}-K(\zeta_{n}^{+}-\zeta_{n}^{-})\leq\zeta_{n+1}^{-}\leq\zeta_{n+1}^{+}\leq\zeta_{n}^{+}+K(\zeta_{n}^{+}-\zeta_{n}^{-})
\end{split}
\end{equation}
for all $n\in\N$ such that for any $n\in\N$ there holds
\begin{equation}\label{estimate-prop-stab-main}
u(t,x-\zeta^{-}_{n})-q_{n}e^{-\om(t-t_{n})}\leq u(t,x;t_{0},u_0)\leq u(t,x-\zeta^{+}_{n})+q_{n}e^{-\om(t-t_{n})}
\end{equation}
for all $x\in\R$ and $t\geq t_{n}=t_{0}+T+n\tilde{T}$.
\end{itemize}
\end{prop}

Assuming Proposition \ref{prop-stability}, let us prove Theorem \ref{thm-stability-tw}.

\begin{proof}[Proof of Theorem \ref{thm-stability-tw}]
We see from \eqref{estimate-prop-stab-shifts} that $\lim_{n\ra\infty}\zeta_{n}^{-}=\zeta_{*}=\lim_{n\ra\infty}\zeta_{n}^{+}$ for some $\zeta_{*}\in\R$, and there is some $C=C(\al_{0})$ such that $|\zeta_{n}^{-}-\zeta_{*}|\leq C\ga^{n}$ and $|\zeta_{n}^{+}-\zeta_{*}|\leq C\ga^{n}$ for all $n\in\N$. Also, from $0\leq q_{n}\leq K\ga^{n}$ and the estimate \eqref{estimate-prop-stab-main}, we have
\begin{equation}\label{an-auxiliary-estimate-stability}
u(t,x-\zeta^{-}_{n})-K\ga^{n}e^{-\om(t-t_{n})}\leq u(t,x;t_{0},u_0)\leq u(t,x-\zeta^{+}_{n})+K\ga^{n}e^{-\om(t-t_{n})}
\end{equation}
for $t\geq t_{n}$. Note that using $|\zeta_{n}^{+}-\zeta_{*}|\leq C\ga^{n}$ and Proposition \ref{prop-transition-wave}, we deduce
\begin{equation*}
u(t,x-\zeta_{n}^{+})\leq u(t,x-\zeta_{*}-C\ga^{n})\leq u(t,x-\zeta_{*})+\tilde C\ga^{n},
\end{equation*}
where $\tilde C=C\cdot\sup_{t\in\R,x\in\R}\{-u_x(t,x)\}$. Similarly, we have
\begin{equation*}
u(t,x-\zeta_{n}^{-})\geq u(t,x-\zeta_{*}+C\ga^{n})\geq u(t,x-\zeta_{*})-\tilde C\ga^{n}.
\end{equation*}
It then follows from \eqref{an-auxiliary-estimate-stability} that
\begin{equation*}
u(t,x-\zeta_{*})-\ga^{n}(\tilde{C}+Ke^{-\om(t-t_{n})})\leq u(t,x;t_{0},u_0)\leq u(t,x-\zeta_{*})+\ga^{n}(\tilde{C}+Ke^{-\om(t-t_{n})})
\end{equation*}
for $t\geq t_{n}$. In particular, for $t\geq t_{0}+T$, there holds
\begin{equation*}
|u(t,x;t_{0},u_0)-u(t,x-\zeta_{*})|\leq \bar Ce^{-r(t-t_{0}-T)}
\end{equation*}
for some $\bar{C}=\bar{C}(\al_{0})>0$ and  $r=r(\al_{0})>0$. The result then follows.
\end{proof}

The rest of this section is devoted to the proof of Proposition \ref{prop-stability}. If no confusion occurs, we write $u(t,x;t_0,u_0)$ as $u(t,x;t_0)$ in the rest of this section.

\begin{proof}[Proof of Proposition \ref{prop-stability}]
$\rm(i)$ Let $\zeta_{1}^{-}$ and $\zeta_{1}^{+}$ be as in Theorem \ref{thm-aprior-estimate}. By Theorem  \ref{thm-aprior-estimate}, we have
\begin{equation*}
\begin{split}
&u(t,x-\zeta^{-}_{1})-\ep e^{-\om(t-t_{0})}\Ga(x-\tilde{\xi}(t)-\zeta^{-}_{1})\\
&\leq u(t,x;t_{0})\leq u(t,x-\zeta^{+}_{1})+\ep e^{-\om(t-t_{0})}\Ga(x-\tilde{\xi}(t)-\zeta^{+}_{1})
\end{split}
\end{equation*}
for all $x\in\R$ and $t\geq t_{0}$. Let $s_{1}>0$ (to be chosen at the end of the proof). Since $e^{-\om(t-t_{0})}\leq e^{-\om(t-t_{0}-s_{1})}$, we have
\begin{equation*}
\begin{split}
&u(t,x-\zeta^{-}_{1})-q_{1}e^{-\om(t-t_{0}-s_{1})}\Ga(x-\tilde{\xi}(t)-\zeta^{-}_{1})\\
&\leq u(t,x;t_{0})\leq u(t,x-\zeta^{+}_{1})+q_{1}e^{-\om(t-t_{0}-s_{1})}\Ga(x-\tilde{\xi}(t)-\zeta^{+}_{1})
\end{split}
\end{equation*}
for all $x\in\R$ and $t\geq t_{0}+s_{1}$, where $q_{1}=\ep\leq\zeta_{0}^{+}-\zeta_{0}^{-}\leq\zeta_{1}^{+}-\zeta_{1}^{-}$ by \eqref{initial-shift-gap-ep}. We remark that the above estimate holds for an arbitrary $s_{1}>0$.

We now show that there are  constants $C>0$ and $\ga_{0}=\ga_{0}(\al_{0},\zeta_{1}^{+}-\zeta_{1}^{-})\in(0,1)$ such that for any $n\geq 2$ there holds
\begin{itemize}
\item[\rm$\textbf{Claim}_{n}$] for $s_{n}=n s_{1}$ there holds
\begin{equation*}
\begin{split}
&u(t,x-\zeta^{-}_{n})-q_{n}e^{-\om(t-t_{0}-s_{n})}\Ga(x-\tilde{\xi}(t)-\zeta^{-}_{n})\\
&\leq u(t,x;t_{0})\leq u(t,x-\zeta^{+}_{n})+q_{n}e^{-\om(t-t_{0}-s_{n})}\Ga(x-\tilde{\xi}(t)-\zeta^{+}_{n})
\end{split}
\end{equation*}
for all $t\geq t_{0}+s_{n}$, where
\begin{equation*}
\begin{split}
\zeta_{1}^{-}-1\leq \zeta_{n}^{-}\leq \zeta_{n}^{+}\leq \zeta_{1}^{+}+1\quad\text{and}\quad 0\leq q_{n}\leq\zeta_{n}^{+}-\zeta_{n}^{-}\leq C\ga_{0}^{n}.
\end{split}
\end{equation*}
\end{itemize}

We prove this by induction. Assuming $\textbf{Claim}_{n}$, we verify $\textbf{Claim}_{n+1}$. Set
\begin{equation*}
R=R(\zeta_{1}^{+}-\zeta_{1}^{-})=2\bigg(d_{\max}+\frac{L_{0}}{2}+1+\frac{\zeta_{1}^{+}-\zeta_{1}^{-}}{2}\bigg).
\end{equation*}
The reason for such a choice is that the set
\begin{equation*}
\Om_{R}=\bigg\{(t,x)\in\R\times\R\bigg|\Big|x-\tilde{\xi}(t)-\frac{\zeta_{1}^{-}+\zeta_{1}^{+}}{2}\Big|\leq R\bigg\}
\end{equation*}
contains the set $\big\{(t,x)\in\R\times\R\big|u(t,x-(\zeta_{1}^{-}-1))\leq\frac{1+\theta_{*}}{2}\,\,\text{and}\,\, u(t,x-(\zeta_{1}^{+}+1))\geq\frac{\theta}{2}\big\}$.

Let us consider $u(t,x-\zeta_{n}^{-})$ and $u(t,x-\zeta_{n}^{+})$ for $(t,x)\in\Om_{3R}$. By Taylor expansion
\begin{equation*}
u(t,x-\zeta_{n}^{+})-u(t,x-\zeta_{n}^{-})=-u_{x}(t,x-\zeta_{n}^{*})(\zeta_{n}^{+}-\zeta_{n}^{-}),
\end{equation*}
where $\zeta_{n}^{*}\in[\zeta_{n}^{-},\zeta_{n}^{+}]\subset[\zeta_{1}^{-}-1,\zeta_{1}^{+}+1]$ by induction assumption.
For $(t,x)\in\Om_{3R}$, we easily check that $|x-\zeta_{n}^{*}-\xi(t)|$ is bounded by some constant depending only on $\zeta_{1}^{+}-\zeta_{1}^{-}$, which together with Proposition \ref{prop-transition-wave}$\rm(ii)$, ensures the existence of some $K=K(\zeta_{1}^{+}-\zeta_{1}^{-})>0$ such that $u_{x}(t,x-\zeta_{n}^{*})\leq-K$. It then follows that
\begin{equation}\label{estimate-om-3R}
u(t,x-\zeta_{n}^{+})-u(t,x-\zeta_{n}^{-})\geq K(\zeta_{n}^{+}-\zeta_{n}^{-}),\quad (t,x)\in\Om_{3R}.
\end{equation}

Set $\tilde{s}_{n}=s_{n}+\tilde{s}$ with $\tilde{s}=\frac{\ln p}{\om}$ for some $p=p(\al_{0},\zeta_{1}^{+}-\zeta_{1}^{-})$ large (to be chosen). Let $\ep_{n}=\frac{\zeta_{n}^{+}-\zeta_{n}^{-}}{pK}$. We deduce from the induction assumption that if $t\geq t_{0}+\tilde{s}_{n}$ and $(t,x)\in\Omega_{3R}$, then
\begin{equation*}
\begin{split}
u(t,x;t_{0})\geq u(t,x-\zeta_{n}^{-})-\frac{\zeta_{n}^{+}-\zeta_{n}^{-}}{p}=u(t,x-\zeta_{n}^{-})-\ep_{n}K,\\
u(t,x;t_{0})\leq u(t,x-\zeta_{n}^{+})+\frac{\zeta_{n}^{+}-\zeta_{n}^{-}}{p}=u(t,x-\zeta_{n}^{+})+\ep_{n}K.
\end{split}
\end{equation*}
Moreover, for $(t,x)\in\Om_{3R}$, we see
\begin{equation*}
\begin{split}
&u(t,x-\zeta_{n}^{+}-\ep_{n})-u(t,x-\zeta_{n}^{+})\geq K\ep_{n}\quad\text{and}\\
&u(t,x-\zeta_{n}^{-}+\ep_{n})-u(t,x-\zeta_{n}^{-})\leq-K\ep_{n}
\end{split}
\end{equation*}
provided $p$ is sufficiently large so that $\ep_{n}$ is small. Hence, for $(t,x)\in\Om_{3R}$ and $t\geq t_{0}+\tilde{s}_{n}$ there hold
\begin{equation}\label{positivity-estimate-harnack}
u(t,x;t_{0})\geq u(t,x-\zeta_{n}^{-}+\ep_{n})\quad\text{and}\quad u(t,x;t_{0})\leq u(t,x-\zeta_{n}^{+}-\ep_{n}).
\end{equation}
Using the estimates \eqref{estimate-om-3R} and \eqref{positivity-estimate-harnack}, and the monotonicity of $u(t,x)$ in $x$, we obtain
\begin{equation*}
\big[u(t,x-\zeta_{n}^{+}-\ep_{n})-u(t,x;t_{0})\big]+\big[u(t,x;t_{0})-u(t,x-\zeta_{n}^{-}+\ep_{n})\big]\geq K(\zeta_{n}^{+}-\zeta_{n}^{-})
\end{equation*}
for $(t,x)\in\Om_{3R}$ and $t\geq t_{0}+\tilde{s}_{n}$. In particular, at the moment $t=t_{0}+\tilde{s}_{n}+\frac{\si}{2}$ for some $\si=\si(\al_{0},\zeta_{1}^{+}-\zeta_{1}^{-})>0$ to be chosen, we must have either
\begin{equation}\label{one-in-two-1}
\sup_{\Om_{3R}\cap(\{t\}\times\R)}\big[u(t,x-\zeta_{n}^{+}-\ep_{n})-u(t,x;t_{0})\big]\geq\frac{K}{2}(\zeta_{n}^{+}-\zeta_{n}^{-})
\end{equation}
or
\begin{equation}\label{one-in-two-2}
\sup_{\Om_{3R}\cap(\{t\}\times\R)}\big[u(t,x;t_{0})-u(t,x-\zeta_{n}^{-}+\ep_{n})\big]\geq\frac{K}{2}(\zeta_{n}^{+}-\zeta_{n}^{-}).
\end{equation}
Due to \eqref{positivity-estimate-harnack}, we can apply Harnack inequality to both $u(t,x-\zeta_{n}^{+}-\ep_{n})-u(t,x;t_{0})$ and $u(t,x;t_{0})-u(t,x-\zeta_{n}^{-}+\ep_{n})$. As a result, there is $p_{0}=p_{0}(\si,\tau)=p_{0}(\al_{0},\zeta_{1}^{+}-\zeta_{1}^{-})>0$, where $\tau=\tau(\al_{0},\zeta_{1}^{+}-\zeta_{1}^{-})>0$ is to be chosen, such that
\begin{equation}\label{one-in-two-1-harnack}
\inf_{\Om_{R}\cap([t_{0}+\tilde{s}_{n}+\si,t_{0}+\tilde{s}_{n}+\si+\tau]\times\R)}\big[u(t,x-\zeta_{n}^{+}-\ep_{n})-u(t,x;t_{0})\big]\geq\frac{p_{0}K}{2}(\zeta_{n}^{+}-\zeta_{n}^{-})
\end{equation}
if \eqref{one-in-two-1} holds, and
\begin{equation}\label{one-in-two-2-harnack}
\inf_{\Om_{R}\cap([t_{0}+\tilde{s}_{n}+\si,t_{0}+\tilde{s}_{n}+\si+\tau]\times\R)}\big[u(t,x;t_{0})-u(t,x-\zeta_{n}^{-}+\ep_{n})\big]\geq\frac{p_{0}K}{2}(\zeta_{n}^{+}-\zeta_{n}^{-})
\end{equation}
if \eqref{one-in-two-2} holds.

From now on, we assume \eqref{one-in-two-1} and \eqref{one-in-two-1-harnack}. The case with \eqref{one-in-two-2} and \eqref{one-in-two-2-harnack} can be treated similarly. Set $r_{0}=\frac{p_{0}K}{2}$. For $(t,x)\in\Om_{R}$ and $t\in[t_{0}+\tilde{s}_{n}+\si,t_{0}+\tilde{s}_{n}+\si+\tau]$, we deduce from the fact $\inf_{t\in\R,x\in\R}u_{x}(t,x)\geq-C_{*}$ for some $C_{*}>0$ by a priori estimates for parabolic equations and the estimate \eqref{one-in-two-1-harnack} that
\begin{equation*}
\begin{split}
&u(t,x-\zeta_{n}^{+}+dr_{0}(\zeta_{n}^{+}-\zeta_{n}^{-}))-u(t,x;t_{0})\\
&\quad\quad\geq u(t,x-\zeta_{n}^{+}-\ep_{n})-u(t,x;t_{0})-C_{*}(\ep_{n}+dr_{0}(\zeta_{n}^{+}-\zeta_{n}^{-}))\\
&\quad\quad\geq r_{0}(\zeta_{n}^{+}-\zeta_{n}^{-})-C_{*}(\ep_{n}+dr_{0}(\zeta_{n}^{+}-\zeta_{n}^{-}))\\
&\quad\quad=\bigg(r_{0}-\frac{C_{*}}{pK}-C_{*}dr_{0}\bigg)(\zeta_{n}^{+}-\zeta_{n}^{-}).
\end{split}
\end{equation*}
We now choose $p=p(\al_{0},\zeta_{1}^{+}-\zeta_{1}^{-})$ sufficient large and $d=d(\al_{0},\zeta_{1}^{+}-\zeta_{1}^{-})$ sufficient small such that $r_{0}-\frac{C_{*}}{pK}-C_{*}dr_{0}\geq0$. Thus, for $(t,x)\in\Om_{R}$ and $t\in[t_{0}+\tilde{s}_{n}+\si,t_{0}+\tilde{s}_{n}+\si+\tau]$, we have
\begin{equation}\label{estimate-om-r-condition}
u(t,x-\zeta_{n}^{+}+dr_{0}(\zeta_{n}^{+}-\zeta_{n}^{-}))-u(t,x;t_{0})\geq0.
\end{equation}

Next, we estimate $u(t,x-\zeta_{n}^{+}+dr_{0}(\zeta_{n}^{+}-\zeta_{n}^{-}))-u(t,x;t_{0})$ for $(t,x)\notin\Om_{R}$ and $t\in[t_{0}+\tilde{s}_{n}+\si,t_{0}+\tilde{s}_{n}+\si+\tau]$. We distinguish between
\begin{equation*}
\begin{split}
\tilde{\Om}_{R}^{-}&=\Om_{R}^{-}\cap([t_{0}+\tilde{s}_{n}+\si,t_{0}+\tilde{s}_{n}+\si+\tau]\times\R)\quad\text{and}\\
\tilde{\Om}_{R}^{-}&=\Om_{R}^{+}\cap([t_{0}+\tilde{s}_{n}+\si,t_{0}+\tilde{s}_{n}+\si+\tau]\times\R),
\end{split}
\end{equation*}
where
\begin{equation*}
\begin{split}
\Om_{R}^{-}&=\bigg\{(t,x)\in\R\times\R\bigg|x-\tilde{\xi}(t)-\frac{\zeta_{1}^{-}+\zeta_{1}^{+}}{2}\leq-R\bigg\}\quad\text{and}\\
\Om_{R}^{+}&=\bigg\{(t,x)\in\R\times\R\bigg|x-\tilde{\xi}(t)-\frac{\zeta_{1}^{-}+\zeta_{1}^{+}}{2}\geq R\bigg\}
\end{split}
\end{equation*}
are the regions left and right to $\Om_{R}$, respectively. Clearly, if $(t,x)\notin\Om_{R}$, then $(t,x)\in\Om_{R}^{-}$ or $(t,x)\in\Om_{R}^{+}$.

\textbf{Case $(t,x)\in\tilde{\Om}_{R}^{-}$.} If $(t,x)\in\tilde{\Om}_{R}^{-}$, we check $x-(\zeta_{1}^{-}-1)-\xi(t)\leq-\frac{L_{0}}{2}$ by the definition of $R$, and hence $u(t,x-(\zeta_{1}^{-}-1))\geq\frac{1+\theta_{*}}{2}$ by \eqref{condition-L}. Choosing $d=d(\al_{0},\zeta_{1}^{+}-\zeta_{1}^{-})$ smaller if necessary so that $dr_{0}<\frac{1}{2}$, we have
\begin{equation*}
\zeta_{1}^{-}-1\leq\zeta_{n}^{-}\leq\zeta_{n}^{+}-dr_{0}(\zeta_{n}^{+}-\zeta_{n}^{-}),
\end{equation*}
where the first inequality is due to the induction assumption. The monotonicity of $u(t,x)$ in $x$ then yields
\begin{equation}\label{aux-om-negative-1}
u(t,x-\zeta_{n}^{+}+dr_{0}(\zeta_{n}^{+}-\zeta_{n}^{-}))\geq u(t,x-(\zeta_{1}^{-}-1))\geq\frac{1+\theta_{*}}{2}.
\end{equation}
Moreover, by the induction assumption
\begin{equation}\label{aux-om-negative-2}
\begin{split}
u(t,x;t_{0})&\geq u(t,x-\zeta^{-}_{n})-q_{n}e^{-\om(t-t_{0}-s_{n})}\Ga(x-\tilde{\xi}(t)-\zeta^{-}_{n})\\
&\geq u(t,x-(\zeta^{-}_{1}-1))-(\zeta_{n}^{+}-\zeta_{n}^{-})e^{-\om\si}\\
&\geq \frac{1+\theta_{*}}{2}-\frac{1-\theta_{*}}{2}=\theta_{*},
\end{split}
\end{equation}
where $\si=\si(\al_{0},\zeta_{1}^{+}-\zeta_{1}^{-})$ is large so that $(\zeta_{n}^{+}-\zeta_{n}^{-})e^{-\om\si}\leq(\zeta_{1}^{+}-\zeta_{1}^{-}+2)e^{-\om\si}\leq\frac{1-\theta_{*}}{2}$. Setting
\begin{equation*}
v(t,x;t_{0})=u(t,x-\zeta_{n}^{+}+dr_{0}(\zeta_{n}^{+}-\zeta_{n}^{-}))-u(t,x;t_{0}),
\end{equation*}
we have
\begin{equation*}
v_{t}=v_{xx}+f(t,u(t,x-\zeta_{n}^{+}+dr_{0}(\zeta_{n}^{+}-\zeta_{n}^{-})))-f(t,u(t,x;t_{0}))=v_{xx}+a(t,x)v,
\end{equation*}
where $a(t,x)\leq -\beta$ due to \eqref{aux-om-negative-1}, \eqref{aux-om-negative-2} and $\rm(H3)$. By \eqref{estimate-om-r-condition}, $v(t,x;t_{0})$ is nonnegative on the boundary of $\tilde{\Om}_{R}^{-}$. At the initial moment $t_{0}+\tilde{s}_{n}+\si$, we deduce from the induction assumption and the fact $\inf_{t\in\R,x\in\R}u_{x}(t,x)\geq-C_{*}$ for some $C_{*}>0$ that
\begin{equation*}
\begin{split}
&v(t_{0}+\tilde{s}_{n}+\si,x;t_{0})\\
&\quad\quad=u(t_{0}+\tilde{s}_{n}+\si,x-\zeta_{n}^{+}+dr_{0}(\zeta_{n}^{+}-\zeta_{n}^{-}))-u(t_{0}+\tilde{s}_{n}+\si,x;t_{0})\\
&\quad\quad\geq u(t_{0}+\tilde{s}_{n}+\si,x-\zeta_{n}^{+}+dr_{0}(\zeta_{n}^{+}-\zeta_{n}^{-}))-u(t_{0}+\tilde{s}_{n}+\si,x-\zeta_{n}^{+})-(\zeta_{n}^{+}-\zeta_{n}^{-})e^{-\om\si}\\
&\quad\quad\geq -C_{*}dr_{0}(\zeta_{n}^{+}-\zeta_{n}^{-})-(\zeta_{n}^{+}-\zeta_{n}^{-})e^{-\om\si}\\
&\quad\quad\geq -C(dr_{0}+e^{-\om\si})(\zeta_{n}^{+}-\zeta_{n}^{-}).
\end{split}
\end{equation*}
Define
\begin{equation*}
\tilde{v}(t)=-C(dr_{0}+e^{-\om\si})(\zeta_{n}^{+}-\zeta_{n}^{-})e^{-\beta(t-(t_{0}+\tilde{s}_{n}+\si))}.
\end{equation*}
It is a space-independent solution of $\tilde{v}_{t}=\tilde{v}_{xx}-\beta\tilde{v}$ with initial data $\tilde{v}(t_{0}+\tilde{s}_{n}+\si)=-C(dr_{0}+e^{-\om\si})(\zeta_{n}^{+}-\zeta_{n}^{-})$. Since $\tilde{v}\leq0$, it satisfies $\tilde{v}_{t}\leq\tilde{v}_{xx}+a(t,x)\tilde{v}$, hence, the comparison principle implies that for any $(t,x)\in\tilde{\Om}_{R}^{-}$
\begin{equation}\label{estimate-om-r-behind-condition}
\begin{split}
&u(t,x-\zeta_{n}^{+}+dr_{0}(\zeta_{n}^{+}-\zeta_{n}^{-}))-u(t,x;t_{0})\\
&=v(t,x;t_{0})\geq\tilde{v}(t)=-C(dr_{0}+e^{-\om\si})(\zeta_{n}^{+}-\zeta_{n}^{-})e^{-\beta(t-(t_{0}+\tilde{s}_{n}+\si))}.
\end{split}
\end{equation}

\textbf{Case $(t,x)\in\tilde{\Om}_{R}^{+}$.} If $(t,x)\in\tilde{\Om}_{R}^{+}$, then $x-(\zeta_{1}^{+}+1)-\xi(t)\geq\frac{L_{0}}{2}$, and hence $u(t,x-(\zeta_{1}^{+}+1))\leq\frac{\theta}{2}$ by \eqref{condition-L}. We then obtain from the monotonicity of $u(t,x)$ in $x$ and the estimate $\zeta_{n}^{+}-dr_{0}(\zeta_{n}^{+}-\zeta_{n}^{-})\leq\zeta_{n}^{+}\leq\zeta_{1}^{+}+1$ that
\begin{equation*}
u(t,x-\zeta_{n}^{+}+dr_{0}(\zeta_{n}^{+}-\zeta_{n}^{-}))\leq u(t,x-(\zeta_{1}^{+}+1))\leq\frac{\theta}{2}.
\end{equation*}
Moreover, by the induction assumption
\begin{equation*}
\begin{split}
u(t,x;t_{0})&\leq u(t,x-\zeta^{+}_{n})+q_{n}e^{-\om(t-t_{0}-s_{n})}\Ga(x-\tilde{\xi}(t)-\zeta^{+}_{n})\\
&\leq u(t,x-(\zeta_{1}^{+}+1))+(\zeta_{1}^{+}-\zeta_{1}^{-}+2)e^{-\om\si}\leq\theta
\end{split}
\end{equation*}
provided $\si=\si(\al_{0},\zeta_{1}^{+}-\zeta_{1}^{-})$ is large so that $(\zeta_{1}^{+}-\zeta_{1}^{-}+2)e^{-\om\si}\leq\frac{\theta}{2}$. Thus, setting $v(t,x;t_{0})=u(t,x-\zeta_{n}^{+}+dr_{0}(\zeta_{n}^{+}-\zeta_{n}^{-}))-u(t,x;t_{0})$, we verify
\begin{equation*}
v_{t}=v_{xx}.
\end{equation*}
For $(t,x)\in\tilde{\Om}_{R}^{+}$, we easily check $x-\tilde{\xi}(t)-\zeta_{n}^{+}\geq L_{0}+1$, thus $\Ga(x-\tilde{\xi}(t)-\zeta_{n}^{+})=e^{-\al(x-\tilde{\xi}(t)-\zeta_{n}^{+}-L_{0})}$ by the definition of $\Ga$. Then, at the initial moment $t=t_{0}+\tilde{s}_{n}+\si$, we deduce from the induction assumption
\begin{equation*}
\begin{split}
&v(t_{0}+\tilde{s}_{n}+\si,x;t_{0})\\
&\quad\quad=u(t_{0}+\tilde{s}_{n}+\si,x-\zeta_{n}^{+}+dr_{0}(\zeta_{n}^{+}-\zeta_{n}^{-}))-u(t_{0}+\tilde{s}_{n}+\si,x;t_{0})\\
&\quad\quad\geq u(t_{0}+\tilde{s}_{n}+\si,x-\zeta_{n}^{+}+dr_{0}(\zeta_{n}^{+}-\zeta_{n}^{-}))-u(t_{0}+\tilde{s}_{n}+\si,x-\zeta_{n}^{+})\\
&\quad\quad\quad-(\zeta_{n}^{+}-\zeta_{n}^{-})e^{-\om\si}e^{-\al(x-\tilde{\xi}(t_{0}+\tilde{s}_{n}+\si)-\zeta_{n}^{+}-L_{0})}\\
&\quad\quad\geq-Cdr_{0}(\zeta_{n}^{+}-\zeta_{n}^{-})e^{-\al(x-\zeta_{n}^{+}-\tilde{\xi}(t_{0}+\tilde{s}_{n}+\si))}-(\zeta_{n}^{+}-\zeta_{n}^{-})e^{-\om\si}e^{-\al(x-\tilde{\xi}(t_{0}+\tilde{s}_{n}+\si)-\zeta_{n}^{+}-L_{0})}\\
&\quad\quad\geq-C(dr_{0}+e^{-\om\si})(\zeta_{n}^{+}-\zeta_{n}^{-})e^{-\al(x-\zeta_{n}^{+}-\tilde{\xi}(t_{0}+\tilde{s}_{n}+\si)-L_{0})},
\end{split}
\end{equation*}
where we used Proposition \ref{prop-transition-wave}$\rm(iv)$ in the second inequality. More precisely, we used the following estimate
\begin{equation*}
\begin{split}
&u(t_{0}+\tilde{s}_{n}+\si,x-\zeta_{n}^{+}+dr_{0}(\zeta_{n}^{+}-\zeta_{n}^{-}))-u(t_{0}+\tilde{s}_{n}+\si,x-\zeta_{n}^{+})\\
&\quad\quad=dr_{0}(\zeta_{n}^{+}-\zeta_{n}^{-})u_{x}(t_{0}+\tilde{s}_{n}+\si,x-\zeta_{n}^{+}+y)\quad (\text{where}\,\,y\in[0,dr_{0}(\zeta_{n}^{+}-\zeta_{n}^{-})])\\
&\quad\quad\geq-Cdr_{0}(\zeta_{n}^{+}-\zeta_{n}^{-})e^{-c_{0}(x-\zeta_{n}^{+}+y-\xi(t_{0}+\tilde{s}_{n}+\si))}\quad(\text{by Proposition \ref{prop-transition-wave}{\rm(iv)}})\\
&\quad\quad\geq-Cdr_{0}(\zeta_{n}^{+}-\zeta_{n}^{-})e^{-\al(x-\zeta_{n}^{+}-\xi(t_{0}+\tilde{s}_{n}+\si))}\quad(\text{since}\,\,\al\leq c_{0})\\
&\quad\quad=-Cdr_{0}(\zeta_{n}^{+}-\zeta_{n}^{-})e^{-\al(x-\zeta_{n}^{+}-\tilde{\xi}(t_{0}+\tilde{s}_{n}+\si))}e^{-\al(\tilde{\xi}(t_{0}+\tilde{s}_{n}+\si)-\xi(t_{0}+\tilde{s}_{n}+\si))}\\
&\quad\quad\geq-Cdr_{0}(\zeta_{n}^{+}-\zeta_{n}^{-})e^{-\al(x-\zeta_{n}^{+}-\tilde{\xi}(t_{0}+\tilde{s}_{n}+\si))}.
\end{split}
\end{equation*}
Moreover, estimate \eqref{estimate-om-r-condition} gives the nonnegativity of $v(t,x;t_{0})$ on the boundary of $\tilde{\Om}_{R}^{-}$. Let $\nu=\frac{\al c_{B}}{2}-\al^{2}$. Since
\begin{equation*}
\tilde{v}(t,x;t_{0})=-C(dr_{0}+e^{-\om\si})(\zeta_{n}^{+}-\zeta_{n}^{-})e^{-\al(x-\zeta_{n}^{+}-\tilde{\xi}(t)-L_{0})}e^{-\nu(t-(t_{0}+\tilde{s}_{n}+\si))}
\end{equation*}
solves $\tilde{v}_{t}\leq\tilde{v}_{xx}$, we conclude from the comparison principle that for $(t,x)\in\tilde{\Om}_{R}^{+}$
\begin{equation}\label{estimate-om-r-ahead-condition}
\begin{split}
&u(t,x-\zeta_{n}^{+}+dr_{0}(\zeta_{n}^{+}-\zeta_{n}^{-}))-u(t,x;t_{0})\\
&=v(t,x;t_{0})\\
&\geq\tilde{v}(t,x;t_{0})=-C(dr_{0}+e^{-\om\si})(\zeta_{n}^{+}-\zeta_{n}^{-})e^{-\al(x-\zeta_{n}^{+}-\tilde{\xi}(t)-L_{0})}e^{-\nu(t-(t_{0}+\tilde{s}_{n}+\si))}.
\end{split}
\end{equation}

So far, we have obtained the following estimate for $u(t,x;t_{0})$ for $t\in[t_{0}+\tilde{s}_{n}+\si,t_{0}+\tilde{s}_{n}+\si+\tau]$
\begin{equation}\label{estimate-auxiliary-reduced-gap}
u(t,x-\zeta^{-}_{n})-q_{n}e^{-\om(t-t_{0}-s_{n})}\Ga(x-\tilde{\xi}(t)-\zeta^{-}_{n})\leq u(t,x;t_{0})\leq\tilde{u}(t,x;t_{0}),
\end{equation}
where the first inequality is the induction assumption and
\begin{equation*}\label{estimate-auxiliary-reduced-gap-upper-bound}
\tilde{u}(t,x;t_{0})=
\begin{cases}
u(t,x-\zeta_{n}^{+}+dr_{0}(\zeta_{n}^{+}-\zeta_{n}^{-})),\quad\text{if}\,\,(t,x)\in\Om_{R},\\
u(t,x-\zeta_{n}^{+}+dr_{0}(\zeta_{n}^{+}-\zeta_{n}^{-}))\\
\quad+C(dr_{0}+e^{-\om\si})(\zeta_{n}^{+}-\zeta_{n}^{-})e^{-\beta(t-(t_{0}+\tilde{s}_{n}+\si))},\quad\text{if}\,\,(t,x)\in\Om_{R}^{-},\\
u(t,x-\zeta_{n}^{+}+dr_{0}(\zeta_{n}^{+}-\zeta_{n}^{-}))\\
\quad+C(dr_{0}+e^{-\om\si})(\zeta_{n}^{+}-\zeta_{n}^{-})e^{-\al(x-\zeta_{n}^{+}-\tilde{\xi}(t)-L_{0})}e^{-\nu(t-(t_{0}+\tilde{s}_{n}+\si))},\quad\text{if}\,\,(t,x)\in\Om_{R}^{+},
\end{cases}
\end{equation*}
is given by \eqref{estimate-om-r-condition}, \eqref{estimate-om-r-behind-condition} and \eqref{estimate-om-r-ahead-condition}. Note that from the induction assumption to \eqref{estimate-auxiliary-reduced-gap}, we have reduced the gap $\zeta_{n}^{+}-\zeta_{n}^{-}$ to $(1-dr_{0})(\zeta_{n}^{+}-\zeta_{n}^{-})$. We now construct $s_{n+1}$, $\zeta_{n+1}^{-}$, $\zeta_{n+1}^{+}$ and $q_{n+1}$, and show $\textbf{Claim}_{n+1}$.

Set
\begin{equation*}
\begin{split}
s_{n+1}&=\tilde{s}_{n}+\si+\tau,\\
\zeta_{n+\frac{1}{2}}^{-}&=\zeta_{n}^{-},\\
\zeta_{n+\frac{1}{2}}^{+}&=\zeta_{n}^{+}-dr_{0}(\zeta_{n}^{+}-\zeta_{n}^{-}),\\
q_{n+\frac{1}{2}}&=\Big[(\zeta_{n}^{+}-\zeta_{n}^{-})e^{-\om\si}+C(dr_{0}+e^{-\om\si})(\zeta_{n}^{+}-\zeta_{n}^{-})\Big]e^{-\om\tau},
\end{split}
\end{equation*}
It then follows from \eqref{estimate-auxiliary-reduced-gap} that at the moment $t=t_{0}+s_{n+1}$,
\begin{equation}\label{estimate-auxiliary-reduced-gap-1}
\begin{split}
&u(t_{0}+s_{n+1},x-\zeta^{-}_{n+\frac{1}{2}})-q_{n+\frac{1}{2}}\Ga(x-\tilde{\xi}(t_{0}+s_{n+1})-\zeta^{-}_{n+\frac{1}{2}})\\
&\leq u(t_{0}+s_{n+1},x;t_{0})\leq u(t_{0}+s_{n+1},x-\zeta_{n+\frac{1}{2}}^{+})+q_{n+\frac{1}{2}}\Ga(x-\tilde{\xi}(t_{0}+s_{n+1})-\zeta_{n}^{+}).
\end{split}
\end{equation}
Since $\Ga(x-\tilde{\xi}(t_{0}+s_{n+1})-\zeta_{n+\frac{1}{2}}^{+})$ is $\Ga(x-\tilde{\xi}(t_{0}+s_{n+1})-\zeta_{n}^{+})$ shifting to the left by $dr_{0}(\zeta_{n}^{+}-\zeta_{n}^{-1})<\frac{1}{2}$, we conclude from the definition of $\Ga$ that there is $C>0$ such that
\begin{equation*}
\Ga(x-\tilde{\xi}(t_{0}+s_{n+1})-\zeta_{n}^{+})\leq C\Ga(x-\tilde{\xi}(t_{0}+s_{n+1})-\zeta_{n+\frac{1}{2}}^{+}),\quad x\in\R.
\end{equation*}
Setting $\tilde{q}_{n+\frac{1}{2}}=Cq_{n+\frac{1}{2}}$, \eqref{estimate-auxiliary-reduced-gap-1} yields
\begin{equation}\label{estimate-auxiliary-reduced-gap-2}
\begin{split}
&u(t_{0}+s_{n+1},x-\zeta^{-}_{n+\frac{1}{2}})-\tilde{q}_{n+\frac{1}{2}}\Ga(x-\tilde{\xi}(t_{0}+s_{n+1})-\zeta^{-}_{n+\frac{1}{2}})\\
&\leq u(t_{0}+s_{n+1},x;t_{0})\leq u(t_{0}+s_{n+1},x-\zeta_{n+\frac{1}{2}}^{+})+\tilde{q}_{n+\frac{1}{2}}\Ga(x-\tilde{\xi}(t_{0}+s_{n+1})-\zeta_{n+\frac{1}{2}}^{+}).
\end{split}
\end{equation}
Also,
\begin{equation}\label{tilde-q-n-half}
\tilde{q}_{n+\frac{1}{2}}\leq C(dr_{0}+e^{-\om\si})(\zeta_{n}^{+}-\zeta_{n}^{-})e^{-\om\tau}\leq
\min\{\ep_{0},\zeta_{n+\frac{1}{2}}^{+}-\zeta_{n+\frac{1}{2}}^{-}\}
\end{equation}
provided $\tau=\tau(\al_{0},\zeta_{1}^{+}-\zeta_{1}^{-})$ is sufficient large, where $\ep_{0}$ is given by \eqref{constant-ep}. Using this estimate and \eqref{estimate-auxiliary-reduced-gap-2}, we can apply Theorem \ref{thm-aprior-estimate} as mentioned in Remark \ref{rem-thm-aprior-estimate} to conclude
\begin{equation}\label{estimate-auxiliary-reduced-gap-3}
\begin{split}
&u(t,x-\zeta^{-}_{n+1})-q_{n+1}e^{-\om(t-t_{0}-s_{n+1})}\Ga(x-\tilde{\xi}(t)-\zeta^{-}_{n+1})\\
&\leq u(t,x;t_{0})\leq u(t,x-\zeta_{n+1}^{+})+q_{n+1}e^{-\om(t-t_{0}-s_{n+1})}\Ga(x-\tilde{\xi}(t)-\zeta_{n+1}^{+})
\end{split}
\end{equation}
for $x\in\R$ and $t\geq t_{0}+s_{n+1}$, where
\begin{equation*}
\zeta_{n+1}^{-}=\zeta_{n+\frac{1}{2}}^{-}-\frac{M}{\om}\tilde{q}_{n+\frac{1}{2}},\quad \zeta_{n+1}^{+}=\zeta_{n+\frac{1}{2}}^{+}+\frac{M}{\om}\tilde{q}_{n+\frac{1}{2}}\quad\text{and}\quad q_{n+1}=\tilde{q}_{n+\frac{1}{2}}.
\end{equation*}
Moreover, by \eqref{tilde-q-n-half}, we have $q_{n+1}\leq\zeta_{n+1}^{+}-\zeta_{n+1}^{-}$. For $\zeta_{n+1}^{-}$ and $\zeta_{n+1}^{+}$, we have
\begin{equation}\label{right-shift-estimate-1}
\begin{split}
\zeta_{n+1}^{-}&=\zeta_{n}^{-}-C\Big[(\zeta_{n}^{+}-\zeta_{n}^{-})e^{-\om\si}+C(dr_{0}+e^{-\om\si})(\zeta_{n}^{+}-\zeta_{n}^{-})\Big]e^{-\om\tau}\\
&\geq\zeta_{n}^{-}-C(dr_{0}+e^{-\om\si})(\zeta_{n}^{+}-\zeta_{n}^{-})e^{-\om\tau}
\end{split}
\end{equation}
and
\begin{equation}\label{left-shift-estimate-1}
\begin{split}
\zeta_{n+1}^{+}&=\zeta_{n}^{+}-dr_{0}(\zeta_{n}^{+}-\zeta_{n}^{-})+\Big[(\zeta_{n}^{+}-\zeta_{n}^{-})e^{-\om\si}+C(dr_{0}+e^{-\om\si})(\zeta_{n}^{+}-\zeta_{n}^{-})\Big]e^{-\om\tau}\\
&\leq\zeta_{n}^{+}-(dr_{0}-C(dr_{0}+e^{-\om\si})e^{-\om\tau})(\zeta_{n}^{+}-\zeta_{n}^{-}).
\end{split}
\end{equation}
It follows that
\begin{equation}\label{gap-estimate-1}
\zeta_{n+1}^{+}-\zeta_{n+1}^{-}\leq(1-dr_{0}+C(dr_{0}+e^{-\om\si})e^{-\om\tau})(\zeta_{n}^{+}-\zeta_{n}^{-}).
\end{equation}
The estimates \eqref{estimate-auxiliary-reduced-gap-3}, \eqref{right-shift-estimate-1}, \eqref{left-shift-estimate-1} and \eqref{gap-estimate-1} are obtained provided \eqref{one-in-two-1} and \eqref{one-in-two-1-harnack} hold. If \eqref{one-in-two-2} and \eqref{one-in-two-2-harnack} hold, then we can also obtain \eqref{estimate-auxiliary-reduced-gap-3} with $\zeta_{n+1}^{-}$ and $\zeta_{n+1}^{+}$ satisfying
\begin{equation*}\label{right-shift-estimate-2}
\zeta_{n+1}^{-}\geq\zeta_{n}^{-}+(dr_{0}-C(dr_{0}+e^{-\om\si})e^{-\om\tau})(\zeta_{n}^{+}-\zeta_{n}^{-})
\end{equation*}
and
\begin{equation*}\label{left-shift-estimate-2}
\zeta_{n+1}^{+}\leq\zeta_{n}^{+}+C(dr_{0}+e^{-\om\si})e^{-\om\tau}(\zeta_{n}^{+}-\zeta_{n}^{-}),
\end{equation*}
respectively, and hence, $\zeta_{n+1}^{+}-\zeta_{n+1}^{-}$ satisfies \eqref{gap-estimate-1} as well.

Choosing $\si=\si(\al_{0},\zeta_{1}^{+}-\zeta_{1}^{-})$ and $\tau=\tau(\al_{0},\zeta_{1}^{+}-\zeta_{1}^{-})$ sufficiently large, we can write the above estimate in the following uniform form: for some $\de_{0}=\de_{0}(\al_{0},\zeta_{1}^{+}-\zeta_{1}^{-})>0$ sufficiently small, there holds
\begin{equation*}
\begin{split}
&\zeta_{n}^{-}-\de_{0}(\zeta_{n}^{+}-\zeta_{n}^{-})\leq\zeta_{n+1}^{-}\leq\zeta_{n+1}^{+}\leq\zeta_{n}^{+}+\de_{0}(\zeta_{n}^{+}-\zeta_{n}^{-}),\\
&\zeta_{n+1}^{+}-\zeta_{n+1}^{-}\leq(1-\de_{0})(\zeta_{n}^{+}-\zeta_{n}^{-}).
\end{split}
\end{equation*}
We then deduce from the induction assumption that $\zeta_{1}^{-}-1\leq\zeta_{n+1}^{-}\leq\zeta_{n+1}^{+}\leq\zeta_{1}^{+}+1$ and then, $\zeta_{n}^{+}-\zeta_{n}^{-}\leq C(1-\de_{0})^{n}$ for some $C>0$.

It remains to choose $s_{1}$. From the proof, we see that
\begin{equation*}
s_{n+1}=\tilde{s}_{n}+\si+\tau=s_{n}+\frac{\ln p}{\om}+\si+\tau,
\end{equation*}
where $p=p(\al_{0},\zeta_{1}^{+}-\zeta_{1}^{-})$, $\si=\si(\al_{0},\zeta_{1}^{+}-\zeta_{1}^{-})$ and $\tau=\tau(\al_{0},\zeta_{1}^{+}-\zeta_{1}^{-})$ are large constants. Thus, choosing $s_{1}=\frac{\ln p}{\om}+\si+\tau$, we have $s_{n}=ns_{1}$. This proves $\textbf{Claim}_{n}$ for $n\geq2$.

Finally, we set $T=Ns_{1}$ for some $N$ sufficiently large to complete the proof of $\rm(i)$.

$\rm(ii)$ We apply Theorem \ref{thm-aprior-estimate} and the iteration arguments for $\textbf{Claim}_{n}$, $n\in\N$ in the proof of $\rm(i)$ to the estimate
\begin{equation*}
\begin{split}
&u(t_{0}+T,x-\zeta^{-}_{T})-q_{T}\Ga(x-\tilde{\xi}(t_{0}+T)-\zeta^{-}_{T})\\
&\leq u(t_{0}+T,x;t_{0})\leq u(t_{0}+T,x-\zeta^{+}_{T})+q_{T}\Ga(x-\tilde{\xi}(t_{0}+T)-\zeta^{+}_{T}),
\end{split}
\end{equation*}
at the new initial moment $t_{0}+T$, which is the result of $\rm(i)$. The only difference between this and the arguments in the proof of $\rm(i)$ is that now the initial gap between the shifts $\zeta_{T}^{-}$ and $\zeta_{T}^{+}$ satisfy $0\leq\zeta_{T}^{-}-\zeta_{T}^{-}\leq 1$. As a result, various constants as in the proof of $\rm(i)$ now do not depend on $\zeta_{T}^{-}-\zeta_{T}^{-}$, and only depend on $\al_{0}$, and hence, we find the result.
\end{proof}


\section{Properties of Generalized Traveling Waves}\label{sec-property-gtw}

In this section, we study fundamental properties of arbitrary generalized traveling waves as defined in Definition \ref{def-generalized-critical}. Properties of particular interest are space monotonicity and exponential decay ahead of the interface as stated in Theorem \ref{thm-monotoncity-exponential-decay}. These two properties play an crucial role in the study of uniqueness of generalized traveling waves, which will be the objective of Section \ref{sec-uniqueness}. We always assume $\rm(H1)$-$\rm(H3)$ in this section.

From now on, we will always consider some fixed generalized traveling wave $v(t,x)$ of \eqref{main-eqn} as defined in Definition \ref{def-generalized-critical}. Let $\xi^{v}(t)$ be the interface location function of $v(t,x)$.

\subsection{Space Monotonicity}\label{subsection-space-mono}

In this subsection, we study the space monotonicity of $v(t,x)$ and prove Theorem \ref{thm-monotoncity-exponential-decay}$\rm(i)$. We first prove the following lemma.

\begin{lem}\label{lem-space-monotonicity}
There exists $h_{1}>0$ such that for any $h\geq h_{1}$ there holds $v(t,x)\leq v(t,x-h)$ for all $x\in\R$ and $t\in\R$.
\end{lem}
\begin{proof}
For $t\in\R$, define
\begin{equation*}
\begin{split}
\Om_{t}&=\bigg\{x\in\R\bigg|\frac{\theta}{2}\leq v(t,x)\leq\frac{1+\theta_{*}}{2}\bigg\},\\
\Om_{t}^{1}&=\bigg\{x\in\big[\inf\Om_{t},\sup\Om_{t}\big]\bigg|v(t,x)\leq\frac{1+\theta_{*}}{2}\bigg\},\\
\Om_{t}^{2}&=\{x\in\R|x\leq\inf\Om_{t}^{1}\}=\{x\in\R|x\leq\inf\Om_{t}\},\\
\Om_{t}^{3}&=\{x\in\R|x\geq\sup\Om_{t}^{1}\}=\{x\in\R|x\geq\sup\Om_{t}\},\\
\Om_{t}^{4}&=\bigg\{x\in\big[\inf\Om_{t},\sup\Om_{t}\big]\bigg|v(t,x)\geq\frac{1+\theta_{*}}{2}\bigg\}
\end{split}
\end{equation*}
and set $\Om_{i}=\cup_{t\in\R}(\{t\}\times\Om_{t}^{i})$ for $i=1,2,3,4$. Notice $\R^{2}=\cup_{i=1}^{4}\Om_{i}$.

We first show that there exists $h_{1}>0$ such that for any $h\geq h_{1}$ there holds $v(t,x)\leq v(t,x-h)$ for $(t,x)\in\Om_{1}$. We see from the definition of generalized traveling wave that
\begin{equation*}
h_{1}:=\sup_{t\in\R}(\sup\Om_{t}^{1}-\inf\Om_{t}^{1})<\infty.
\end{equation*}
Since $v(t,x)\geq\frac{1+\theta_{*}}{2}$ for all $x\leq\inf\Om_{t}^{1}$ and $t\in\R$, we deduce for $(t,x)\in\Om_{1}$ and $h\geq h_{1}$ that $x-h\leq\inf\Om_{t}^{1}$, and hence, $v(t,x-h)\geq\frac{1+\theta_{*}}{2}\geq v(t,x)$.

Next, we show that for any $h\geq h_{1}$ there holds $v(t,x)\leq v(t,x-h)$ for $(t,x)\in\Om_{2}$. Let $h\geq h_{1}$ and set $\phi(t,x)=v(t,x-h)-v(t,x)$ for $(t,x)\in\Om_{2}$. We see that $\phi(t,x)$ satisfies
\begin{equation*}
\begin{cases}
\phi_{t}=\phi_{xx}+a(t,x)\phi,\quad (t,x)\in\Om_{2},\\
\phi(t,x)\geq0,\quad (t,x)\in\partial\Om_{2}:=\cup_{t\in\R}\{(t,\inf\Om_{t}^{2})\}\subset\Om_{1},
\end{cases}
\end{equation*}
where $a(t,x)=\frac{f(t,v(t,x-h))-f(t,v(t,x))}{v(t,x-h)-v(t,x)}$. Since $v(t,x)\geq\frac{1+\theta_{*}}{2}$ for $(t,x)\in\Om_{2}$ and $x-h\in\Om_{t}^{2}$ if $x\in\Om_{t}^{2}$, we have $v(t,x-h)\geq\frac{1+\theta_{*}}{2}$ for $(t,x)\in\Om_{2}$. It then follows from $\rm(H3)$ that $a(t,x)\leq-\beta$ for $(t,x)\in\Om_{2}$. For contradiction, let us assume
\begin{equation*}
r_{2}:=\inf_{(t,x)\in\Om_{2}}\phi(t,x)<0.
\end{equation*}
Then, we can find a sequence $\{(t_{n},x_{n})\}_{n\in\N}\subset\Om_{2}$ such that $\phi(t_{n},x_{n})\leq r_{2}(1-\frac{1}{2n})$ for all $n\in\N$. Note that $\sup_{n\in\N}(\inf\Om_{t_{n}}^{2}-x_{n})\leq d_{2}$ for some $d_{2}>0$, otherwise $\phi(t_{n},x_{n})\ra0$ as $n\ra\infty$ by the definition of generalized traveling wave. Moreover, since $\phi\geq0$ on $\partial\Om_{2}$ and $\sup_{n\in\N}\phi(t_{n},x_{n})\leq\frac{r_{2}}{2}$, regularity of $v(t,x)$ implies that $x_{n}$ stays uniformly away from $\inf\Om_{t_{n}}^{2}$, that is, $\inf_{n\in\N}(\inf\Om_{t_{n}}^{2}-x_{n})\geq \tilde{d}_{2}$ for some $\tilde{d}_{2}>0$. Thus, by a prior estimates for parabolic equations, say, $\sup_{(t,x)\in\R^{2}}|v_{t}(t,x)|<\infty$ and $\sup_{(t,x)\in\R^{2}}|v_{x}(t,x)|<\infty$, we can find some small $\hat{d}_{2}>0$ such that
\begin{equation*}
\phi(t,x)\leq\frac{r_{2}}{4}\quad\text{for}\quad (t,x)\in B_{n}:=\Big\{(t,x)\in\R^{2}\Big|(t-t_{n})^{2}+(x-x_{n})^{2}\leq\hat{d}_{2}^{2}\Big\}\subset\Om_{2}
\end{equation*}
for all $n\in\N$. Now, for $n\in\N$, define $\phi_{n}(t,x)=\phi(t+t_{n},x+x_{n})$ for $(t,x)\in\R^{2}$, which satisfies $\phi_{n}\leq\frac{r_{2}}{4}$ on $B_{0}:=\{(t,x)\in\R|t^{2}+x^{2}\leq\hat{d}_{2}^{2}\}$ and $(\phi_{n})_{t}-(\phi_{n})_{xx}\geq-\frac{\beta r_{2}}{4}$ for $(t,x)\in B_{0}$. A priori estimates for parabolic equations then ensure the existence of some subsequence, still denoted by $n$, such that $\{\phi_{n}\}_{n\in\N}$ converges to some $\tilde{\phi}$ uniformly in $B_{0}$. It then follows that
\begin{equation*}
\tilde{\phi}_{t}-\tilde{\phi}_{xx}\geq-\frac{\beta r_{2}}{4}>0\quad\text{for}\quad (t,x)\in B_{0}.
\end{equation*}
Moreover, $\tilde{\phi}\geq r_{2}$ on $B_{0}$ and $\tilde{\phi}(0,0)=\lim_{n\ra\infty}\phi(t_{n},x_{n})=r_{2}$. That is, on $B_{0}$, $\tilde{\phi}$ attains its minimum at $(0,0)$, which is an interior point of $B_{0}$. It's a contradiction by maximum principle. Hence, for any $h\geq h_{1}$, we have $v(t,x)\leq v(t,x-h)$ for $(t,x)\in\Om_{2}$.

We now show that for any $h\geq h_{1}$ there holds $v(t,x)\leq v(t,x-h)$ for all $(t,x)\in\Om_{3}$. Let $h_{2}>0$ and set $\phi(t,x)=v(t,x-h)-v(t,x)$ for $(t,x)\in\Om_{3}$. Since $v(t,x)\leq\frac{\theta}{2}$ for $(t,x)\in\Om_{3}$, $f(t,v(t,x))=0$ for $(t,x)\in\Om_{3}$. We then verify
\begin{equation*}
\begin{cases}
\phi_{t}-\phi_{xx}=f(t,v(t,x-h))\geq0,\quad (t,x)\in\Om_{3},\\
\phi(t,x)\geq0,\quad (t,x)\in\partial\Om_{3}:=\cup_{t\in\R}\{(t,\inf\Om_{t}^{3})\}\subset\Om_{1}.
\end{cases}
\end{equation*}
For contradiction, suppose
\begin{equation*}
r_{3}:=\inf_{(t,x)\in\Om_{3}}\phi(t,x)<0.
\end{equation*}
Let $\{(t_{n},x_{n})\}_{n\in\N}\subset\Om_{3}$ be such that $\phi(t_{n},x_{n})\leq r_{3}(1-\frac{1}{2n})$ for all $n\in\N$. For $n\in\N$, set $d_{n}=\text{\rm dist}((t_{n},x_{n}),\Om_{1}\cap(\{t\leq t_{n}\}\times\R))$. As above, there are $\tilde{d}_{3}>\hat{d}_{3}>0$ such that $d_{n}\in[\hat{d}_{3},\tilde{d}_{3}]$ for $n\in\N$. For $n\in\N$, we have
\begin{equation*}
\phi_{t}-\phi_{xx}\geq0\quad\text{for}\quad (t,x)\in B_{n},
\end{equation*}
where
\begin{equation*}
B_{n}=\Big\{(t,x)\in\R^{2}\Big|(t-t_{n})^{2}+(x-x_{n})^{2}\leq d_{n}^{2}\,\,\text{and}\,\,t\leq t_{n}\Big\}\subset\Om_{3}.
\end{equation*}
Note that we can find some subsequence, still denoted by $n$, such that $d_{n}\ra\bar{d}_{3}\in[\hat{d}_{3},\tilde{d}_{3}]$ as $n\ra\infty$ and $\{\phi_{n}:=\phi(\cdot+t_{n},\cdot+x_{n})\}_{n\in\N}$ converges to some $\bar{\phi}$ uniformly in $B_{0}:=\{(t,x)\in\R^{2}|t^{2}+x^{2}\leq\bar{d}_{3}^{2}\,\,\text{and}\,\,t\leq0\}$. Moreover,
\begin{equation*}
\bar{\phi}_{t}-\bar{\phi}_{xx}\geq0\quad\text{for}\quad (t,x)\in B_{0}
\end{equation*}
and, on $B_{0}$, $\bar{\phi}$ attains its minimum $r_{3}$ at $(0,0)$. Thus, maximum principle implies $\bar{\phi}\equiv r_{3}$ on $B_{0}$. However, for any $n\in\N$, $\partial B_{n}\cap\Om_{1}\neq\emptyset$, and hence, there's $(\tilde{t}_{n},\tilde{x}_{n})\in\partial B_{n}\cap\Om_{1}$ such that $\phi(\tilde{t}_{n},\tilde{x}_{n})\geq0$. As a result, $\bar{\phi}\geq0$ at some point on $\partial B_{0}$. This is a contradiction. Hence,  for any $h\geq h_{1}$, we have $v(t,x)\leq v(t,x-h)$ for $(t,x)\in\Om_{3}$

Finally, we show that for any $h\geq h_{1}$ there holds $v(t,x)\leq v(t,x-h)$ for all $(t,x)\in\Om_{4}$. Let $h\geq h_{1}$. Note by the definition of $h_{1}$, we actually have $v(t,x-h)\geq\frac{1+\theta_{*}}{2}$ for $x\leq\sup\Om_{t}^{1}$ and $t\in\R$. In particular, $v(t,x-h)\geq\frac{1+\theta_{*}}{2}$ for $(t,x)\in\Om_{4}$. Thus, we are in a situation similar to the case of $\Om_{2}$, and we can argue similarly to obtain the result.

In conclusion, for any $h\geq h_{1}$ there holds $v(t,x)\leq v(t,x-h)$ for all $x\in\R$ and $t\in\R$. This completes the proof.
\end{proof}

We now prove  Theorem \ref{thm-monotoncity-exponential-decay}$\rm(i)$.

\begin{proof}[Proof of Theorem \ref{thm-monotoncity-exponential-decay}(i)]
By the classical results of Angenent (see \cite[Theorem A and Theorem B]{Ang88}), the proposition follows if we can show that for any $h>0$ there holds $v(t,x)\leq v(t,x-h)$ for all $x\in\R$ and $t\in\R$. For this purpose, define
\begin{equation*}
h_{0}=\inf\big\{h_{1}>0\big|\text{for any}\,\,h\geq h_{1}\,\,\text{there holds}\,\,v(t,x)\leq v(t,x-h)\,\,\text{for all}\,\,(t,x)\in\R^{2}\big\}
\end{equation*}
and suppose $h_{0}>0$ for contradiction. Clearly, $v\not\equiv v(\cdot,\cdot-h_{0})$. Thus, by maximum principle, there holds $v(t,x)<v(t,x-h_{0})$ for all $x\in\R$ and $t\in\R$. For $t\in\R$, let $\Om_{t}=\big\{x\in\R\big|\frac{\theta}{2}\leq v(t,x)\leq\frac{1+\theta_{*}}{2}\big\}$ as in the proof of Lemma \ref{lem-space-monotonicity}, and set
\begin{equation*}
\tilde{\Om}=\bigcup_{t\in\R}\Big(\{t\}\times\big[\inf\Om_{t},\sup\Om_{t}\big]\Big).
\end{equation*}
Setting $\phi(t,x)=v(t,x-h_{0})-v(t,x)>0$ for $x\in\R$ and $t\in\R$, we claim that
\begin{equation}\label{a-claim-space-mono}
\inf_{(t,x)\in\tilde{\Om}}\phi(t,x)>0.
\end{equation}

Suppose \eqref{a-claim-space-mono} is false. We can find a sequence $\{(t_{n},x_{n})\}_{n\in\N}\subset\tilde{\Om}$ such that $\phi(t_{n},x_{n})\ra0$ as $n\ra\infty$. For $n\in\N$, define for $x\in\R$ and $t\in\R$
\begin{equation*}
\begin{split}
v_{n}(t,x)&=v(t+t_{n},x+x_{n}),\\
\phi_{n}(t,x)&=\phi(t+t_{n},x+x_{n})=v_{n}(t,x-h_{0})-v_{n}(t,x).
\end{split}
\end{equation*}
We see that there is some subsequence, still denoted by $n$, and a function $\tilde{v}$ such that $v_{n}$ and $v_{n}(\cdot,\cdot-h_{0})$ converge locally uniformly to $\tilde{v}$ and $\tilde{v}(\cdot,\cdot-h_{0})$, respectively. Of course, $\phi_{n}$ converges locally uniformly to $\tilde{\phi}:=\tilde{v}(\cdot,\cdot-h_{0})-\tilde{v}$. Moreover, as $\phi_{n}(0,0)\ra0$ as $n\ra\infty$ and $\phi_{n}\geq0$ satisfying $(\phi_{n})_{t}=(\phi_{n})_{xx}+a_{n}(t,x)\phi_{n}$ with $a_{n}(t,x)=\frac{f(t+t_{0},v_{n}(t,x-h_{0}))-f(t+t_{0},v_{n}(t,x))}{v_{n}(t,x-h_{0})-v_{n}(t,x)}$ bounded, we conclude from the Harnack inequality that $\tilde{\phi}\equiv0$. This yields
\begin{equation}\label{space-periodicity}
\tilde{v}(\cdot,\cdot-h_{0})=\tilde{v}.
\end{equation}
However, since $\{(t_{n},x_{n})\}_{n\in\N}\subset\tilde{\Om}$, the uniform-in-time limits in the defintion of generalized traveling waves implies the uniform-in-$n$ limits
\begin{equation*}
\lim_{x\ra-\infty}v_{n}(0,x)=1\quad\text{and}\quad\lim_{x\ra\infty}v_{n}(0,x)=0.
\end{equation*}
This yields $\lim_{x\ra-\infty}\tilde{v}(0,x)=1$ and $\lim_{x\ra\infty}\tilde{v}(0,x)=0$, which contradicts \eqref{space-periodicity}. Hence, \eqref{a-claim-space-mono} holds.

From \eqref{a-claim-space-mono}, we are able to find some $\de_{0}>0$ such that if $\de\in(0,\de_{0}]$ then
\begin{equation}\label{shifted-infimum}
\inf_{(t,x)\in\tilde{\Om}}(v(t,x-(h_{0}-\de))-v(t,x))\geq0.
\end{equation}
Indeed, due to the uniform boundedness of $v_{x}$, i.e., $\sup_{(t,x)\in\R^{2}}|v_{x}(t,x)|<\infty$, there exists some $\de_{0}\in(0,h_{0})$ such that for any $\de\in(0,\de_{0}]$ there holds
\begin{equation*}
v(t,x+\de-h_{0})-v(t,x-h_{0})\in\Big[-\frac{r_{0}}{2},\frac{r_{0}}{2}\Big],\quad (t,x)\in\tilde{\Om},
\end{equation*}
where $r_{0}=\inf_{(t,x)\in\tilde{\Om}}(v(t,x-h_{0})-v(t,x))>0$ by \eqref{a-claim-space-mono}. It then follows that for any $\de\in(0,\de_{0}]$
\begin{equation*}
v(t,x+\de-h_{0})-v(t,x)=v(t,x+\de-h_{0})-v(t,x-h_{0})+v(t,x-h_{0})-v(t,x)\geq\frac{r_{0}}{2}
\end{equation*}
for all $(t,x)\in\tilde{\Om}$. This establishes \eqref{shifted-infimum}.

Now, using \eqref{shifted-infimum}, we can repeat the arguments for $\Om_{2}$ and $\Om_{3}$ as in the proof of Lemma \ref{lem-space-monotonicity} to conclude
\begin{equation*}
\inf_{(t,x)\in\Om_{2}}(v(t,x-(h_{0}-\de))-v(t,x))\geq0\quad\text{and}\quad\inf_{(t,x)\in\Om_{3}}(v(t,x-(h_{0}-\de))-v(t,x))\geq0
\end{equation*}
for any $\de\in(0,\de_{0}]$. This together with \eqref{a-claim-space-mono} and \eqref{shifted-infimum} implies that for any $h\geq h_{0}-\de_{0}$ there holds
$v(t,x)\leq v(t,x-h)$ for $x\in\R$ and $t\in\R$, which contradicts the minimality of $h_{0}$. Hence, $h_{0}=0$ and the result follows.
\end{proof}

\subsection{Exponential Decay Ahead of Interface}\label{subsec-exponential-decay}

Consider a generalized traveling wave $v(t,x)$ of \eqref{main-eqn} with interface location function $\xi^{v}(t)$. In this subsection, we study the exponential decay of $v(t,x)$ ahead of the interface and prove Theorem \ref{thm-monotoncity-exponential-decay}$\rm(ii)$. With Theorem \ref{thm-monotoncity-exponential-decay}$\rm(i)$, we understand that $v(t,x)$ is strictly decreasing in $x$ for any $t\in\R$. We first prove two lemmas.

For $\la\in(0,1)$, let $\xi^{v}_{\la}:\R\ra\R$ be such that $v(t,\xi^{v}_{\la}(t))=\la$ for all $t\in\R$. By  Theorem \ref{thm-monotoncity-exponential-decay}$\rm(i)$, $\xi^{v}_{\la}$ is well-defined and continuously differentiable. Define $\psi^{v}_{\la}:\R\times\R\ra(0,1)$ by setting $\psi^{v}_{\la}(t,x)=v(t,x+\xi^{v}_{\la}(t))$. We show

\begin{lem}\label{lem-property-xi-v}
For each $\la\in(0,1)$, there hold
\begin{itemize}
\item[\rm(i)] $\psi^{v}_{\la}(t,x)$ is a profile of $v(t,x)$, that is,
\begin{equation*}
\lim_{x\ra-\infty}\psi^{v}_{\la}(t,x)=1,\,\,\lim_{x\ra\infty}\psi^{v}_{\la}(t,x)=0\,\,\text{uniformly in}\,\,t\in\R;
\end{equation*}

\item[\rm(ii)] there exists $L_{\la}>0$ such that $|\xi^{v}(t)-\xi^{v}_{\la}(t)|\leq L_{\la}$ for all $t\in\R$.
\end{itemize}
\end{lem}
\begin{proof}
$\rm(i)$ From the definition of the generalized traveling wave and the space monotonicity, for any $0<\la_{1}<\la_{2}<1$, there exists $L_{\la_{1},\la_{2}}>0$ such that $\xi^{v}_{\la_{1}}(t)-\xi^{v}_{\la_{2}}(t)\leq L_{\la_{1},\la_{2}}$ for all $t\in\R$.

Fix any $\la_{0}\in(0,1)$. Then, for any $\la>\la_{0}$, we have for all $x\leq-L_{\la_{0},\la}$
\begin{equation*}
\psi^{v}_{\la_{0}}(t,x)=v(t,x+\xi^{v}_{\la_{0}}(t))\geq v(t,\xi^{v}_{\la_{0}}(t)-L_{\la_{0},\la})\geq v(t,\xi_{\la}(t))=\la.
\end{equation*}
This shows the uniform-in-time limit $\lim_{x\ra-\infty}\psi_{\la_{0}}(t,x)=1$. Similarly, we have the limit $\lim_{x\ra\infty}\psi_{\la_{0}}(t,x)=0$ uniformly in $t\in\R$. This proves $\rm(i)$.

$\rm(ii)$ Since $\psi^{v}(t,x-\xi^{v}(t))=v(t,x)=\psi^{v}_{\la}(t,x-\xi_{\la}^{v}(t))$, we have
\begin{equation}\label{profile-equality}
\psi^{v}(t,x)=\psi^{v}_{\la}(t,x-\xi_{\la}^{v}(t)+\xi^{v}(t)).
\end{equation}
Now, suppose there exists $\{t_{n}\}_{n\in\N}$ with $|t_{n}|\ra\infty$ as $n\ra\infty$ such that $|\xi^{v}(t_{n})-\xi^{v}_{\la}(t_{n})|\ra\infty$ as $n\ra\infty$. Then, there must be a subsequence, still denoted by $\{t_{n}\}_{n\in\N}$, such that either $\xi^{v}(t_{n})-\xi_{\la}^{v}(t_{n})\ra\infty$ as $n\ra\infty$ or $\xi^{v}(t_{n})-\xi_{\la}^{v}(t_{n})\ra-\infty$ as $n\ra\infty$. Suppose the former is true (the case the later holds can be treated similarly), then setting $t=t_{n}$ and $x_{n}=\frac{\xi^{v}(t_{n})-\xi_{\la}^{v}(t_{n})}{2}$ in \eqref{profile-equality}, we have $\psi^{v}(t_{n},x_{n})=\psi_{\la}^{v}(t_{n},-x_{n})$. We then conclude from uniform-in-time limits of $\psi^{v}(t,x)$ and $\psi_{\la}^{v}(t,x)$ as $x\ra\pm\infty$ that
\begin{equation*}
0=\lim_{n\ra\infty}\psi^{v}(t_{n},x_{n})=\lim_{n\ra\infty}\psi_{\la}^{v}(t_{n},-x_{n})=1.
\end{equation*}
It's a contradiction. Hence, $\xi^{v}(t)-\xi^{v}_{\la}(t)$ remains bounded as $t$ varies in $\R$.
\end{proof}

We next show the rightward propagation nature of $\xi^{v}(t)$, and hence, of $\xi^{v}_{\la}(t)$ for all $\la\in(0,1)$.

\begin{lem}\label{lem-interface-propagation-improved-general}
There exist constants $C^{v}_{1}>0$, $C^{v}_{2}>0$ and $d^{v}>0$, and a twice continuously differentiable function $\tilde{\xi}^{v}:\R\ra\R$ satisfying
\begin{equation*}
\begin{split}
\frac{c_{B}}{2}\leq\dot{\tilde{\xi}}^{v}(t)\leq C^{v}_{1},\quad t\in\R,\\
|\ddot{\tilde{\xi}}^{v}(t)|\leq C^{v}_{2},\quad t\in\R
\end{split}
\end{equation*}
such that
\begin{equation*}
|\tilde{\xi}^{v}(t)-\xi^{v}(t)|\leq d^{v}_{},\quad t\in\R,
\end{equation*}
where $c_{B}>0$ the speed of traveling waves of \eqref{eqn-bi-1}. In particular, for any $\la\in(0,1)$, there exists $L_{\la}>0$ such that $|\tilde{\xi}^{v}(t)-\xi^{v}_{\la}(t)|\leq L_{\la}$ for all $t\in\R$.
\end{lem}
\begin{proof}
Since we have no idea about the value $v(t,\xi^{v}(t))$ and do not know whether $\xi^{v}(t)$ is continuous, instead of modifying $\xi^{v}(t)$, we modify $\xi_{\la_{0}}^{v}(t)$ for some $\la_{0}\in(\theta,1)$ as we did in Proposition \ref{prop-interface-rightward-propagation-improved} for $\xi(t)$. We sketch the proof within four steps.

\paragraph{\textbf{Step 1}} There are $t^{v}_{B}>0$ and $t^{v}_{I}>0$ such that
\begin{equation}\label{propagate-estimate-generalized}
\frac{3c_{B}}{4}(t-t_{0}-t^{v}_{B})\leq \xi_{\la_{0}}^{v}(t)-\xi_{\la_{0}}^{v}(t_{0})\leq \frac{5c_{I}}{4}(t-t_{0}-t^{v}_{I}),\quad t\geq t_{0}.
\end{equation}
The proof follows from the arguments as in the proof of Lemma \ref{lem-interface-rightward-propagation}. The only difference is that we need to use a different $\psi_{*}$, which in this case can be taken as a uniformly continuous and nonincreasing function satisfying
\begin{equation*}
\psi_{*}(x)=
\begin{cases}
\la_{0},&\quad x\leq x_{0},\\
0,&\quad x\geq0.
\end{cases}
\end{equation*}
for some fixed $x_{0}<0$. We then have $\psi_{*}\leq u(t_{0},\cdot+\xi_{\la_{0}}(t_{0}))$ for all $t_{0}\in\R$. The rest of the proof proceeds as in the proof of Lemma \ref{lem-interface-rightward-propagation}.

\paragraph{\textbf{Step 2}} Let $C_{0}>\frac{5}{4}c_{I}t^{v}_{I}$. Let $t_{0}\in\R$. There are a sequence $\{T_{n-1}(t_{0})\}_{n\in\N}$ with $T_0(t_{0})=t_0$ and functions
\begin{equation*}
\eta(t;T_{n-1}(t_{0}))=\xi_{\la_{0}}^{v}(T_{n-1}(t_{0}))+C_{0}+\frac{c_{B}}{2}(t-T_{n-1}(t_{0})),\,\, t\in[T_{n-1}(t_{0}),T_{n}(t_{0})]
\end{equation*}
for $n\in\N$ such that the following hold: for any $n\in\N$
\begin{itemize}
\item[\rm(i)] $T_{n}(t_{0})-T_{n-1}(t_{0})\in[T_{\min},T_{\max}]$, where $0<T_{\min}<T_{\max}<\infty$ depend only on $c_{B}$, $t_{B}^{v}$, $c_{I}$ and $t_{I}^{v}$;

\item[\rm(ii)] $\xi_{\la_{0}}^{v}(t)<\eta(t;T_{n-1}(t_{0}))\,\,\text{for}\,\,t\in[T_{n-1}(t_{0}),T_{n}(t_{0}))$\\
 and $\xi_{\la_{0}}^{v}(T_{n}(t_{0}))=\eta(T_{n}(t_{0});T_{n-1}(t_{0}))$;

\item[\rm(iii)] $0\leq\eta(t;T_{n-1}(t_{0}))-\xi_{\la_{0}}^{v}(t)\leq C_{0}+\frac{3}{4}c_{B}t^{v}_{B}$ for $t\in[T_{n-1}(t_{0}),T_{n}(t_{0})]$.
\end{itemize}
The proof follows in the same way as that of Lemma \ref{lem-interface-rightward-propagation-improved}, since we only need \eqref{propagate-estimate-generalized} and the continuity of $\xi^{v}_{\la_{0}}(t)$ as remarked after the proof of Lemma \ref{lem-interface-rightward-propagation}. The lower bound and the upper bound in $\rm(i)$ follow from and \eqref{propagate-estimate-generalized}.

\textbf{Step 3.} Let $t_{0}\in\R$. There exist constants $C_{1}^{v}>0$, $C_{2}^{v}>0$ and $d^{v}>0$, and a twice continuously differentiable function $\tilde{\xi}^{v}(\cdot;t_{0}):[t_{0},\infty)\ra\R$ satisfying for $t\geq t_{0}$
\begin{equation*}
\frac{c_{B}}{2}\leq\dot{\tilde{\xi}}^{v}(t;t_{0})\leq C_{1}^{v},\quad |\ddot{\tilde{\xi}}^{v}(t;t_{0})|\leq C_{2}^{v}\quad\text{and}\quad0\leq\tilde{\xi}^{v}(t;t_{0})-\xi^{v}_{\la_{0}}(t)\leq d^{v}.
\end{equation*}
Moreover, $\{\dot{\tilde{\xi}}^{v}(\cdot;t_{0})\}_{t_{0}\leq0}$ and $\{\ddot{\tilde{\xi}}^{v}(\cdot;t_{0})\}_{t_{0}\leq0}$ are uniformly bounded and uniformly Lipschitz continuous.

To see this, we first define $\xi^{v}(\cdot;t_{0}):[t_{0},\infty)\ra\R$ by setting
\begin{equation*}
\xi^{v}(t;t_{0})=\eta(t;T_{n-1}(t_{0})),\quad t\in[T_{n-1}(t_{0}),T_{n}(t_{0})),\,\,n\in\N.
\end{equation*}
We then modify $\xi^{v}(t;t_{0})$ as in the proof of Lemma \ref{lem-interface-rightward-propagation-improved} to get $\tilde{\xi}^{v}(t;t_{0})$. For the modification, it concerns some function $\de(t)$ as in the proof of Lemma \ref{lem-interface-rightward-propagation-improved}, which is only required to be continuously differentiable there, but clearly we can make it twice continuously differentiable with uniformly bounded first and second derivatives as remarked after the proof of Lemma \ref{lem-interface-rightward-propagation-improved}.

\textbf{Step 4.} By \textbf{Step 3}, Arzel\`{a}-Ascoli theorem and the diagonal argument, we can argue as in the proof of Proposition \ref{prop-interface-rightward-propagation-improved} to conclude that $\{\tilde{\xi}^{v}(\cdot;t_{0})\}_{t_{0}\leq0}$ converges locally uniformly to some twice continuously differentiable function $\tilde{\xi}^{v}:\R\ra\R$ satisfying $\frac{c_{B}}{2}\leq\dot{\tilde{\xi}}^{v}(t)\leq C_{1}^{v}$, $|\ddot{\tilde{\xi}}^{v}(t)|\leq C_{2}^{v}$ and $0\leq\tilde{\xi}^{v}(t)-\xi^{v}_{\la_{0}}(t)\leq d^{v}$ for $t\in\R$. The lemma then follows from $\sup_{t\in\R}|\xi^{v}(t)-\xi^{v}_{\la_{0}}(t)|<\infty$ due to Lemma \ref{lem-property-xi-v}.
\end{proof}

We now prove Theorem \ref{thm-monotoncity-exponential-decay}$\rm(ii)$.

\begin{proof}[Proof of Theorem \ref{thm-monotoncity-exponential-decay}(ii)]
By Lemma \ref{lem-property-xi-v}$\rm(ii)$ and Lemma \ref{lem-interface-propagation-improved-general}, there holds $L_{\theta}:=\sup_{t\in\R}|\tilde{\xi}^{v}(t)-\xi_{\theta}^{v}(t)|<\infty$. For $t\in\R$, we define
\begin{equation*}
\hat{\xi}^{v}(t)=\tilde{\xi}^{v}(t)+L_{\theta}.
\end{equation*}
Then, $\hat{\xi}^{v}$ satisfies all the properties for $\tilde{\xi}^{v}$ as in Lemma \ref{lem-interface-propagation-improved-general}. In particular, it satisfies the properties as in the statement of the theorem. It remains to show that $v(t,x)$ is exponential decay ahead of $\hat{\xi}^{v}(t)$ uniformly in $t\in\R$.

Set $\hat{v}(t,x)=v(t,x+\hat{\xi}^{v}(t))$ for $x\geq0$ and $t\in\R$. Since $\hat{\xi}^{v}(t)\geq\xi_{\theta}^{v}(t)$ by the definition, we obtain from monotonicity that $v(t,x+\hat{\xi}^{v}(t))\leq\theta$, and hence $f(t,v(t,x+\hat{\xi}^{v}(t)))=0$ and $\hat{v}(t,x)=\frac{1}{\theta}v(t,x+\hat{\xi}^{v}(t))\leq1$ for all $x\geq0$ and $t\in\R$. We then readily check that
$\hat{v}(t,x)$ satisfies
\begin{equation*}
\begin{cases}
\hat{v}_{t}=\hat{v}_{xx}+\dot{\hat{\xi}}\hat{v}_{x},\quad x\geq0,\,\,t\in\R,\\
\hat{v}(t,0)\leq1,\quad t\in\R,\\
\lim_{x\ra\infty}\hat{v}(t,x)=0\,\,\text{uniformly in}\,\,t\in\R.
\end{cases}
\end{equation*}
Now, define
\begin{equation*}
\phi(t,x)=e^{-\frac{c_{B}}{2}x}-\hat{v}(t,x),\quad x\geq0,\,\,t\in\R.
\end{equation*}
Since $\dot{\hat{\xi}}(t)=\dot{\tilde{\xi}}(t)\geq\frac{c_{B}}{2}$ for all $t\in\R$ by Lemma \ref{lem-interface-propagation-improved-general}, we see that $\phi(t,x)$ satisfies
\begin{equation*}
\begin{cases}
\phi_{t}\geq\phi_{xx}+\dot{\hat{\xi}}\phi_{x},\quad x\geq0,\,\,t\in\R,\\
\phi(t,0)\geq0,\quad t\in\R,\\
\lim_{x\ra\infty}\phi(t,x)=0\,\,\text{uniformly in}\,\,t\in\R.
\end{cases}
\end{equation*}

We claim that $\phi(t,x)\geq0$ for $x\geq0$ and $t\in\R$. For contradiction, suppose this is not the case, that is,
\begin{equation*}
r_{0}:=\inf_{x\geq0,t\in\R}\phi(t,x)<0.
\end{equation*}
Let $\{(t_{n},x_{n})\}_{n\in\N}\subset\R\times[0,\infty)$ be such that $\phi(t_{n},x_{n})\leq r_{0}(1-\frac{1}{2n})$ for all $n\in\N$. For $n\in\N$, let $d_{n}=\text{dist}((t_{n},x_{n}),\R\times\{x=0\})=x_{n}$. Since $\phi(t_{n},x_{n})\in(r_{0},\frac{r_{0}}{2}]$ for all $n\in\N$ and $\lim_{x\ra\infty}\phi(t,x)=0$ uniformly in $t\in\R$, there's $\tilde{d}_{0}>0$ such that $d_{n}=x_{n}\leq\tilde{d}_{0}$ for all $n\in\N$. Moreover, the uniform estimate $\sup_{x\geq0,t\in\R}|\phi_{x}(t,x)|<\infty$ and the fact $\inf_{t\in\R}\phi(t,0)\geq0$ ensure the existence of some $\hat{d}_{0}>0$ such that $d_{n}=x_{n}\geq\hat{d}_{0}$ for all $n\in\N$. Hence, $d_{n}\in[\hat{d}_{0},\tilde{d}_{0}]$ for all $n\in\N$. Clearly,
\begin{equation*}
\phi_{t}\geq\phi_{xx}+\dot{\hat{\xi}}\phi_{x}\quad\text{for}\quad(t,x)\in B_{n},
\end{equation*}
where
\begin{equation*}
B_{n}=\Big\{(t,x)\in\R^{2}\Big|(t-t_{n})^{2}+(x-x_{n})^{2}\leq d_{n}^{2}\,\,\text{and}\,\,t\leq t_{n}\Big\}\subset\R\times[0,\infty).
\end{equation*}
For $n\in\N$, set $\xi_{n}:=\dot{\hat{\xi}}(\cdot+t_{n})$ and $\phi_{n}:=\phi(\cdot+t_{n},\cdot+x_{n})$. Then, we can find some subsequence, still denoted by $n$, such that $d_{n}\ra d_{0}$ as $n\ra\infty$ for some $d_{0}\in[\hat{d}_{0},\tilde{d}_{0}]$, $\{\xi_{n}\}_{n\in\N}$ converges to some $\bar{\xi}$ uniformly on  $B_{0}:=\{(t,x)\in\R^{2}|t^{2}+x^{2}\leq d_{0}^{2}\,\,\text{and}\,\,t\leq0\}$ (where we used $|\ddot{\hat{\xi}}^{v}(t)|=|\ddot{\tilde{\xi}}^{v}(t)|\leq C^{v}_{2}$ for $t\in\R$ by Lemma \ref{lem-interface-propagation-improved-general}), and $\{\phi_{n}\}_{n\in\N}$ converges to some $\bar{\phi}$ uniformly on $B_{0}$. Moreover, $\bar{\phi}$ satisfies
\begin{equation*}
\bar{\phi}_{t}\geq\bar{\phi}_{xx}+\bar{\xi}\bar{\phi}_{x}\quad\text{for}\quad(t,x)\in B_{0}
\end{equation*}
and, on $B_{0}$, $\bar{\phi}$ attains its minimum $r_{0}$ at $(0,0)$. Maximum principle then implies that $\bar{\phi}\equiv r_{0}$ on $B_{0}$. However, since $(t_{n},0)\in\partial B_{n}$ for all $n\in\N$ and $\phi(t_{n},0)\geq0$, we obtain that $\phi\geq0$ at some point on $\partial B_{0}$. It's a contradiction. Hence, $\inf_{x\geq0,t\in\R}\phi(t,x)\geq0$ and the proof is complete.
\end{proof}


\section{Uniqueness of Generalized Traveling Waves}\label{sec-uniqueness}

In this section, we investigate the uniqueness of generalized traveling waves of \eqref{main-eqn}. Throughout this section, we assume $u^f(t,x)$ is the generalized traveling wave of \eqref{main-eqn} in Proposition \ref{prop-transition-wave} and write $u^f(t,x)$ and $\xi^f(t)$ as $u(t,x)$ and $\xi(t)$, respectively. We assume that $v(t,x)$ is an arbitrary generalized traveling wave of \eqref{main-eqn}. Hence, $v(t,x)$ satisfies all properties as in Section \ref{sec-property-gtw}.

Recall that $\xi_{\theta}^{v}:\R\ra\R$ and $\xi:\R\ra\R$ are such that $v(t,\xi_{\theta}^{v}(t))=\theta=u(t,\xi(t))$ for all $t\in\R$. By considering space translations of $v(t,x)$ and $u(t,x)$, we may assume, without loss of generality, that $\xi_{\theta}^{v}(0)=0=\xi(0)$, that is, $v(0,0)=\theta=u(0,0)$. We will assume this normalization as well as $\rm(H1)$-$\rm(H3)$ in the rest of this section, and therefore, Theorem \ref{thm-uniqueness-introduction} reads

\begin{thm}\label{thm-uniqueness-normalized}
There holds $v(t,x)=u(t,x)$ for all $x\in\R$ and $t\in\R$.
\end{thm}

Before proving the above theorem, we first prove two lemmas. The first one concerns the boundedness of interface locations between $v(t,x)$ and $u(t,x)$. For $\la\in(0,1)$, let $\xi_{\la}:\R\ra\R$ be such that $u(t,\xi_{\la}(t))=\la$ for all $t\in\R$. By the space monotonicity of $u(t,x)$, $\xi_{\la}$ is well-defined and unique. In particular, $\xi_{\theta}\equiv\xi$.

\begin{lem}\label{lem-interface-width-two-wave} For any $\la_{1},\la_{2}\in(0,1)$, there holds $\sup_{t\in\R}|\xi_{\la_{1}}^{v}(t)-\xi_{\la_{2}}(t)|<\infty$. In particular,
there holds $\sup_{t\in\R}|\xi_{\theta}^{v}(t)-\xi(t)|<\infty$.
\end{lem}
\begin{proof} We note that, by Proposition \ref{prop-transition-wave}$\rm(iii)$, Lemma \ref{lem-property-xi-v}$\rm(ii)$, it suffices to prove the ``in particular" part $\sup_{t\in\R}|\xi_{\theta}^{v}(t)-\xi(t)|<\infty$. To do so, we show $\rm(i)$ $\sup_{t\leq0}|\xi_{\theta}^{v}(t)-\xi(t)|<\infty$ and $\rm(ii)$ $\sup_{t\geq0}|\xi_{\theta}^{v}(t)-\xi(t)|<\infty$.

$\rm(i)$ Suppose on the contrary $\sup_{t\leq0}|\xi_{\theta}^{v}(t)-\xi(t)|=\infty$. Then there exists $t_{n}\ra-\infty$ as $n\ra\infty$ such that $|\xi_{\theta}^{v}(t_{n})-\xi(t_{n})|\ra\infty$ as $n\ra\infty$. Suppose first $\xi_{\theta}^{v}(t_{n})-\xi(t_{n})\ra\infty$ as $n\ra\infty$. Let $\al\in(0,\frac{c_{0}}{2})$ be small and fixed, where $c_{0}$ is given in Proposition \ref{prop-transition-wave}$\rm(iv)$. Recall that $u(t,x)\leq e^{-c_{0}(x-\xi(t))}$ for $x\geq\xi(t)$. Then, for any $\ep>0$ small and any $\zeta_{0}>0$, we can find some large $N$ such that
\begin{equation*}
u(t_{N},x-\zeta_{0})-\ep\Ga_{\al}(x-\tilde{\xi}(t_{N})-\zeta_{0})\leq v(t_{N},x),\quad x\in\R,
\end{equation*}
where $\Ga_{\al}$ is given in \eqref{function-Ga} and $\tilde{\xi}:\R\ra\R$ is given in Proposition \ref{prop-interface-rightward-propagation-improved}. Applying Theorem \ref{thm-aprior-estimate}, we find
\begin{equation*}
u(t,x-\zeta_{1})-\ep\leq u(t,x-\zeta_{1})-\ep e^{-\om(t-t_{N})}\Ga_{\al}(x-\tilde{\xi}(t)-\zeta_{1})\leq v(t,x),\quad x\in\R,\,\,t\geq t_{N},
\end{equation*}
where $\zeta_{1}=\zeta_{0}-\frac{M\ep}{\om}$ is close to $\zeta_{0}$ if $\ep$ is sufficiently small. Setting $t=0$ and $x=\zeta_{1}+\xi(0)$ in the above estimate, we find $\theta-\ep\leq v(0,\zeta_{1}+\xi(0))$, which leads to $\xi_{\theta-\ep}^{v}(0)\geq\zeta_{1}+\xi(0)$. Since $\xi_{\theta-\ep}^{v}(0)-\xi_{\theta}^{v}(0)\leq L_{\theta-\ep,\theta}$, we have $\xi_{\theta}^{v}(0)-\xi(0)\geq\zeta_{1}-L_{\theta-\ep,\theta}>0$ if we choose $\zeta_{0}>0$ sufficiently large once $\ep>0$ is fixed. It's a contradiction to the normalization $\xi_{\theta}^{v}(0)=0=\xi(0)$.

Now assume $\xi_{\theta}^{v}(t_{n})-\xi(t_{n})\ra-\infty$ as $n\ra\infty$.
By Theorem \ref{thm-monotoncity-exponential-decay}$\rm(ii)$ and its proof, we have $\hat{\xi}^{v}(t)=\tilde{\xi}^{v}(t)+L_{\theta}$ and $v(t,x)\leq\theta e^{-\frac{c_{B}}{2}(x-\hat{\xi}^{v}(t))}$ for $x\geq\hat{\xi}^{v}(t)$. Also, recall $\sup_{t\in\R}|\xi_{\theta}^{v}(t)-\tilde{\xi}^{v}(t)|<\infty$. Let $\al\in(0,\frac{c_{B}}{2})$ be small and fixed. Then, for any $\ep>0$ small and any $\zeta_{0}<0$, we can find some large $N$ such that
\begin{equation*}
v(t_{N},x)\leq u(t_{N},x-\zeta_{0})+\ep\Ga_{\al}(x-\tilde{\xi}(t_{N})-\zeta_{0}),\quad x\in\R.
\end{equation*}
Applying Theorem \ref{thm-aprior-estimate}, we find
\begin{equation*}
v(t,x)\leq u(t,x-\zeta_{1})+\ep e^{-\om(t-t_{N})}\Ga_{\al}(x-\tilde{\xi}(t)-\zeta_{1})\leq u(t,x-\zeta_{1})+\ep,\quad x\in\R,\,\,t\geq t_{N},
\end{equation*}
where $\zeta_{1}=\zeta_{0}+\frac{M\ep}{\om}$. Setting $t=0$ and $x=\zeta_{1}+\xi(0)$, we obtain $v(0,\zeta_{1}+\xi(0))\leq\theta+\ep$, which leads to $\xi_{\theta+\ep}^{v}(0)\leq\zeta_{1}+\xi(0)$. Since $\xi_{\theta}^{v}(0)-\xi_{\theta+\ep}^{v}(0)\leq L_{\theta,\theta+\ep}$, we arrive at $\xi_{\theta}^{v}(0)-\xi(0)\leq\zeta_{1}+L_{\theta,\theta+\ep}<0$ if we choose $-\zeta_{0}$ is sufficiently large once $\ep>0$ is fixed. It's a contradiction. Hence, $\sup_{t\leq0}|\xi_{\theta}^{v}(t)-\xi(t)|<\infty$.

$\rm(ii)$ Fix a small $\al\in(0,\min\{\frac{c_{B}}{4},\frac{c_{0}}{2}\})$. On one hand, due to the normalization $\xi_{\theta}^{v}(0)=\xi(0)$ and the fact that $v(0,x)\leq\theta e^{-\frac{c_{B}}{2}(x-\hat{\xi}^{v}(0))}$ for $x\geq\hat{\xi}^{v}(0)$, for any small $\ep>0$, there holds
\begin{equation*}
v(0,x)\leq u(0,x-\zeta_{0})+\ep\Ga_{\al}(x-\tilde{\xi}(0)-\zeta_{0}),\quad x\in\R
\end{equation*}
for large $\zeta_{0}>0$.  Theorem \ref{thm-aprior-estimate} then yields
\begin{equation*}
v(t,x)\leq u(t,x-\zeta_{1})+\ep e^{-\om t}\Ga_{\al}(x-\tilde{\xi}(t)-\zeta_{1})\leq u(t,x-\zeta_{1})+\ep,\quad x\in\R,\,\,t\geq0,
\end{equation*}
where $\zeta_{1}=\zeta_{0}+\frac{M\ep}{\om}$. Setting $x=\zeta_{1}+\xi(t)$, we find $v(t,\zeta_{1}+\xi(t))\leq\theta+\ep$, which implies $\xi_{\theta+\ep}^{v}(t)\leq\zeta_{1}+\xi(t)$ for all $t\geq0$. It then follows from $\xi_{\theta}^{v}(t)-\xi_{\theta+\ep}^{v}(t)\leq L_{\theta,\theta+\ep}$ for all $t\in\R$ that $\xi_{\theta}^{v}(t)-\xi(t)\leq\zeta_{1}+L_{\theta,\theta+\ep}$ for all $t\geq0$.

On the other hand, for any small $\ep>0$, there holds
\begin{equation*}
u(0,x+\zeta_{0})-\ep\Ga_{\al}(x-\tilde{\xi}(0)+\zeta_{0})\leq v(0,x),\quad x\in\R
\end{equation*}
for large $\zeta_{0}>0$. Again, by Theorem \ref{thm-aprior-estimate}, we have
\begin{equation*}
u(t,x+\zeta_{1})-\ep\leq u(t,x+\zeta_{1})-\ep e^{-\om t}\Ga_{\al}(x-\tilde{\xi}(t)+\zeta_{1})\leq v(t,x),\quad x\in\R,\,\,t\geq0,
\end{equation*}
where $\zeta_{1}=\zeta_{0}+\frac{M\ep}{\om}$. Setting $x=-\zeta_{1}+\xi(t)$, we have $\theta-\ep\leq v(t,-\zeta_{1}+\xi(t))$, which leads to $\xi_{\theta-\ep}^{v}(t)\geq-\zeta_{1}+\xi(t)$ for all $t\geq0$. Since $\xi_{\theta-\ep}^{v}(t)-\xi_{\theta}^{v}(t)\leq L_{\theta-\ep,\theta}$ for all $t\in\R$, we find $\xi_{\theta}^{v}(t)-\xi(t)\geq-\zeta_{1}-L_{\theta-\ep,\theta}$ for all $t\geq0$. This establishes $\sup_{t\geq0}|\xi_{\theta}^{v}(t)-\xi(t)|<\infty$, and hence, completes the proof.
\end{proof}

The second lemma giving a comparison of $u(t,x)$ and $v(t,x)$ is similar to Lemma \ref{lem-space-monotonicity}.

\begin{lem}\label{lem-auxiliary-lemma}
There exists $h_{0}\geq0$ such that for any $h\geq h_{0}$ there holds $u(t,x)\leq v(t,x-h)$ for $x\in\R$ and $t\in\R$.
\end{lem}
\begin{proof}
Note that due to the space monotonicity of $v(t,x)$ by Theorem \ref{thm-monotoncity-exponential-decay}$\rm(i)$, we only need to show for some $h_{0}\geq0$ there holds $u(t,x)\leq v(t,x-h_{0})$ for $x\in\R$ and $t\in\R$. The proof is similar to, and even simpler than that of Lemma \ref{lem-space-monotonicity}, since we now have space monotonicity. Let us sketch the proof.

For $t\in\R$, let
\begin{equation}\label{three-sets}
\Om_{t}^{l}=\Big(-\infty,\xi_{\frac{1+\theta_{*}}{2}}(t)\Big],\quad\Om_{t}^{m}=\Big[\xi_{\frac{1+\theta_{*}}{2}}(t),\xi_{\frac{\theta}{2}}(t)\Big]\quad\text{and}\quad\Om_{t}^{r}=\Big[\xi_{\frac{\theta}{2}}(t),\infty\Big).
\end{equation}
For $i=l,m,r$, set $\Om_{i}=\cup_{t\in\R}(\{t\}\times\Om_{t}^{i})$.

We first claim that there is some $h_{0}>0$ such that $u(t,x)\leq v(t,x-h_{0})$ for $(t,x)\in\Om_{m}$. In fact, due to Lemma \ref{lem-interface-width-two-wave}, we have
\begin{equation*}
h_{0}:=\sup_{t\in\R}\Big|\xi^{v}_{\frac{1+\theta_{*}}{2}}(t)-\xi_{\frac{\theta}{2}}(t)\Big|<\infty.
\end{equation*}
If $x\leq\xi_{\frac{\theta}{2}}(t)$, then $x-h_{0}\leq\xi^{v}_{\frac{1+\theta_{*}}{2}}(t)$, and hence $v(t,x-h_{0})\geq\frac{1+\theta_{*}}{2}$ by monotonicity. Since $u(t,x)\leq\frac{1+\theta_{*}}{2}$ for $(t,x)\in\Om_{m}$, the claim follows.

For $\Om_{l}$, we note that if $(t,x)\in\Om_{l}$, then we have both $u(t,x)\geq\frac{1+\theta_{*}}{2}$ and $v(t,x-h_{0})\geq\frac{1+\theta_{*}}{2}$, and hence $\phi(t,x)=v(t,x-h_{0})-u(t,x)$ satisfies $\phi_{t}=\phi_{xx}+a(t,x)\phi$ with $a(t,x)\leq-\beta$ by $\rm(H3)$. We then proceed as in the proof of Lemma \ref{lem-space-monotonicity} for $\Om_{2}$ to conclude that
\begin{equation*}
\inf_{(t,x)\in\Om_{l}}(v(t,x-h_{0})-u(t,x))\geq0.
\end{equation*}

For $\Om_{r}$, we have $u(t,x)\leq\frac{\theta}{2}$ hence, $f(t,u(t,x))=0$ for $(t,x)\in\Om_{r}$. Thus, $\phi(t,x)=v(t,x-h_{0})-u(t,x)$ satisfies $\phi_{t}-\phi_{xx}=f(t,v(t,x-h_{0}))\geq0$. We then proceed as in the proof of Lemma \ref{lem-space-monotonicity} for $\Om_{3}$ to conclude that
\begin{equation*}
\inf_{(t,x)\in\Om_{r}}(v(t,x-h_{0})-u(t,x))\geq0.
\end{equation*}

In conclusion, $u(t,x)\leq v(t,x-h_{0})$ for $x\in\R$ and $t\in\R$, and the lemma follows.
\end{proof}

Now, we prove Theorem \ref{thm-uniqueness-normalized}.

\begin{proof}[Proof of Theorem \ref{thm-uniqueness-normalized}]
We modify the proof of Theorem \ref{thm-monotoncity-exponential-decay}$\rm(i)$. By Lemma \ref{lem-auxiliary-lemma}, there holds
\begin{equation*}
h_{*}:=\inf\big\{h_{0}\geq0\big|u(t,x)\leq v(t,x-h_{0}),\,\,x\in\R,\,\,t\in\R\big\}<\infty.
\end{equation*}
We claim that $h_{*}=0$. For contradiction, suppose $h_{*}>0$. Since $u(t,x)\leq v(t,x-h_{*})$ for all $x\in\R$ and $t\in\R$, and $u(0,0)=v(0,0)<v(0,-h_{*})$ by the normalization and monotonicity, we have $u(t,x)<v(t,x-h_{*})$ for all $x\in\R$ and $t\in\R$.

For $i=l,m,r$, let $\Om_{t}^{i}$ be as in \eqref{three-sets} and set $\Om_{i}^{-}=\cup_{t\leq0}(\{t\}\times\Om_{t}^{i})$. Setting
\begin{equation*}
\phi(t,x)=v(t,x-h_{*})-u(t,x),\quad x\in\R,\,\,t\in\R,
\end{equation*}
we show
\begin{equation}\label{estimate-infimum-middle-region}
\inf_{(t,x)\in\Om_{m}^{-}}\phi(t,x)>0.
\end{equation}
In fact, if \eqref{estimate-infimum-middle-region} fails, then there's $\{(t_{n},x_{n})\}_{n\in\N}\subset\Om_{m}^{-}$ with $t_{n}\ra-\infty$ as $n\ra\infty$ such that $\phi(t_{n},x_{n})\ra0$ as $n\ra\infty$. Since $\phi\geq0$, we conclude from Harnack inequality that $\phi_{n}(t,x):=\phi(t+t_{n},x+x_{n})$ converges locally uniformly to $0$ as $n\ra\infty$. This, in particular, implies that for any $\ep>0$, there exists $M=M(\ep)>0$ and $N=N(\ep)>0$ such that
\begin{equation}\label{an-estimate-auxiliary-222}
\sup_{|x|\geq M}\phi_{n}(0,x)\leq\frac{\ep}{100}\,\,\text{for all}\,\, n\in\N\quad\text{and}\quad\sup_{|x|\leq M}\phi_{n}(0,x)\leq\frac{\ep}{100}\,\,\text{for all}\,\, n\geq N.
\end{equation}
The first one holds for all $n\in\N$ is due to the uniform-in-time limits at $\pm\infty$ of generalized traveling waves and the fact $(t_{n},x_{n})\in\Om_{m}^{-}$ for all $n\in\N$. The second one is due to the locally uniform limit as above.

Moreover, since $(t_{n},x_{n})\in\Om_{m}^{-}$ for all $n\in\N$, there holds $\sup_{n\in\N}|\xi(t_{n})-x_{n}|<\infty$. Also, we recall $v(t_{n},x)$ is exponential decay ahead of $\hat{\xi}^{v}(t_{n})$ by Theorem \ref{thm-monotoncity-exponential-decay}$\rm(ii)$ and $\sup_{n\in\N}|\hat{\xi}^{v}(t_{n})-\xi(t_{n})|<\infty$ by Lemma \ref{lem-property-xi-v}$\rm(ii)$, Lemma \ref{lem-interface-propagation-improved-general} and Lemma \ref{lem-interface-width-two-wave}. Now, using \eqref{an-estimate-auxiliary-222}, we have
\begin{equation*}
\phi_{n}(0,x-x_{n})=v(t_{n},x-h_{*})-u(t_{n},x)\leq\frac{\ep}{100}\quad\text{for}\,\, x\leq x_{n}
\end{equation*}
provided $n\geq N$. For $x\geq x_{n}$, we have from \eqref{an-estimate-auxiliary-222}
\begin{equation*}
\phi_{n}(0,x-x_{n})=v(t_{n},x-h_{*})-u(t_{n},x)\leq\min\Big\{\frac{\ep}{100},v(t_{n},x-h_{*})\Big\}
\end{equation*}
provided $n\geq N$. We then conclude from the exponential decay of $v(t_{n},x-h_{*})$ and the uniform bounds:
\begin{equation*}
\sup_{n\in\N}|\xi(t_{n})-x_{n}|<\infty,\quad\sup_{n\in\N}|\hat{\xi}^{v}(t_{n})-\xi(t_{n})|<\infty\quad\text{and}\quad\sup_{n\in\N}|\xi(t_{n})-\tilde{\xi}(t_{n})|<\infty,
\end{equation*}
that for any $\ep>0$ there holds
\begin{equation*}
\phi_{n}(0,x)\leq\ep\Ga_{\al}(x-\tilde{\xi}(t_{n})),\quad x\in\R
\end{equation*}
if $n$ is sufficiently large, where $\al\in(0,\frac{c_{B}}{4})$ is small. The above estimate is the same as
\begin{equation*}
v(t_{n},x-h_{*})\leq u(t_{n},x)+\ep\Ga_{\al}(x-\tilde{\xi}(t_{n})),\quad x\in\R.
\end{equation*}
Theorem \ref{thm-aprior-estimate} then implies that
\begin{equation*}
v(t,x-h_{*})\leq u(t,x-\zeta_{1})+\ep e^{-\om(t-t_{n})}\Ga_{\al}(x-\tilde{\xi}(t)-\zeta_{1})\leq u(t,x-\zeta_{1})+\ep,\quad x\in\R,\,\,t\geq t_{n},
\end{equation*}
where $\zeta_{1}=\frac{M\ep}{\om}$. Setting $x=\xi(t)+\zeta_{1}$, we find
$v(t,\xi(t)+\zeta_{1}-h_{0})\leq\theta+\ep$, which yields $\xi_{\theta+\ep}^{v}(t)\leq\xi(t)+\zeta_{1}-h_{*}$, and then, $\xi^{v}_{\theta}(t)\leq\xi(t)+\zeta_{1}-h_{*}+\xi^{v}_{\theta}(t)-\xi_{\theta+\ep}^{v}(t)$. Now, setting $t=0$ and choosing $\ep>0$ small so that $\zeta_{1}+\xi^{v}_{\theta}(0)-\xi_{\theta+\ep}^{v}(0)<\frac{h_{*}}{2}$, we deduce from the normalization that $0\leq-\frac{h_{*}}{2}$. It's a contradiction. Hence, \eqref{estimate-infimum-middle-region} holds.

Minicing the arguments for \eqref{shifted-infimum}, we conclude from \eqref{estimate-infimum-middle-region}, $\sup_{x\in\R,t\in\R}|v_{x}(t,x)|<\infty$ and $\sup_{x\in\R,t\in\R}|u_{x}(t,x)|<\infty$ that there is small $\de_{*}\in(0,h_{*})$ such that
\begin{equation*}
\inf_{(t,x)\in\Om_{m}^{-}}(v(t,x-(h_{*}-\de_{*}))-u(t,x))\geq0.
\end{equation*}
We now argue as in the proof of Lemma \ref{lem-space-monotonicity} for $\Om_{2}$ and $\Om_{3}$ (here, we need to consider the restriction $t\leq0$) to find
\begin{equation*}
\inf_{(t,x)\in\Om_{i}^{-}}(v(t,x-(h_{*}-\de_{*}))-u(t,x))\geq0,\quad i=l,r.
\end{equation*}
Thus, $\inf_{x\in\R,t\leq0}(v(t,x-(h_{*}-\de_{*}))-u(t,x))\geq0$. Maximum principle then implies that $\inf_{(t,x)\in\R^{2}}(v(t,x-(h_{*}-\de_{*}))-u(t,x))\geq0$, which contradicts the minimality of $h_{*}$. Hence, $h_{*}=0$. It then follows that $v(t,x)\geq u(t,x)$ for all $x\in\R$ and $t\in\R$. Since $v(0,0)=u(0,0)$ by normalization, maximum principle ensures $v\equiv u$. This completes the proof.
\end{proof}

\section{Recurrence of Generalized Traveling Waves}\label{sec-recurrence}

In this section, we study the recurrence of the wave profile $\psi^f(t,x)=u^{f}(t,x+\xi^{f}(t))$ and the wave speed $\dot{\xi}^{f}(t)$ of the unique generalized traveling wave $u^f(t,x)$ of \eqref{main-eqn}, and the almost periodicity of $\psi^f(t,x)$ and $\dot{\xi}^{f}(t)$ in $t$ when $f(t,u)$ is almost periodic in $t$.

We are going to prove Theorem \ref{thm-recurrence}. Before this, let us first recall the definition of almost periodic functions and some basic properties.

\begin{defn}\label{almost-periodic-def}
\begin{itemize}
\item[\rm(i)] A continuous function $g:\R\to \R$ is called {\rm almost periodic} if for any sequence $\{\alpha_n^{'}\}_{n\in\N}\subset\R$, there
is a subsequence $\{\alpha_n\}_{n\in\N}\subset\{\alpha_n^{'}\}_{n\in\N}$ such that $\lim_{n\to\infty}g(t+\alpha_n)$ exists uniformly in $t\in\R$.

\item[\rm(ii)] Let $g(t,u)$ be a continuous function of $(t,u)\in\R\times\R$. $g$ is said to be {\rm almost periodic in $t$ uniformly with respect to $u$ in bounded sets} if
$g$ is uniformly continuous in $t\in\R$  and  $u$ in bounded sets,  and for each $u\in \R$, $g(t,u)$ is almost periodic in $t$.

\item[\rm (iii)]  Let $g\in C(\R\times \R,\R)$ be  almost periodic
 in $t$ uniformly with respect to $u$, and
$$
g(t,u)\sim \sum_{\lambda \in\R}a_\lambda (u)e^{i\lambda t}
$$ be the Fourier series of $g$ (see \cite{Vee}, \cite{Yi} for
the definition). Then ${\mathcal {S}}(g)=\{\lambda|a_\lambda(u)\not
\equiv 0\}$ is called the {\rm Fourier spectrum} of $g$, and
\begin{equation*}
{\mathcal {M}}(g)=\text{the smallest additive subgroup of }\,\,\R
\,\,\text{containing}\,\,{\mathcal {S}}(g)
\end{equation*}
is called the {\rm frequency module}
of $g$.
\end{itemize}
\end{defn}

\begin{rem}
\label{almost-periodic-rk}
\begin{itemize}
\item[\rm(i)] Suppose that $g:\R\to \R$ is continuous and almost periodic. Then $\lim_{T\to\infty}\frac{1}{T}\int_0^T g(t)dt$ exists.

\item[\rm(ii)] Let $g(t,u)$ be a continuous function of $(t,u)\in\R\times\R$. $g$ is almost periodic in $t$ uniformly with respect to
 $u$ in bounded sets  if and only if
 $g$ is uniformly continuous in $t\in\R$  and $u$ in bounded sets, and for any sequences $\{\alpha_n^{'}\}_{n\in\N}$, $\{\beta_n^{'}\}_{n\in\N}\subset \R$, there are subsequences $\{\alpha_n\}_{n\in\N}\subset\{\alpha_n^{'}\}_{n\in\N}$, $\{\beta_n\}_{n\in\N}\subset\{\beta_n^{'}\}_{n\in\N}$
such that
$$
\lim_{n\to\infty}\lim_{m\to\infty}g(t+\alpha_n+\beta_m,u)=\lim_{n\to\infty}g(t+\alpha_n+\beta_n,u)
$$
for each $(t,u)\in\R\times\R$ (see \cite[Theorems 1.17 and 2.10]{Fink}).

\item[\rm (iii)]  Assume that $g_1,g_2\in C(\R\times
\R,\R)$ are two uniformly almost periodic functions in its
first independent variable. If for any sequence $\{t_n\}_{n\in\N}\subset \R$
with $g_1\cdot t_n\to g_1$ as $n\to\infty$, there holds $g_2\cdot t_n\to g_2$ as $n\to\infty$, then $\mathcal{M}(g_2)\subset \mathcal{M}(g_1)$ (it follows from the results in \cite{Yi} (see also \cite{Fink}).

\end{itemize}
\end{rem}

In the rest of this section, we assume that $\rm(H1)$-$\rm(H4)$ hold. Then, any $g\in H(f)$ satisfies $\rm(H1)$-$\rm(H3)$. Let $u^g(t,x)$ be the unique generalized traveling wave of \eqref{main-eqn-1} with the continuously differentiable interface location function $\xi^{g}(t)$ satisfying $u^g(t,\xi^g(t))=\theta$ and the normalization $\xi^g(0)=0$. Let $\psi^{g}(t,x)=u^{g}(t,x+\xi^{g}(t))$ be the profile function. For any $u_0\in C_{\rm unif}^b(\R,\R)$, let $u(t,x;u_0,g)$ be the solution of \eqref{main-eqn-1} with $u(0,x;u_0,g)=u_0(x)$.

Next, we prove a lemma.

\begin{lem}
\label{recurrence-lem}
\begin{itemize}
\item[\rm(i)]   For any $g\in H(f)$,
 \begin{equation}
 \label{recurrence-eq1}
 \psi^g(\tau,x)=\psi^{g\cdot \tau}(0,x),\quad \forall\,\, \tau,x\in\R.
 \end{equation}

\item[\rm(ii)] $\lim_{x\to -\infty}\psi^g(t,x)=1$ and $\lim_{x\to\infty}\psi^g(t,x)=0$ uniformly in $t\in\R$ and $g\in H(f)$.

\item[\rm(iii)] $\sup_{g\in H(f),t\in\R}|\dot\xi^g(t)|<\infty$.
\end{itemize}
\end{lem}

\begin{proof}
$\rm(i)$ Observe that for any given $\tau\in\R$, $u_1(t,x)=\psi^{g\cdot \tau}(t,x-\xi^{g\cdot\tau}(t))$ is a generalized traveling wave of
 \begin{equation}
 \label{recurrence-eq2}
 u_t=u_{xx}+g(t+\tau,u)
\end{equation}
with $u_1(t,\xi^{g\cdot \tau}(t))=\theta$. Observe also that $u_2(t,x)=\psi^g(t+\tau,x-\xi^g(t+\tau))$ is a generalized traveling wave of \eqref{recurrence-eq2} with $u_2(t,\xi^g(t+\tau))=\theta$. Then by Theorem \ref{thm-uniqueness-introduction},
\begin{equation}\label{an-equality-recurrence}
\psi^{g\cdot \tau}(t,x)=\psi^g (t+\tau,x).
\end{equation}
Setting $t=0$, we get \eqref{recurrence-eq1}.

$\rm(ii)$ By \eqref{an-equality-recurrence}, for any $\tau\in\R$,
\begin{equation*}
\psi^{f\cdot \tau}(t,x)=\psi^{f}(t+\tau,x), \quad x\in\R,\,\,t\in\R.
\end{equation*}
As a consequence, we have
\begin{equation}
\label{aux-recurrence-eq1}
\lim_{x\to -\infty}\psi^{f\cdot\tau}(t,x)=1,\quad\lim_{x\to\infty}\psi^{f\cdot\tau}(t,x)=0\quad\text{uniformly in}\,\,t\in\R\,\,\text{and}\,\,\tau\in\R.
\end{equation}
Moreover,
\begin{equation}\label{aux-recurrence-eq1-1}
u^{f\cdot \tau}(t,x+\xi^{f\cdot\tau}(t))=u^{f}(t+\tau,x+\xi^{f}(t+\tau)), \quad x\in\R,\,\,t\in\R.
\end{equation}
In particular, for any $\tau\in\R$.
\begin{equation}\label{aux-recurrence-eq1-2}
u_{x}^{f\cdot\tau}(t,\xi^{f\cdot\tau}(t))=u_{x}^{f}(t+\tau,\xi^{f}(t+\tau)),\quad t\in\R.
\end{equation}
Setting $x=0$ in \eqref{aux-recurrence-eq1-1} and then differentiating the resulting equality with respect to $t$, we obtain for any $t\in\R$
\begin{equation*}
\begin{split}
\dot{\xi}^{f\cdot\tau}(t)&=\frac{\frac{d}{dt}[u^{f}(t+\tau,x+\xi^{f}(t+\tau))]-u_{t}^{f\cdot\tau}(t,\xi^{f\cdot\tau}(t))}{u_{x}^{f\cdot\tau}(t,\xi^{f\cdot\tau}(t))}\\
&=\frac{\frac{d}{dt}[u^{f}(t+\tau,x+\xi^{f}(t+\tau))]-u_{t}^{f\cdot \tau}(t,\xi^{f\cdot \tau}(t))}{u_{x}^{f}(t+\tau,\xi^{f}(t+\tau))},
\end{split}
\end{equation*}
where we used \eqref{aux-recurrence-eq1-2} in the second equality. By a priori estimates for parabolic equations and Proposition \ref{prop-transition-wave}$\rm(ii)$, $\frac{d}{dt}[u^{f}(t+\tau,x+\xi^{f}(t+\tau))]$ and  $u_{t}^{f\cdot\tau}(t,\xi^{f\cdot\tau}(t))$ is bounded uniformly in $t\in\R$ and $\tau\in\R$, and $u_{x}^{f}(t+\tau,\xi^{f}(t+\tau))$ is negative uniformly in $t\in\R$ and $\tau\in\R$. Hence, $\dot{\xi}^{f\cdot\tau}(t)$ is bounded uniformly in $t\in\R$ and $\tau\in\R$, i.e.,
\begin{equation}\label{aux-recurrence-eq2}
\sup_{t\in\R,\tau\in\R}|\dot{\xi}^{f\cdot\tau}(t)|<\infty.
\end{equation}

For any $g\in H(f)$,  there is $\{t_n\}\subset \R$ such that $g_n:=f\cdot t_{n}\to g$ in $H(f)$.
By \eqref{aux-recurrence-eq1} and a priori estimates for parabolic equations,  there exists a continuous function $\psi(\cdot,\cdot;g):\R\times\R\ra[0,1]$ such that, up to a subsequence,
\begin{equation*}
\psi^{g_ {n}}(t,x)\ra \psi(t,x;g)\quad\text{as}\quad n\ra\infty\,\,\text{locally uniformly in}\,\,(t,x)\in\R\times\R
\end{equation*}
and
\begin{equation}\label{uniform-form-limit-recurrence}
\lim_{x\to -\infty}\psi(t, x;g)=1,\,\,\lim_{x\to\infty}\psi(t,x;g)=0\,\,\text{uniformly in}\,\, t\in\R,\,\, g\in H(f).
\end{equation}

We claim that $\psi^g(t,x)=\psi(t,x;g)$.
In fact, as a special case of \eqref{aux-recurrence-eq2},
\begin{equation}\label{uniform-bounded-interface-recurrence}
\sup_{t\in\R,n\in\N}|\dot{\xi}^{g_{n}}(t)|<\infty.
\end{equation}
As a result, there exists a continuous function $\xi(\cdot;g):\R\ra\R$ such that, up to a subsequence,
\begin{equation}\label{locally-uniform-convergence-recurrence}
\xi^{g_{n}}(t)\ra \xi(t;g)\quad\text{as}\quad n\ra\infty\,\,\text{locally uniformly in}\,\,t\in\R.
\end{equation}
Hence
\begin{equation*}
\psi^{g_n}(t,x-\xi^{g_n}(t))\to\psi(t,x-\xi(t;g);g)
\end{equation*}
as $n\to\infty$ locally uniformly in $(t,x)\in\R\times\R$.
Observe that $u(t,x;\psi^{g_n}(0,\cdot),g_n)(=\psi^{g_n}(t,x-\xi^{g_n}(t)))$ is an entire solution of \eqref{main-eqn-1} with $g$ being replaced by $g_n$ and
\begin{equation*}
u(t,x;\psi^{g_n}(0,\cdot),g_n)\to u(t,x;\psi(0,\cdot;g),g)
\end{equation*}
as $n\to\infty$ locally uniformly in $(t,x)\in\R\times\R$. It then follows that
\begin{equation*}
u(t,x;\psi(0,\cdot;g),g)=\psi(t,x-\xi(t;g);g),\quad x\in\R,\,\,t\in\R.
\end{equation*}
Thus, $\psi(t,x-\xi(t;g);g)$ is an entire solution of \eqref{main-eqn-1}. Set $u(t,x;g):=\psi(t,x-\xi(t;g);g)$. Due to \eqref{uniform-form-limit-recurrence}, for $u(t,x;g)$ being a generalized traveling wave, it remains to show that $\xi(t;g)$ is differentiable and $\sup_{t\in\R}|\dot{\xi}(t;g)|<\infty$. To do so, we first see that $u(t,x;g)$ is strictly decreasing in $x$ by the maximum principle and the fact that $\psi(t,x;g)$ is nonincreasing in $x$. Moreover, since $u(t,\xi(t;g);g)=\psi(t,0;g)=\lim_{n\ra\infty}\psi^{g_{n}}(t,0)=\theta$ for any $t\in\R$, $\xi(t;g)$ is continuously differentiable.
 Then, there must hold
\begin{equation}\label{aux-bound-of-interface-location}
\sup_{t\in\R}|\dot{\xi}(t;g)|\leq\sup_{t\in\R,n\in\N}|\dot{\xi}^{g_{n}}(t)|<\infty,
\end{equation}
otherwise we can easily deduce a contradiction from \eqref{uniform-bounded-interface-recurrence} and \eqref{locally-uniform-convergence-recurrence}. Consequently, $u(t,x;g)=\psi(t,x-\xi(t;g);g)$ is a generalized traveling wave of \eqref{main-eqn-1}. By Theorem \ref{thm-uniqueness-introduction}, we have
 $$
 \psi^g(t,x)=\psi(t,x;g).
 $$
This proves the claim, and then, $\rm(ii)$ follows from \eqref{uniform-form-limit-recurrence}.

$\rm(iii)$ It follows from \eqref{aux-recurrence-eq2} and \eqref{aux-bound-of-interface-location}.
\end{proof}

Finally, we  prove Theorem \ref{thm-recurrence}.

 \begin{proof}[Proof of Theorem \ref{thm-recurrence}]
 First of all, \eqref{recurrence-thm-eq1} follows from Lemma \ref{recurrence-lem}(i).

Second, we prove \eqref{recurrence-thm-eq2}, that is,  for any $g\in H(f)$,
\begin{equation*}
\dot \xi^g(t)=-\frac{\psi_{xx}^g(t,0)+g(t,\psi^g(t,0))}{\psi_x^g(t,0)}.
\end{equation*}
Differentiating $u^g(t,\xi^g(t))=\psi^g(t,0)=\theta$, we get $u^g_t(t,\xi^g(t))+u_x^g(t,\xi^g(t))\dot \xi^g(t)=0$, that is,
\begin{equation*}
\dot \xi^g(t)=-\frac{u^g_t(t,\xi^g(t))}{u_x^g(t,\xi^g(t))},
\end{equation*}
which is meaningful by Theorem \ref{thm-monotoncity-exponential-decay}$\rm(i)$. Note that
$$
u^g_t(t,x)=u^g_{xx}(t,x)+g(t,u^g(t,x)),
$$
and
$$
u^g_x(t,x)=\psi^g_x(t,x-\xi^g(t)),\quad u^g_{xx}(t,x)=\psi^g_{xx}(t,x-\xi^g(t)).
$$
We then have \eqref{recurrence-thm-eq2}.

Next, we show \eqref{recurrence-thm-eq3}, that is,  the mapping
 \begin{equation}
 \label{recurrence-eq4}
 [H(f)\ni g\mapsto \psi^g(0,\cdot)\in C_{\rm unif}^b(\R,\R)]\quad \text{
 is continuous.}
 \end{equation}
  Suppose that $g^*\in H(f)$, $\{g_n\}_{n\in\N}\subset H(f)$ and $g_n\to g^*$ in $H(f)$ as $n\to \infty$, that is,
 $$
 g_n(t,x)\to g^*(t,x)\,\,\text{as}\,\,n\to\infty\,\,\text{locally uniformly in}\,\,(t,x)\in\R\times\R.
 $$
Thus, by Proposition \ref{prop-transition-wave} and Lemma \ref{recurrence-lem},  without loss of generality, we may assume that there are continuous functions $\xi^{*}:\R\ra\R$ and $\psi^{*}:\R\times\R\ra[0,1]$ such that
 $$
 \xi^{g_n}(t)\to\xi^*(t),\,\,\psi^{g_n}(t,x)\to \psi^*(t,x)
 $$
as $n\to\infty$ uniformly in $x\in\R$ and locally uniformly in $t\in\R$. Hence
 $$
 \psi^{g_n}(t,x-\xi^{g_n}(t))\to\psi^*(t,x-\xi^*(t))
 $$
as $n\to\infty$ uniformly in $x\in\R$ and locally uniformly in $t\in\R$. Observe that $u(t,x;\psi^{g_n}(0,\cdot),g_n)=\psi^{g_n}(t,x-\xi^{g_n}(t))$ is an entire solution of \eqref{main-eqn-1} with $g$ being replaced by $g_n$ and
 $$
 u(t,x;\psi^{g_n}(0,\cdot),g_n)\to u(t,x;\psi^*(0,\cdot),g^*)
$$
as $n\to\infty$ locally uniformly in $(t,x)\in\R\times\R$.
It then follows that
 $$
 u(t,x;\psi^*(0,\cdot),g^*)=\psi^*(t,x-\xi^*(t)).
 $$
This implies that $\psi^*(t,x-\xi^*(t))$ is an entire solution of \eqref{main-eqn-1}. By Lemma \ref{recurrence-lem}(ii),
$$
\lim_{x\to -\infty}\psi^*(t,x)=1,\quad\lim_{x\to\infty}\psi^*(t,x)=0\quad\text{uniformly in}\,\,t\in\R.
$$
Note that $\psi^*(t,x)$ is nonincreasing in $x$. This together with comparison principle for parabolic equations  implies that $\psi_x(t,x)<0$ for all $t\in\R$ and $x\in\R$ and then $\xi^*(t)$ is continuously differentiable. By Lemma \ref{recurrence-lem}(iii), we have
$$
\sup_{t\in\R}|\dot\xi^*(t)|<\infty.
$$
It then follows that  $\psi^*(t,x-\xi^*(t))$ is a
 generalized traveling wave of \eqref{main-eqn-1} with $g$ being replaced by $g^*$.
Then by Theorem \ref{thm-uniqueness-introduction}, $\psi^*(t,x)=\psi^{g^*}(t,x)$, and therefore,
 $$
 \psi^{g_n}(0,x)\to \psi^{g^*}(0,x)\,\,\text{as}\,\,n\to\infty\,\,\text{uniformly in}\,\,x\in\R.
 $$
That is, \eqref{recurrence-eq4} holds.

Now assume that $f(t,u)$ is almost periodic in $t$ uniformly with respect to $u$ in bounded sets. For any given sequences  $\{\alpha_{n}^{'}\}_{n\in\N}$, $\{\beta_n^{'}\}_{n\in\N}\subset\R$, there are subsequences $\{\alpha_n\}_{n\in\N}\subset\{\alpha_{n}^{'}\}_{n\in\N}$, $\{\beta_n\}_{n\in\N}\subset\{\beta_n^{'}\}_{n\in\N}$ such that
 $$
 \lim_{n\to\infty}\lim_{m\to\infty}f(t+\alpha_n+\beta_m,u)=\lim_{n\to\infty}f(t+\alpha_n+\beta_n,u)
 $$
for all $t\in\R$ and $u\in\R$. Let $g(t,u)=\lim_{m\to\infty}f(t+\beta_m,u)$ and $h(t,u)=\lim_{n\to\infty}g(t+\alpha_n,u)$. By \eqref{recurrence-eq1} and \eqref{recurrence-eq4},
 $$
 \lim_{m\to\infty}\psi^f(t+\beta_m,x)=\lim_{m\to\infty}\psi^{f\cdot(t+\beta_m)}(0,x)=\psi^{g\cdot t}(0,x),
 $$
 $$
 \lim_{n\to\infty}\psi^g(t+\alpha_n,x)=\lim_{n\to\infty}\psi^{g\cdot (t+\alpha_n)}(0,x)=\psi^{h\cdot t}(0,x)=\psi^h (t,x),
 $$
 and
 $$
 \lim_{n\to\infty} \psi^f(t+\alpha_n+\beta_n,x)=\lim_{n\to\infty}\psi^{f\cdot(t+\alpha_n+\beta_n)}(0,x)=\psi^{h\cdot t}(0,x)=\psi^h(t,x)
 $$
 for all $t\in\R$ and $x\in\R$.
 Therefore,
 $$
 \lim_{n\to\infty}\lim_{m\to\infty}\psi^{f}(t+\alpha_n+\beta_m,x)=\lim_{n\to\infty}\psi^{f}(t+\alpha_n+\beta_n,x),\quad \forall\,\, t\in\R,\,\, x\in\R.
 $$
 Obviously, $\psi^f(t,x)$ is uniformly continuous in $t\in\R$ and $x\in\R$.
 By Remark  \ref{almost-periodic-rk}(ii), $\psi^f(t,x)$ is almost periodic in $t$ uniformly with respect to $x$.
 Moreover, by Remark   \ref{almost-periodic-rk}(iii) and \eqref{recurrence-eq4},
 $\mathcal{M}(\psi^f(\cdot,\cdot))\subset\mathcal{M}(f(\cdot,\cdot))$.

 Finally,  note that
 $$
 \psi^f_x(t,x)=\lim_{h\to 0}\frac{\psi^f(t,x+h)-\psi^f(t,x)}{h}\quad\text{and}\quad\psi_{xx}^f(t,0)=\lim_{h\to 0}\frac{\psi^f_x(t,x+h)-\psi^f_x(t,x)}{h}
 $$
 uniformly in $t\in\R$ and $x\in\R$. Hence $\psi^f_x(t,x)$ and $\psi^f_{xx}(t,x)$ are almost periodic in $t$ uniformly with respect to $x$.
 Then
 by \eqref{recurrence-thm-eq2},  $\dot\xi^f(t)$ is almost periodic and the limit
 $$
\lim_{t\to\infty}\frac{\xi^f(t)-\xi^f(0)}{t}=\lim_{t\to\infty}\frac{1}{t}\int_0^t \dot\xi^f(\tau)d\tau
 $$
exists. Moreover, we also have $\mathcal{M}(\psi_x^f(\cdot,\cdot))$, $\mathcal{M}(\psi^f_{xx}(\cdot,\cdot))\subset\mathcal{M}(f(\cdot,\cdot))$ by Remark \ref{almost-periodic-rk}(iii) and \eqref{recurrence-eq4}. Then, by \eqref{recurrence-thm-eq2} again,
$\mathcal{M}(\dot\xi^f(\cdot))\subset\mathcal{M}(f(\cdot,\cdot))$.
 \end{proof}

\section*{Acknowledgements}

We would like to thank the referee for carefully reading the original manuscript and providing helpful suggestions resulting in some improvements of our results.

\bibliographystyle{amsplain}

\begin{thebibliography}{10}

\bibitem{AlBaCh99} N. Alikakos, P. W. Bates and X. Chen, Periodic traveling waves and locating oscillating patterns in multidimensional domains. \textit{Trans. Amer. Math. Soc.} 351 (1999), no. 7, 2777-2805.

\bibitem{Ang88} S. Angenent, The zero set of a solution of a parabolic equation. \textit{J. Reine Angew. Math.} 390 (1988), 79-96.

\bibitem{ArWe75} D. G. Aronson and  H. F. Weinberger, Nonlinear diffusion in population genetics, combustion, and nerve pulse propagation. \textit{Lecture Notes in Math.}, Vol. 446, Springer, Berlin, 1975.

\bibitem{ArWe78} D. G. Aronson and  H. F. Weinberger, Multidimensional nonlinear diffusion arising in population genetics. \textit{Adv. in Math.} 30 (1978), no. 1, 33-76.


\bibitem{BeHa02} H. Berestycki and F. Hamel, Front propagation in periodic excitable media. \textit{Comm. Pure Appl. Math.} 55 (2002), no. 8, 949-1032.

\bibitem{BeHa07} H. Berestycki and F. Hamel, Generalized travelling waves for reaction-diffusion equations. Perspectives in nonlinear partial differential equations, 101-123, \textit{Contemp. Math.}, 446, Amer. Math. Soc., Providence, RI, 2007.

\bibitem{BeHa12} H. Berestycki and F. Hamel, Generalized transition waves and their properties. \textit{Comm. Pure Appl. Math.} 65 (2012), no. 5, 592-648.

\bibitem{BeLaLi90} H. Berestycki, B. Larrouturou and P.-L. Lions, Multi-dimensional travelling-wave solutions of a flame propagation model. \textit{Arch. Rational Mech. Anal.} 111 (1990), no. 1, 33-49.

\bibitem{BeNi91} H. Berestycki and L. Nirenberg, On the method of moving planes and the sliding method. \textit{Bol. Soc. Brasil. Mat. (N.S.)} 22 (1991), no. 1, 1-37.

\bibitem{BeNiSc85} H. Berestycki, B. Nicolaenko and B. Scheurer, Traveling wave solutions to combustion models and their singular limits. \textit{SIAM J. Math. Anal.} 16 (1985), no. 6, 1207-1242.

\bibitem{BaCh99} P. Bates and A. Chmaj, A discrete convolution model for phase transitions. \textit{Arch. Ration. Mech. Anal.} 150 (1999), no. 4, 281-305.

\bibitem{Ch09} F. Chen, Stability and uniqueness of traveling waves for system of nonlocal evolution equations with bistable nonlinearity. \textit{Discrete Contin. Dyn. Syst.} 24 (2009), no. 3, 659-673.

\bibitem{Ch97} X. Chen, Existence, uniqueness, and asymptotic stability of traveling waves in nonlocal evolution equations. \textit{Adv. Differential Equations} 2 (1997), no. 1, 125-160.

\bibitem{ChGu02} X. Chen and J. Guo, Existence and asymptotic stability of traveling waves of discrete quasilinear monostable equations. \textit{J. Differential Equations} 184 (2002), no. 2, 549-569.

\bibitem{ChGuWu08} X. Chen, J. Guo and C.-C. Wu, Traveling waves in discrete periodic media for bistable dynamics. \textit{Arch. Ration. Mech. Anal.} 189 (2008), no. 2, 189-236.

\bibitem{DiHaZh14} W. Ding, F. Hamel and X.-Q. Zhao, Bistable pulsating fronts for reaction-diffusion equations in a periodic habitat. arXiv:1408.0723.


\bibitem{FiMc77} P. C. Fife and J. B. McLeod, The approach of solutions of nonlinear diffusion equations to travelling front solutions. \textit{Arch. Ration. Mech. Anal.} 65 (1977), no. 4, 335-361.

\bibitem{FiMc80} P. C. Fife and J. B. McLeod, A phase plane discussion of convergence to travelling fronts for nonlinear diffusion. \textit{Arch. Rational Mech. Anal.} 75 (1980/81), no. 4, 281-314.

\bibitem{Fink} A.M. Fink, Almost Periodic Differential Equations, \textit{Lectures Notes in Mathematics}, Vol. 377, Springer-Verlag, Berlin-New York, 1974.



\bibitem{Kam76} Y. Kametaka, On the nonlinear diffusion equation of Kolmogorov-Petrovskii-Piskunov type. \textit{Osaka J. Math.} 13 (1976), no. 1, 11-66.

\bibitem{Ka60} Ja. I. Kanel, The behavior of solutions of the Cauchy problem when the time tends to infinity, in the case of quasilinear equations arising in the theory of combustion. \textit{Dokl. Akad. Nauk SSSR} 132 268-271; translated as Soviet Math. Dokl. 1 1960 533-536.

\bibitem{Ka61} Ja. I. Kanel, Certain problems on equations in the theory of burning. \textit{Dokl. Akad. Nauk SSSR} 136 277-280 (Russian); translated as Soviet Math. Dokl. 2 1961 48-51.

\bibitem{Ka62} Ja. I. Kanel, Stabilization of solutions of the Cauchy problem for equations encountered in combustion theory. \textit{Mat. Sb. (N.S.)} 59 (101) 1962 suppl., 245-288.

\bibitem{Ka64} Ja. I. Kanel, Stabilization of the solutions of the equations of combustion theory with finite initial functions. \textit{Mat. Sb. (N.S.)} 65 (107) 1964 398-413.

\bibitem{LeKe00} T. Lewis and J. Keener, Wave-block in excitable media due to regions of depressed excitability. \textit{SIAM J. Appl. Math.} 61 (2000), no. 1, 293-316.

\bibitem{MeRoSi10} A. Mellet, J.-M. Roquejoffre and Y. Sire, Generalized fronts for one-dimensional reaction-diffusion equations. \textit{Discrete Contin. Dyn. Syst.} 26 (2010), no. 1, 303-312.

\bibitem{MNRR09} A. Mellet, J. Nolen, J.-M. Roquejoffre and L. Ryzhik, Stability of generalized transition fronts. \textit{Comm. Partial Differential Equations} 34 (2009), no. 4-6, 521-552.

\bibitem{MaWu07} S. Ma and J. Wu, Existence, uniqueness and asymptotic stability of traveling wavefronts in a non-local delayed diffusion equation. \textit{J. Dynam. Differential Equations} 19 (2007), no. 2, 391-436.

\bibitem{Na14} G. Nadin, Critical travelling waves for general heterogeneous one-dimensional reaction-diffusion equations, \textit{Ann. Inst. H. Poincar\'{e} Anal. Non Lin\'{e}aire}. DOI: 10.1016/j.anihpc.2014.03.007

\bibitem{NaRo12} G. Nadin and L. Rossi, Propagation phenomena for time heterogeneous KPP reaction-diffusion equations. \textit{J. Math. Pures Appl.} (9) 98 (2012), no. 6, 633-653.

\bibitem{NoRy09} J. Nolen and L. Ryzhik, Traveling waves in a one-dimensional heterogeneous medium. \textit{Ann. Inst. H. Poincar\'{e} Anal. Non Lin\'{e}aire} 26 (2009), no. 3, 1021-1047.

\bibitem{NRRZ12} J. Nolen, J.-M. Roquejoffre, L. Ryzhik and A. Zlato\v{s}, Existence and non-existence of Fisher-KPP transition fronts. \textit{Arch. Ration. Mech. Anal.} 203 (2012), no. 1, 217-246.

\bibitem{KPP37} A. Kolmogorov, I. Petrowsky, N. Piscunov, Study of the diffusion equation  with growth of the quantity of matter and its application to a biology problem, \textit{Bjul. Moskovskogo Gos. Univ.} 1 (1937) 1-26.

\bibitem{Ro94} J.-M. Roquejoffre, Convergence to travelling waves for solutions of a class of semilinear parabolic equations. \textit{J. Differential Equations} 108 (1994), no. 2, 262-295.

\bibitem{Sa76} D. H. Sattinger, On the stability of waves of nonlinear parabolic systems. \textit{Advances in Math.} 22 (1976), no. 3, 312-355.

\bibitem{Sh99-1} W. Shen, Travelling waves in time almost periodic structures governed by bistable nonlinearities. I. Stability and uniqueness. \textit{J. Differential Equations} 159 (1999), no. 1, 1-54.

\bibitem{Sh99-2} W. Shen, Travelling waves in time almost periodic structures governed by bistable nonlinearities. II. Existence. \textit{J. Differential Equations} 159 (1999), no. 1, 55-101.

\bibitem{Sh04} W. Shen, Traveling waves in diffusive random media. \textit{J. Dynam. Differential Equations} 16 (2004), no. 4, 1011-1060.

\bibitem{Sh06} W. Shen, Traveling waves in time dependent bistable equations. \textit{Differential Integral Equations} 19 (2006), no. 3, 241-278.

\bibitem{Sh11} W. Shen, Existence, uniqueness, and stability of generalized traveling waves in time dependent monostable equations. \textit{J. Dynam. Differential Equations} 23 (2011), no. 1, 1-44.

\bibitem{Sh11-1} W. Shen, Existence of generalized traveling waves in time recurrent and space periodic monostable equations. \textit{J. Appl. Anal. Comput.} 1 (2011), no. 1, 69-93.

\bibitem{ShSh14} W. Shen and Z. Shen, Transition fronts in time heterogeneous and random media of ignition type. arXiv:1407.7579.

\bibitem{SmZh00} H. Smith and X.-Q. Zhao, Global asymptotic stability of traveling waves in delayed reaction-diffusion equations. \textit{SIAM J. Math. Anal.} 31 (2000), no. 3, 514-534.

\bibitem{Uch78} K. Uchiyama, The behavior of solutions of some nonlinear diffusion equations for large time. \textit{J. Math. Kyoto Univ.} 18 (1978), no. 3, 453-508.

\bibitem{Vee} W. A. Veech, Almost  automorphic functions on groups,
{\it Amer. J. Math.}  {\bf 87} (1965), pp. 719-751.

\bibitem{We02} H. Weinberger, On spreading speeds and traveling waves for growth and migration models in a periodic habitat. \textit{J. Math. Biol.} 45 (2002), no. 6, 511-548.

\bibitem{Xin91} J. Xin, Existence and uniqueness of travelling waves in a reaction-diffusion equation with combustion nonlinearity. \textit{Indiana Univ. Math. J.} 40 (1991), no. 3, 985-1008.

\bibitem{Xin92} J. Xin, Existence of planar flame fronts in convective-diffusive periodic media. \textit{Arch. Rational Mech. Anal.} 121 (1992), no. 3, 205-233.

\bibitem{Xin93} J. Xin, Existence and nonexistence of traveling waves and reaction-diffusion front propagation in periodic media. \textit{J. Statist. Phys.} 73 (1993), no. 5-6, 893-926.


\bibitem{Yi} Y. Yi, Almost automorphic and almost periodic dynamics in dkew-product semiflows, Part I. Almost automorphy and almost periodicity,  \textit{Mem. Amer. Math. Soc.} 136 (1998), no. 647.

\bibitem{Zl12} A. Zlato\v{s}, Transition fronts in inhomogeneous Fisher-KPP reaction-diffusion equations. \textit{J. Math. Pures Appl.} (9) 98 (2012), no. 1, 89-102.

\bibitem{Zl13} A. Zlato\v{s}, Generalized traveling waves in disordered media: existence, uniqueness, and stability. \textit{Arch. Ration. Mech. Anal.} 208 (2013), no. 2, 447-480.
\end{thebibliography}

\end{document}